\documentclass[11pt,reqno]{amsart}
\usepackage[tmargin=1.5in,bmargin=1.4in,rmargin=1.4in,lmargin=1.4in]{geometry}

\usepackage{amsmath,amsthm,amssymb,mathrsfs,bbm,tocvsec2}
\usepackage[breaklinks=true]{hyperref}

\theoremstyle{plain}
\newtheorem{theorem}{Theorem}[section]
\newtheorem{corollary}[theorem]{Corollary}
\newtheorem{lemma}[theorem]{Lemma}
\newtheorem{proposition}[theorem]{Proposition}

\theoremstyle{definition}
\newtheorem{definition}[theorem]{Definition}
\newtheorem{example}[theorem]{Example}
\newtheorem{notation}[theorem]{Notation}
\newtheorem{remark}[theorem]{Remark}
\newtheorem{question}[theorem]{Question}

\numberwithin{equation}{section}
\numberwithin{table}{section}

\newcommand{\ba}{\mathbf{a}}
\newcommand{\bb}{\mathbf{b}}
\newcommand{\bc}{\mathbf{c}}
\newcommand{\bd}{\mathbf{d}}

\newcommand{\bj}{\mathbf{j}}
\newcommand{\bk}{\mathbf{k}}
\newcommand{\bmu}{{\boldsymbol{\mu}}}

\newcommand{\bs}{\mathbf{s}}

\newcommand{\bw}{\mathbf{w}}
\newcommand{\bx}{\mathbf{x}}
\newcommand{\by}{\mathbf{y}}
\newcommand{\bz}{\mathbf{z}}
\newcommand{\bJ}{\mathbf{1}}
\newcommand{\bZ}{\mathbf{0}}
\newcommand{\C}{\mathbb{C}}
\newcommand{\cB}{\mathcal{B}}

\newcommand{\cP}{\mathcal{P}}

\newcommand{\eps}{\varepsilon}

\newcommand{\Id}{\mathrm{Id}}
\newcommand{\diag}{\mathop{\mathrm{diag}}}
\newcommand{\inc}[2]{{#1}^{#2, \uparrow}}
\newcommand{\intd}{\,\rd}
\newcommand{\N}{\mathbb{N}}
\newcommand{\PF}{\mathrm{PF}}
\newcommand{\pd}{\mathrm{pd}}
\newcommand{\psd}{\mathrm{psd}}
\newcommand{\Q}{\mathbb{Q}}
\newcommand{\R}{\mathbb{R}}
\newcommand{\rd}{\mathrm{d}}
\newcommand{\sF}{\mathscr{F}}
\newcommand{\sFTN}{\sF^\TN}
\newcommand{\sFTP}{\sF^\TP}
\newcommand{\TN}{\mathrm{TN}}
\newcommand{\TP}{\mathrm{TP}}
\newcommand{\Z}{\mathbb{Z}}

\renewcommand{\Pr}{\mathbb{P}}

\begin{document}
\title[TP kernels, P\'olya frequency functions, and their transforms]%
{Totally positive kernels, P\'olya frequency functions,
and their transforms}

\author{Alexander Belton}
\address[A.~Belton]{Department of Mathematics and Statistics,
Lancaster University, Lancaster, UK}
\email{\tt a.belton@lancaster.ac.uk}

\author{Dominique Guillot}
\address[D.~Guillot]{University of Delaware, Newark, DE, USA}
\email{\tt dguillot@udel.edu}

\author{Apoorva Khare}
\address[A.~Khare]{Department of Mathematics, Indian Institute of
Science; and Analysis and Probability Research Group; Bangalore, India}
\email{\tt khare@iisc.ac.in}

\author{Mihai Putinar}
\address[M.~Putinar]{University of California at Santa Barbara, CA,
USA and Newcastle University, Newcastle upon Tyne, UK} 
\email{\tt mputinar@math.ucsb.edu, mihai.putinar@ncl.ac.uk}

\date{27th November 2021}



\begin{abstract}
The composition operators preserving total non-negativity and total
positivity for various classes of kernels are classified, following
three themes. Letting a function act by post composition on kernels
with arbitrary domains, it is shown that such a composition operator
maps the set of totally non-negative kernels to itself if and only
if the function is constant or linear, or just linear if it preserves
total positivity. Symmetric kernels are also discussed, with a similar
outcome. These classification results are a byproduct of two
matrix-completion results and the second theme: an extension of
A.M.~Whitney's density theorem from finite domains to subsets of the
real line. This extension is derived via a discrete convolution with
modulated Gaussian kernels. The third theme consists of analyzing,
with tools from harmonic analysis, the preservers of several families
of totally non-negative and totally positive kernels with additional
structure: continuous Hankel kernels on an interval, P\'olya frequency
functions, and P\'olya frequency sequences. The rigid structure of
post-composition transforms of totally positive kernels acting on
infinite sets is obtained by combining several specialized situations
settled in our present and earlier works.
\end{abstract}

\maketitle

\settocdepth{section}
\tableofcontents

\section{Introduction and main results}

\subsection{Total positivity}

Let $X$ and $Y$ be totally ordered sets. A kernel
$K : X \times Y \to \R$ is said to be \emph{totally positive} if the
matrix $( K( x_i, y_j ) )_{i, j = 1}^n$ is totally positive (that is,
all of its minors are positive) for any choice of $x_1 < \cdots < x_n$
and $y_1 < \cdots < y_n$, where $n$ is an arbitrary positive integer.
Similarly, the kernel $K$ is said to be \emph{totally non-negative} if
$( K(x_i, y_j) )_{i, j = 1}^n$ is totally non-negative (that is, all of
its minors are non-negative)\footnote{The monographs
\cite{Karlin, Pinkus} use the terms ``strict total positivity'' and
``total positivity'' instead of ``total positivity'' and ``total
non-negativity'' respectively.}. For almost a century, these classes
of kernels and matrices surfaced in the most unexpected circumstances,
and this trend continues in full force today. The foundational work
\cite{GK}, the survey \cite{Ando}, the early monograph \cite{GK1}, and
the more recent publications \cite{Karlin,TP,Pinkus,FJ} offer ample
references to the fascinating history of total positivity, as well as
accounts of its many surprising applications. Total positivity
continues to make impacts in areas such as representation theory
\cite{L-1,L-2,Ri}, network analysis \cite{Postnikov}, cluster algebras
\cite{BFZ,FZ-1,FZ-2}, Gabor analysis \cite{GRS}, statistics
\cite{Efron,KP}, and combinatorics \cite{Br1,Br2}. A surprising link
between positive Grassmannians, seen as the geometric impersonation of
total positivity, and integrable systems \cite{KW-1,KW-2} is also
currently developing at a fast pace.

A natural way to construct new kernels from existing ones is to
compose them with a given function $F$. More precisely, every suitable
function $F$ induces a composition operator $C_F$ mapping the kernel
$K$ to $C_F( K ) := F \circ K$. The aim of the present work is to
determine for which functions $F$ the operator $C_F$ leaves invariant
the set of totally positive kernels defined on $X \times Y$, and to
answer the analogous question for total non-negativity.  When $X$ and
$Y$ are finite, this is equivalent to determining when the
\emph{entrywise calculus} $( a_{i j} ) \mapsto ( F( a_{i j} ) )$
induced by $F$ preserves the total positivity or total non-negativity
of matrices.

The study of when entrywise transformations preserve notions of
positivity has a long history. One of the first rigidity theorems for
such transforms was proved by Schoenberg, who showed in the 1940s that
entrywise transforms preserving positive semidefiniteness of matrices
of all sizes must be given by convergent power series with
non-negative coefficients \cite{Schoenberg42}. That any such function
preserves positive semidefiniteness when applied to matrices of
arbitrary dimensions follows immediately from the Schur product
theorem \cite{Schur1911}. Reformulated in the language of kernels,
Schoenberg's theorem shows that the composition operator $C_F$ leaves
invariant the set of positive-semidefinite kernels if and only if $F$
admits a power series representation with non-negative
coefficients. Schoenberg's discovery was part of a larger project of
classifying the invariant distances on homogeneous spaces which are
isometrically equivalent to a Hilbert-space distance; see Bochner's
very informative article \cite{Bochner-pd} for more details. This
circle of ideas was further extended by the next generation of
analysts, to operations which preserve Fourier coefficients of
measures \cite{HKKR}. The analogous result to Schoenberg's theorem for
matrices of a fixed size is more subtle. Roger Horn's doctoral
dissertation contains the fundamental observation, attributed by Horn
to L\"owner, that the size of the positive matrices preserved by a
smooth transform imposes non-negativity constraints on roughly the
same number of its derivatives \cite{horn}. This observation left a
significant mark on probability theory \cite{Horn-toeplitz}.
Determining the exact set of functions that preserve positivity when
applied entrywise to positive-semidefinite matrices of a fixed
dimension remains an open problem and the subject of active research
\cite{BGKP-fixeddim, BGKP-hankel, KT}.

More detail about the evolution of matrix positivity transforms and
their applications in areas such as data science and probability
theory are contained in our recent surveys \cite{BGKP-survey1,
BGKP-survey2}. The investigation of entrywise transforms preserving
total positivity have recently revealed novel connections to type-$A$
representation theory and to combinatorics. We refer the reader to the
recent works \cite{BGKP-fixeddim,GKR-critG} and the recent paper
\cite{KT} by Khare and Tao for more details.

\subsection{Main results}

Recall that a kernel is said to be totally non-negative of order
$p$, denoted $\TN_p$, if all of its minors of size $p \times p$ and
smaller are non-negative. Similarly one defines $\TP_p$ kernels; see
Definition~\ref{Dtntp} for the precise details.

An initial step in classifying total-positivity preservers over an
arbitrary domain $X \times Y$ is to consider separately the cases
where at least one of $X$ and $Y$ is finite and when both are
infinite. The first of these cases leads to the following theorem.

\begin{theorem}\label{T1}
Let $X$ and $Y$ be totally ordered sets, each of size at least $4$,
and let $F: [ 0, \infty ) \to \R$. The operator
$C_F : K \mapsto F \circ K$ maps the set of totally non-negative
kernels of any fixed order at least $4$ to itself if and only if $F$
is constant, so that $F( x ) = c$, or linear, so that $F( x ) = c x$,
with $c \geq 0$. The same holds if totally non-negative kernels are
replaced by totally positive kernels of any fixed order at least $4$,
now with the requirement that $c > 0$.
\end{theorem}

We develop the proof of Theorem~\ref{T1} over several sections. In
fact, we prove more; we provide a full characterization of entrywise
transforms that preserve total non-negativity on $m \times n$ matrices
or symmetric $n \times n$ matrices, for any fixed values of $m$ and
$n$, finite or infinite. We also prove the analogous classifications
for preservers of total positivity on matrices of each size. See
Tables~\ref{Table-tn} and~\ref{Table-tp} for further details,
including variants involving test sets of symmetric matrices.

The proof strategy is broadly as follows. For preservers of total
non-negativity, note that totally non-negative matrices of smaller
size can be embedded into larger ones; this allows us to use, at each
stage, properties of preservers for lower dimensions. Thus, we show
the class of preservers to be increasingly restrictive as the
dimension grows, and as soon as we reach $4 \times 4$ matrices (or
$5 \times 5$ matrices if our test matrices are taken to be symmetric),
we obtain the main result.

For total positivity, the problem is more subtle: as zero entries are
not allowed, one can no longer use the previous technique. Instead,
the key observation is that totally positive matrices are dense in
totally non-negative matrices; this is an approximation theorem due to
A.M.~Whitney, which reduces the problem for continuous functions and
finite sets $X$ and $Y$ to the previous case. The next step then is to
prove the continuity of all total-positivity preservers; we achieve
this by solving two totally positive matrix-completion problems.
Finally, to go from finite $X$ and $Y$ to the case where one of $X$
and $Y$ is infinite, we extend Whitney's approximation theorem to
totally positive kernels on arbitrary subsets of $\R$, as follows.

\begin{theorem}\label{Twhitney2}
Given non-empty subsets $X$ and $Y$ of $\R$, and a positive integer~$p$,
any bounded $\TN_p$ kernel $K : X \times Y \to \R$ can be
approximated locally uniformly at points of continuity in the interior
of $X \times Y$ by a sequence of $\TP_p$ kernels on $X \times Y$.
Furthermore, if $X = Y$ and $K$ is symmetric, then the kernels in the
approximating sequence may be taken to be symmetric.
\end{theorem}

The proof of Theorem~\ref{Twhitney2} is developed in
Section~\ref{SWhitneyExt}, using discretized Gaussian convolution.
The Gaussian function is found throughout mathematics, and this paper
is no exception. It finds itself a crucial ingredient for several of
the arguments below. As well as the discrete convolution, it allows
regular P\'olya frequency functions to be approximated by totally
positive ones, and is employed in various places as a totally positive
kernel which is particularly straightforward to manipulate.

The only remaining case is the classification of total
positivity-preservers for kernels over $X \times Y$, with both $X$ and
$Y$ infinite. In this situation, the test sets used to obtain the
previous results are no longer sufficient, and new tools and test
classes of kernels with additional structure are called for.

When considering other forms of structured kernels, the Hankel
and Toeplitz classes stand out. The study of Hankel kernels with
countable domains leads naturally to moment-preserving maps, and
these form the main body of our previous investigation
\cite{BGKP-hankel}. In the present article, we provide the
classification of preservers for both Hankel and Toeplitz kernels with
domains which are a continuum. A Hankel kernel has the form
\[
X \times X \to \R; \ ( x, y ) \mapsto f( x + y ),
\]
whereas a Toeplitz kernel has the form
\[
X \times X \to \R; \ ( x, y ) \mapsto f( x - y ),
\]
where $X \subseteq \R$.
  
The main results are summarized by the next five theorems; for more
details, see Theorems~\ref{Thankel2} and~\ref{TPhankel} (the Hankel
case), Theorems~\ref{TPolya} and~\ref{TtoeplitzTP} (P\'olya frequency
functions), Theorems~\ref{Ttoeplitz} and~\ref{TtoeplitzTP} (measurable
Toeplitz kernels), Theorems~\ref{Tpfseq} and~\ref{Tpfseq2} (P\'olya
frequency sequences), and Theorems~\ref{T1sidedPFF}
and~\ref{T1sidedPFseq} (one-sided P\'olya frequency functions and
sequences).

The class of $\TN$ or $\TP$ preservers for Hankel kernels consists
essentially of absolutely monotonic functions. This is outlined in the
following result, and our proof relies on prior work of Bernstein,
Hamburger, Mercer, and Widder.

\begin{theorem}
Let $X \subseteq \R$ be an open interval and let
$F : [ 0, \infty ) \to \R$. The composition map $C_F$ preserves the
set of continuous $\TN$ Hankel kernels on $X \times X$ if and only if
$F( x ) = \sum_{k = 0}^\infty c_k x^k$ on $( 0, \infty )$, with
$c_k \geq 0$ for all $k$ and $F( 0 ) \geq 0$.

A similar statement holds for preservers of $\TP$ Hankel kernels on
$X \times X$.
\end{theorem}

In contrast, $\TN$ Toeplitz kernels possess a far more restricted
class of preservers. Recall that a P\'olya frequency function
$\Lambda$ is an integrable function on~$\R$, non-zero at two or more
points, such that the Toeplitz kernel
$T_\Lambda : \R \times \R \to \R; \ ( x, y ) \mapsto \Lambda( x - y )$
is $\TN$.

\begin{theorem}
Let $F : [ 0, \infty ) \to [ 0, \infty )$. The composition map $C_F$ 
preserves the set of P\'olya frequency functions if and only if
$F( x ) = c x$ with $c > 0$.

A similar statement holds for preservers of the class of $\TP$ kernels
of the form $T_\Lambda$, where $\Lambda$ is a P\'olya frequency function.
\end{theorem}

If the integrability condition is removed, then the class of
preservers of Toeplitz kernels is enlarged slightly. In the following
theorem, measurability is required to hold in the sense of Lebesgue.

\begin{theorem}
Let $F : [ 0, \infty ) \to [ 0, \infty )$ be non-zero. The composition
map $C_F$ preserves $\TN$ measurable Toeplitz kernels on
$\R \times \R$ if and only if $F( x ) = c$ or $F( x ) = c x$ or
$F( x ) = c \bJ_{x > 0}$, for some $c > 0$.

The only preservers of $\TP$ Toeplitz kernels on $\R \times \R$,
whether measurable or not, are the dilations $F( x ) = c x$ with
$c > 0$.
\end{theorem}

The discrete analogue of a P\'olya frequency function or a Toeplitz
kernel is a P\'olya frequency sequence, that is, a real sequence
$\ba = ( a_n )_{n \in \Z}$ such that the Toeplitz kernel
$T_\ba : \Z \times \Z \to \R; \ ( i, j ) \mapsto a_{i - j}$ is $\TN$.
These sequences have been widely studied in function theory,
approximation theory, and combinatorics. It turns out their preservers
display the same rigidity.

\begin{theorem}\label{Tpfsq}
Let $F : [ 0, \infty ) \to [ 0, \infty )$. The composition map $C_F$
preserves P\'olya frequency sequences if and only if $F( x ) = c$ or
$F( x ) = c x$, with $c \geq 0$.

A similar result holds for the preservers of $\TP$ P\'olya frequency
sequences.
\end{theorem}

In fact, we prove a more general version of Theorem~\ref{Tpfsq} with
the same rigidity, Theorem~\ref{Tpfseq2}, where the common domain of
the kernels is a pair of subsets that each contain arbitrarily long
arithmetic progressions with equal increments.

As a final variation, we characterize the preservers of one-sided
analogues of P\'olya frequency functions and sequences, and of
Toeplitz kernels, where a kernel is said to be one-sided if the
associated function is: that is, it vanishes on an infinite
semi-axis. The preservers of such kernels, when compared to the
classifications obtained in the three previous theorems, turn out to
be similarly restricted.

\begin{theorem}
Let $F : [ 0, \infty ) \to [ 0, \infty )$. The composition map $C_F$
preserves the following classes,
\begin{enumerate}
\item one-sided P\'olya frequency functions,
\item one-sided $\TN$ measurable Toeplitz kernels on $\R \times \R$,
\item one-sided P\'olya frequency sequences,
\end{enumerate}
if and only if the function $F$ has the following form in each case,
where $c > 0$:
\begin{enumerate}
\item $F( x ) = c x$;
\item $F( x ) = c x$, $F( x ) = c \bJ_{x > 0}$, or $F( x ) = 0$;
\item $F( x ) = c x$, or $F( x ) = 0$.
\end{enumerate}
\end{theorem}

These results on the preservers of Toeplitz kernels rely on work
of Schoenberg and his collaborators on P\'olya frequency functions
and sequences; a comprehensive exposition of this is found in Karlin's
treatise~\cite{Karlin}.

As expected, these classes of invariant kernels and their preservers
touch harmonic analysis, in particular the Fourier--Laplace
transform. We elaborate a few details alongside our classification
proofs in the main body of this article.

As part of our analysis, we provide an example of an even P\'olya frequency
function $M$ with the property that $M^n$ is not a P\'olya frequency
function for every integer $n \geq 2$, and also a one-sided version of
such a function.

As a final consequence, we come full circle to classify $\TP$
preservers, by refining the class of test matrices used to prove
Theorem~\ref{T1} and its ramifications.

\begin{theorem}\label{Ctp}
Let $X$ and $Y$ be infinite, totally ordered sets that admit a $\TP$
kernel on $X \times Y$. A function
$F : ( 0, \infty ) \to ( 0, \infty )$ is such that $C_F$ preserves the
set of $\TP$ kernels on $X \times Y$ if and only
$F( x ) = c x$ with $c > 0$.

A similar result holds for the non-constant preservers of $\TN$
kernels on $X \times Y$, and for the preservers of symmetric
$\TP$ or $\TN$ kernels on $X \times X$.
\end{theorem}

The proof for preservers of $\TP$ kernels on $X \times Y$ uses
Theorem~\ref{Tpfseq2}, together with the observation that any $\TP$
kernel on $X \times Y$ may be used to realize $X$ and $Y$ as subsets
of $\R$. To classify the preservers of symmetric $\TP$ kernels on
$X \times X$, we exploit and unify in a coherent proof most of the
concepts arising in this paper: Vandermonde and Hankel kernels,
P\'olya frequency functions and sequences, order-preserving
embeddings, discretization, and Whitney-type density theorems.

\subsection{Contents}

This work has three distinct themes: preservers of $\TN$ kernels and
$\TP$ kernels, approximation of $\TN$ kernels by $\TP$ kernels, and
preservers of structured kernels possessing various forms of positivity.
In Sections~\ref{Snon-neg} and~\ref{Spos}, we provide
characterizations for endomorphisms of $\TN$ and $\TP$ kernels, under
various restrictions: symmetric or not, matricial (having finite
domains) or not, and so on. See Section~\ref{Stables} for a tabulated
compilation of these. Next, in Section~\ref{Svasudeva}, we strengthen
a result of Vasudeva to show that $\TN$ preservation for symmetric
$2 \times 2$ matrices with positive entries is equivalent to
preservation of positive semidefiniteness on a much smaller set, and
show that such functions must be continuous. Section~\ref{SWhitneyExt}
contains the results on discrete Gaussian convolution required to
extend Whitney's approximation theorem and so establish some of the
preceding characterizations. Section~\ref{Shankel} is devoted to the
classification of composition operators which preserve $\TN$ Hankel
kernels defined on a continuum, and also the $\TP$
case. Section~\ref{Stoeplitz} examines those transforms which leave
invariant the class of P\'olya frequency functions and measurable
Toeplitz kernels (for both the $\TN$ and $\TP$ cases), and
Section~\ref{Stoeplitz2} considers the analogous P\'olya frequency
sequences. The preservers of one-sided P\'olya frequency functions and
sequences are characterized in Section~\ref{Stoeplitz3}. We conclude
in Section~\ref{Sqns} by completing the classification problem for
$\TP$ preservers in both the general setting and for the case of
symmetric kernels. With the practitioner in mind, our final section
collects some information about minimal test families that assure the
rigidity of preservers for the larger class of kernels to which they
belong. At the end of this last section, we record some of the ad hoc
notation used in this paper.

The determination of post-composition transforms for totally positive
kernels that is obtained in the following pages may seem rather
discouraging at first sight: there are only trivial ones. However,
many of the technical ingredients appearing in the proofs may be of
independent interest, such as the Whitney-type approximation for
kernels with infinite support, the $\TP$ completion of Hankel
matrices, the structure of Loewner monotone maps, $\TN$ transforms of
the Gaussian kernel, a family of even P\'olya frequency functions
whose higher integer powers cease to be P\'olya frequency functions,
and an order-preserving embedding of the supports of $\TP_2$-kernels
into the real line.

\subsection*{Acknowledgments}

We thank Percy Deift for raising the question of classifying total
positivity preservers and pointing out the relevance of such a result
for current studies in mathematical physics. We also thank Alan Sokal
for his comments on a preliminary version of the paper. A.B.~is
grateful for the hospitality of the Indian Institute of Science,
Bangalore, where part of this work was carried out. D.G.~was partially
supported by a University of Delaware Research Foundation grant, by a
Simons Foundation collaboration grant for mathematicians, and by a
University of Delaware strategic initiative grant. A.K.~was partially
supported by Ramanujan Fellowship grant SB/S2/RJN-121/2017, MATRICS
grant MTR/2017/000295, and SwarnaJayanti Fellowship grants
SB/SJF/2019-20/14 and DST/SJF/MS/2019/3 from SERB and DST (Govt.~of
India), by grant F.510/25/CAS-II/2018(SAP-I) from UGC (Govt.~of
India), and by a Young Investigator Award from the Infosys Foundation.
M.P.~was partially supported by a Simons Foundation collaboration grant
for mathematicians.

\section{Preliminaries and overview}

Throughout this paper, the abbreviation ``$\TN$'' stands either for
the class of totally non-negative matrices, or the
total-non-negativity property for a matrix or a kernel, and similarly
for ``$\TP$''. A kernel is a map $K : X \times Y \to \R$, where $X$
and $Y$ are sets; a kernel is \emph{symmetric} if $X = Y$ and $K( x, y
) = K( y, x )$ for all $x$, $y \in X$.

\begin{notation}
If $X$ is a totally ordered set and
$n \in \N := \{ 1, 2, 3, \ldots \}$, then
\[
\inc{X}{n} := %
\{ \bx = ( x_1, \ldots, x_n ) \in X^n : x_1 < \cdots < x_n \}
\]
and $[n] := \{ 1, \ldots, n \}$. If $K : X \times Y \to \R$,
where $X$ and $Y$ are totally ordered sets, together with
$\bx \in \inc{X}{m}$ and $\by \in \inc{Y}{n}$, then $K[ \bx; \by ]$ is
defined to be the $m \times n$ matrix such that
\[
K[ \bx; \by ]_{i j} = K( x_i, y_j ) %
\qquad ( i = 1, \ldots, m; j = 1, \ldots, n ).
\]
\end{notation}

\begin{definition}\label{Dtntp}
Let $p \in \N$. We say that $K$ is
\begin{itemize}
\item[(i)] $\TN_p$ if $\det K[ \bx; \by ] \geq 0$ for all
$n \in [p]$, $\bx \in \inc{X}{n}$, and $\by \in \inc{Y}{n}$;
\item[(ii)]
$\TP_p$ if $\det K[ \bx; \by ] > 0$ for all $n \in [p]$,
$\bx \in \inc{X}{n}$, and $\by \in \inc{Y}{n}$.
\end{itemize}
If this holds, we say that the kernel is $\TN$ or $\TP$ of
\emph{order}~$p$, denoted by $\TN_p$ and $\TP_p$ respectively.
The kernel $K$ is said to be $\TN$ if it is $\TN_p$
for all $p \in \N$, and similarly for $\TP$, in which case the
order is infinite.
\end{definition}

\begin{example}\label{EgenVDM}
Given $n \in \N$, positive constants $u_1 < \cdots < u_n$ and
real constants $\alpha_1 < \cdots < \alpha_n$, the generalized
Vandermonde matrix $V = ( u_i^{\alpha_j} )_{i, j = 1}^n$ is totally
positive \cite[Chapter~XIII, \S8, Example~1]{Gantmacher_Vol2}.
Reversing the order of both the rows and columns of~$V$ preserves
$\TP$ (and $\TN$), so the same is true if both sets of inequalities
are reversed.

In particular, the kernels
\[
K : I \times \R \to \R; \ ( x, y ) \mapsto x^y %
\quad \text{and} \quad %
K' : I' \times \R \to \R; \ ( x, y ) \mapsto e^{x y}
\]
are $\TP$ for any sets $I \subseteq ( 0, \infty )$ and
$I' \subseteq \R$.
\end{example}

\begin{example}\label{EgenVDM2}
As a generalization of Example~\ref{EgenVDM}, let $X$ and $Y$ be
totally ordered sets, let $g_X : X \to ( 0, \infty )$ and
$g_Y : Y \to ( 0, \infty )$ be arbitrary, and let
$h_X : X \to ( 0, \infty )$ and $h_Y : Y \to \R$ be increasing.
Then the kernel
\begin{equation}\label{Eexamples}
K : X \times Y \to ( 0, \infty ); \ ( x, y ) \mapsto %
g_X( x ) h_X( x )^{h_Y( y )} g_Y( y )
\end{equation}
is $\TP$. To see this, note that
if $n \in \N$, $\bx \in \inc{X}{n}$ and $\by \in \inc{Y}{n}$, then
\[
K[ \bx; \by ] = \diag( g_X[ \bx ] ) %
( h_X( x_i )^{h_Y( y_j )} )_{i, j = 1}^n \diag( g_Y[ \by ] ),
\]
where $\diag( g_X[ \bx ] )$ is the matrix with
$( g_X( x_1 ), \ldots, g_X( x_n ) )$ on the leading diagonal and zeros
elsewhere, and similarly for $\diag( g_Y[ \by ] )$.
Example~\ref{EgenVDM} now gives that $K[ \bx; \by ]$ has positive
determinant.

In general, for any kernel on $X \times Y$, the properties of being
$\TN$, $\TN_p$, $\TP$, or $\TP_p$ are each preserved after multiplying
by functions $g_X : X \to ( 0, \infty )$ or
$g_Y : Y \to ( 0, \infty )$.
\end{example}

\subsection{Overview of classification results for $\TN$ and $\TP$
kernels}\label{Stables}

Our primary focus in this paper is to classify the functions which,
under composition, preserve classes of totally positive or totally
non-negative kernels on $X \times Y$, where $X$ and $Y$ are totally
ordered sets. In Section~\ref{Snon-neg} and~\ref{Spos}, we consider
sixteen different classes of kernels, according to the following
binary possibilities:
\begin{enumerate}
\item totally non-negative or totally positive;
\item matricial, so that $X$ and $Y$ are finite, or non-matricial,
so that at least one of $X$ and $Y$ is infinite;  
\item order $p$ with $p \geq \min\{ |X|, |Y| \}$ or
$p < \min\{ |X|, |Y| \}$;
\item symmetric, requiring that $X = Y$, or not.
\end{enumerate}

\begin{remark}
If at least one of $X$ and $Y$ is finite, then the preservers of $\TP$
kernels on $X \times Y$ are precisely the preservers of $\TP_p$
kernels on $X \times Y$, for any $p \geq \min \{ |X|, |Y| \}$. The
same observation holds if $\TP$ is replaced by $\TN$, and whether or
not symmetry is imposed. Thus, the first alternative in (3) above may
be replaced by $p = \min \{ |X|, |Y| \}$ and we do this henceforth.
\end{remark}

We now present tabulations of our classification results from the next
two sections.

{\renewcommand{\arraystretch}{1.3}
\begin{table}[h]
\begin{tabular}{|c|c|c|c|c|}
\hline
Characterization of & matricial & non-matricial & symmetric &
symmetric\\
endomorphisms & & & matricial & non-matricial\\
\hline
$p = \min \{ |X|, |Y| \}$ &
Theorem~\ref{Tfixeddim} & Corollary~\ref{Ckernel} &
Theorem~\ref{Tsymmetric} & Theorem~\ref{Tsymmetric} \\
\hline
$p < \min \{ |X|, |Y| \}$ &
Theorem~\ref{Tfixeddim}, & Corollary~\ref{Ckernel}, &
Theorem~\ref{Tsymm1} & Theorem~\ref{Tsymm1} \\
& Remark~\ref{Rtnp} & Remark~\ref{Rtnp} & & \\
\hline
\end{tabular}\vspace*{2mm}
\caption{Total non-negativity preservers}\label{Table-tn}
\end{table}

\begin{table}[h]\label{TTPendo}
\begin{tabular}{|c|c|c|c|c|}
\hline
Characterization of & matricial & non-matricial & symmetric &
symmetric\\
endomorphisms & & & matricial & non-matricial\\
\hline
$p = \min \{ |X|, |Y| \}$ &
Theorem~\ref{Ttp} & Theorem~\ref{Ttpkernel} &
Theorem~\ref{TsymmetricTP} & Corollary~\ref{CtpkernelSym} \\
\hline
$p < \min \{ |X|, |Y| \}$ &
Theorem~\ref{Ttpkernel} & Theorem~\ref{Ttpkernel} &
Theorem~\ref{Tsymm2} & Theorem~\ref{Tsymm2} \\
\hline
\end{tabular}\vspace*{2mm}
\caption{Total-positivity preservers}\label{Table-tp}
\end{table}}

For the most part, the conclusions in the matricial and non-matricial
situations are similar or even the same. However, and especially for
$\TP$ preservers, the proofs are harder when at least one of the index
sets is infinite. In addition to the results for the matricial cases,
and the ideas behind their proofs, we require other, more involved
techniques to extend these results to kernels. A particular issue is
the lack of a tractable test set of $\TP$ kernels.

The preservers of symmetric $\TN_p$ or $\TP_p$ kernels differ
depending on whether $p = \min \{ |X|, |Y| \}$ or
$p < \min \{ |X|, |Y| \}$, and in the latter case these preservers
coincide with the preservers of all $p \times p$ matrices which are
$\TN$ or $\TP$. The proofs rely on the careful analysis of preservers
of totally non-negative kernels in each fixed dimension: we show that
our test sets of $p \times p$ matrices which are $\TN$ occur
already as minors of $( p + 1 ) \times ( p + 1)$ symmetric $\TN$
matrices.

\section{Total non-negativity preservers}\label{Snon-neg}

We now begin to formulate and prove our characterization results for
$\TN$ preservers. In this section, we are interested in understanding
the following family of functions.

\begin{definition}
Given two totally ordered sets $X$ and $Y$, let
\begin{equation}
\sFTN_{X, Y} := \{ F : [ 0, \infty ) \to \R \mid %
\text{if $K : X \times Y \to \R$ is totally non-negative, so is } %
F \circ K \}.
\end{equation}
\end{definition}

We observe first that $\sFTN_{X, Y}$ depends on only rather
coarse features of $X$ and~$Y$. A totally ordered set has an
\emph{ascending chain} if it contains an infinite sequence of
elements $x_1 < x_2 < \cdots$, and similarly for a
\emph{descending chain}. It is well known that an infinite totally
ordered set must contain an ascending chain or a descending chain (or
both).%
\footnote{Here is a proof for completeness: let $X$ be infinite and
totally ordered, let $x_1 \in X$ and suppose $\{ x \in X : x > x_1 \}$
is infinite. (The proof is similar if $\{ x \in X : x < x_1 \}$ is
infinite.) If there is no maximum element in $X$, then starting from
$x_1$ one can inductively produce an ascending chain, as desired.
Else, set $y_1 := \max X$. Inductively, given $y_1$, \ldots, $y_k$,
either the infinite set $[ x_1, y_k )$ has a maximum element
$y_{k+1}$, or one can find an ascending chain $x_1 < x_2 < \cdots$ in
$[ x_1, y_k )$ as before. Thus, if there is no ascending chain
starting from $x_1$, there is a descending chain
$y_1 > y_2 > \cdots $.} We say that two totally ordered sets
\emph{have chains of the same type} if both contain an ascending chain
or both contain a descending chain.

Recall that $[n] := \{ 1, \ldots, n \}$ whenever $n \in \N$.

\begin{proposition}\label{Pdim}
Let $X$ and $Y$ be totally ordered sets. Then 
\begin{enumerate}
\item $\sFTN_{X, Y} = \sFTN_{[n], [n]}$ if at least one of $X$ and $Y$
is finite, and $n = \min\{ | X |, | Y | \}$,
\item $\sFTN_{X, Y} = \sFTN_{\N, \N}$ if $X$ and $Y$ have chains of
the same type,
\item $\sFTN_{X, Y} = \sFTN_{\N, -\N}$ if $X$ and $Y$ are infinite and
do not have chains of the same type.
\end{enumerate}
\end{proposition}

\begin{proof}
The key observation is that, for any $X_0 \subseteq X$ and
$Y_0 \subseteq Y$, a totally non-negative kernel
$K: X_0 \times Y_0 \to \R$ trivially extends to a totally non-negative
kernel $\tilde{K}: X \times Y \to \R$ by ``padding by zeros'', that
is, by setting $K( x, y ) = 0$ whenever
$( x, y ) \in ( X \times Y ) \setminus ( X_0 \times Y_0 )$.
Conversely, it is immediate that any totally non-negative kernel on
$X \times Y$ restricts to a totally non-negative kernel on
$X_0 \times Y_0$.

If $X$ and $Y$ are both infinite, then each contains a copy of $\N$ or
$-\N$. If they both contain copies of $\N$, then padding by zeros
gives (2); similarly, if they both contain copies of $-\N$, noting
that reversing the order of both rows and columns preserves $\TN$, as
observed in Example~\ref{EgenVDM}. If $X$ and $Y$ do not contain
chains of the same type, then (3) holds, reversing rows and columns to
swap the roles of $\N$ and $-\N$ if required.
\end{proof}

Proposition~\ref{Pdim} shows that characterising $\sFTN_{X, Y}$ is
equivalent to determining which functions $F: \R \to \R$ preserve
total non-negativity when applied entrywise to totally non-negative
matrices of a fixed dimension (if $X$ or $Y$ is finite), or to totally
non-negative matrices of all dimensions (if $X$ and $Y$ are
infinite). The main result in this section answers this
question.

Given a domain $I \subseteq \R$, a function $F : I \to \R$,
and a matrix $A = ( a_{i j} ) \in I^{m \times n}$, we denote by
$F[ A ] := ( F( a_{i j} ) )$ the matrix obtained by applying $F$ to
the entries of $A$. We denote the Hadamard powers of~$A$ by
$A^{\circ \alpha} := ( a_{i j}^\alpha )$. The convention
$0^0 := 1$ is adopted throughout.

\begin{theorem}\label{Tfixeddim}
Let $F: [ 0, \infty ) \to \R$ be a function and let
$d := \min\{ m, n \}$, where~$m$ and~$n$ are positive
integers. The following are equivalent.
\begin{enumerate}
\item $F$ preserves total non-negativity entrywise on $m \times n$
matrices.
\item $F$ preserves  total non-negativity entrywise on $d \times d$
matrices.
\item $F$ is either a non-negative constant or
\begin{enumerate}
\item[(a)] $(d = 1)$ $F( x ) \geq 0$;
\item[(b)] $(d = 2)$ $F( x ) = c {\bf 1}_{x > 0}$ or
$c x^\alpha$ for some $c > 0$ and some $\alpha > 0$;
\item[(c)] $(d = 3)$ $F( x ) = c x^\alpha$ for some $c > 0$ and some
$\alpha \geq 1$;
\item[(d)] $(d \geq 4)$ $F( x ) = c x$ for some $c > 0$.
\end{enumerate}
\end{enumerate}
\end{theorem}

\begin{proof}
That $(1) \iff (2)$ is obvious, since the minors of a $m \times n$
matrix have dimension at most $d$. We will now prove that
$(2) \iff (3)$ for each value of $d$.

The result is obvious when $d = 1$, since in this case a matrix is
$\TN$ if and only if its entry is non-negative.

Suppose $F[-]$ preserves $\TN$ on $2 \times 2$ matrices and note that
$F( x ) \geq 0$ for all~$x \geq 0$. Next, consider the following
totally non-negative matrices:
\begin{equation}\label{E2matrices}
A( x, y ) := \begin{pmatrix}
x & x y \\ 1 & y
\end{pmatrix} \qquad \text{and} \qquad
B( x, y ) := \begin{pmatrix}
x y & x \\ y & 1
\end{pmatrix} \qquad (x, y \geq 0).
\end{equation}
Considering the determinants of $F[ A( x, y ) ]$ and
$F[ B( x, y ) ]$ gives that
\begin{equation}\label{Emult}
F( x y ) F( 1 ) = F( x ) F( y ) %
\qquad \text{for all } x, y \geq 0.
\end{equation}
If $F( 1 ) = 0$, then $F( x ) F( y ) = 0$, so $F( x ) = 0$ for
all $x \geq 0$. We will therefore assume that $F( 1 ) > 0$.
If $F( x ) = 0$ for any $x > 0$, then Equation~(\ref{Emult}) implies
that $F \equiv 0$, so we assume that $F( x ) > 0$ for all
$x > 0$. Applying $F$ to the $\TN$ matrix
\begin{equation}\label{Esymrank1}
\begin{pmatrix}
x & \sqrt{x y} \\
\sqrt{x y} & y
\end{pmatrix} \qquad (x, y \geq 0),
\end{equation}
we conclude that $F( \sqrt{x y} )^2 \leq F( x ) F( y )$. As a
result, the function $G( x ) = \log F( e^x )$ is mid-point
convex on~$\R$. Also, applying $F$ to the $\TN$ matrix
\[
\begin{pmatrix}
y & x \\ x & y
\end{pmatrix} \qquad (y \geq x \geq 0)
\]
implies that $F$, so $G$, is non-decreasing. By
\cite[Theorem~71.C]{roberts-varberg}, we conclude that $G$ is
continuous on $\R$, and so $F$ is continuous on $( 0, \infty )$.
Moreover, since $F( 1 ) \neq 0$, Equation~(\ref{Emult}) implies
\[
\frac{F( x y )}{F( 1 )} = \frac{F( x )}{F( 1 )} \frac{F( y )}{F( 1 )},
\]
that is, the function $F / F( 1 )$ is multiplicative. From these
facts, there exists $\alpha \geq 0$ such that
$F( x ) = F( 1 ) x^\alpha$ for all $x > 0$. Finally, setting
$y = 0$ in Equation~(\ref{Emult}), we see that
\[
F( 0 ) F( 1 ) = F( x ) F( 0 ) \qquad \text{for all } x \geq 0.
\]
Thus either $F( 0 ) = 0$ or $F \equiv F( 1 )$; in either case,
the function $F$ has the required form. The converse is immediate,
and this proves the result in the case $d = 2$. 

Next, suppose $F$ preserves $\TN$ on $3 \times 3$ matrices and is
non-constant. Since the matrix $A \oplus \bZ_{1 \times 1}$ is totally
non-negative if the $2 \times 2$ matrix $A$ is, we conclude by
part~(b) that $F( x ) = c x^\alpha$ for some $c > 0$ and
$\alpha \geq 0$.  The matrix
\begin{equation}\label{E3matrix}
C := \begin{pmatrix}
1 & 1 / \sqrt{2} & 0 \\
1 / \sqrt{2} & 1 & 1 / \sqrt{2} \\
0 & 1 / \sqrt{2} & 1
\end{pmatrix}
\end{equation}
is totally non-negative, and
$\det F[ C ] = c^3 ( 1 - 2^{1 - \alpha} )$. It follows
that $F$ does not preserve $\TN$ on $3 \times 3$ matrices
when $\alpha < 1$. For higher powers, we use the following
result \cite[Theorem 4.2]{JW}; see \cite[Theorem 5.2]{FJS} for a
shorter proof.
\begin{equation}\label{E3x3}
\alpha \geq 1 \qquad \implies \qquad %
x^\alpha \text{ preserves $\TN$ and $\TP$ on $3 \times 3$ matrices}.
\end{equation}
This concludes the proof of the case $d = 3$.

Finally, suppose $F$ is non-constant and preserves $\TN$ on
$4 \times 4$ matrices. Similarly to the above, considering matrices
of the form $A \oplus \bZ_{1 \times 1}$ gives, by part (c), that
$F( x ) = c x^\alpha$ for some $c > 0$ and some $\alpha \geq 1$. We
now appeal to \cite[Example~5.8]{FJS}, which examines Hadamard powers
of the family of matrices
$N( \eps, x ) := \bJ_{4 \times 4} + x M( \eps )$, where
\begin{equation}\label{E4matrix}
\bJ_{4 \times 4} := \begin{pmatrix}
1 & 1 & 1 & 1 \\ 1 & 1 & 1 & 1 \\ 1 & 1 & 1 & 1 \\ 1 & 1 & 1 & 1
\end{pmatrix} \qquad \text{and} \qquad %
M( \eps ) := \begin{pmatrix}
0 & 0 & 0 & 0 \\ 0 & 1 & 2 & 3 \\
0 & 2 & 4 + \eps & 6 + \frac{5}{2} \eps \\
0 & 3 & 8 & 14 + \eps
\end{pmatrix}.
\end{equation}
As shown therein, the matrix $N( \eps, x )$ is $\TN$ for all
$\eps \in ( 0, 1 )$ and $x > 0$. Moreover, for small $x$ and any
$\alpha > 1$, the determinant of the Hadamard power
\[
\det N( \eps, x )^{\circ \alpha} = \eps^2 \alpha^3 x^3 + %
\frac{1}{4} ( 8 - 70 \eps - 59 \eps^2 - 4 \eps^3 ) %
(\alpha^3 - \alpha^4 ) x^4 + O(x^5).
\]
Thus $\det F[ N( \eps, x ) ] < 0$ for sufficiently small
$\eps = \eps( \alpha ) > 0$ and $x > 0$. We conclude that
$F( x ) = c x$ if $d = 4$. More generally, if $F$ preserves $\TN$ on
$d \times d$ matrices, where $d \geq 4$, then $F$ also preserves
$\TN$ on $4 \times 4$ matrices, and so $F( x ) = c x$ for some
$c > 0$, as desired. The converse is immediate.
\end{proof}

Proposition~\ref{Pdim} and Theorem~\ref{Tfixeddim} immediately combine
to yield the following exact description of the set $\sFTN_{X, Y}$.

\begin{corollary}\label{Ckernel}
Let $X$ and $Y$ be totally ordered sets. Then
\begin{enumerate}
\item
$\sFTN_{X, Y} = %
\{ F : \R \to \R \mid F( x ) \geq 0 \text{ for all } x \in \R\}$ 
if $\min\{ | X |, | Y | \} = 1$.
\item
$\sFTN_{X, Y} = \{ c, \ c {\bf 1}_{x > 0}, \ c x^\alpha : c \geq 0, \ \alpha > 0 \}$
if $\min\{ | X |, | Y | \} = 2$. 
\item
$\sFTN_{X, Y} = \{ c, \ c x^\alpha : c \geq 0, \ \alpha \geq
1 \}$ if $\min\{ | X |, | Y | \} = 3$.
\item
$\sFTN_{X, Y} = \{ c, \ c x : c \geq 0 \}$
if $\min\{ | X |, | Y | \} \geq 4$ or if $X$ and $Y$ are infinite.
\end{enumerate}
\end{corollary}

\begin{remark}\label{Rtnp}
Given a positive integer $p \leq \min\{ | X |, | Y | \}$,
Corollary~\ref{Ckernel} immediately classifies the collection of all
functions mapping the set of $\TN_p$ kernels on $X \times Y$ to
itself. This is because any $\TN_p$ kernel on $[p] \times Y$ or
$X \times [p]$ extends by ``padding by zeros'', as in the proof of
Proposition~\ref{Pdim}, to a $\TN$ kernel on $X \times Y$.
\end{remark}

\subsection{Preservers of symmetric $\TN$ matrices and kernels}

Theorem~\ref{Tfixeddim} and Corollary \ref{Ckernel} have a natural
analogue for totally non-negative matrices and kernels which are
symmetric. Note that any such matrix has non-negative principal minors
and is therefore positive semidefinite.

\begin{theorem}\label{Tsymmetric}
Let $F: [ 0, \infty ) \to \R$ and let $d$ be a positive integer. The
following are equivalent.
\begin{enumerate}
\item $F$ preserves total non-negativity entrywise on symmetric
$d \times d$ matrices.
\item $F$ is either a non-negative constant or
\begin{enumerate}
\item[(a)] $(d = 1)$ $F( x ) \geq 0$;
\item[(b)] $(d = 2)$ $F$ is non-negative, non-decreasing,
and multiplicatively mid-convex, that is,
$F( \sqrt{x y} )^2 \leq F( x ) F( y )$ for all
$x$, $y \in [ 0,\infty )$, so continuous on $( 0, \infty )$;
\item[(c)] $(d = 3)$ $F( x ) = c x^\alpha$ for some $c > 0$ and some
$\alpha \geq 1$;
\item[(d)] $(d = 4)$ $F( x ) = c x^\alpha$ for some $c > 0$ and some
$\alpha \in \{1\} \cup [ 2, \infty )$;
\item[(e)] $(d \geq 5)$ $F( x ) = c x$ for some $c > 0$.
\end{enumerate}
\end{enumerate}
The same characterizations hold for the preservers of symmetric
$\TN$ kernels on $X \times X$, where $X$ is a totally ordered set of size
$d$, which may now be infinite.
\end{theorem}

\begin{proof}
The result is trivial when $d = 1$. When $d = 2$, a symmetric matrix
is $\TN$ if and only if it is positive semidefinite, so part~(b)
follows immediately from \cite[Theorem~2.5]{GKR-lowrank}.

Now, suppose $F$ preserves $\TN$ entrywise on symmetric
$3 \times 3$ matrices and is non-constant. Considering
matrices of the form $A \oplus \bZ_{1 \times 1}$, it follows from
part~(b) that $F$ is non-decreasing and continuous on $( 0, \infty )$.
Applying $F$ entrywise to the matrix $x \, \Id_3$ for $x > 0$,
where $\Id_3$ is the $3 \times 3$ identity matrix, it follows easily
that $F( 0 ) = 0$. Next, let
$L := \lim_{\eps \to 0^+} F( \eps )$, which exists since $F$
is non-decreasing, and let $C$ be the $\TN$ matrix in
Equation~(\ref{E3matrix}). Then
$0 \leq \lim_{\eps \to 0^+} \det F[ \eps C ] = -L^3$,
whence $L = 0$. Thus $F$ is continuous on $[ 0, \infty )$. Next,
consider the symmetric totally non-negative matrices
\begin{equation}\label{Etp2}
A'( x, y ) := \begin{pmatrix}
x^2 & x & x y \\
x & 1 & y \\
x y & y & y^2
\end{pmatrix} \quad \text{and} \quad %
B'(x,y) := \begin{pmatrix}
x^2 y & x y & x \\
x y & y & 1 \\
x & 1 & 1 / y
\end{pmatrix} \qquad ( x \geq 0, y > 0 ).
\end{equation}
Note that $A'( x, y )$ contains the matrix $A( x, y )$ from
Equation~(\ref{E2matrices}) as a submatrix, and the same is true for
$B'( x, y )$ and $B( x, y )$. As in the proof of
Theorem~\ref{Tfixeddim}(a), it follows that
\[
F( x y ) F( 1 ) = F( x ) F( y ) \qquad %
\text{for all } x, y \geq 0.
\]
Proceeding as there, and noting that the matrix $C$ from
Equation~(\ref{E3matrix}) is symmetric, we obtain $c > 0$ and
$\alpha \geq 1$ such that $F( x ) = c x^\alpha$. Moreover, each
function of this form preserves $\TN$ entrywise, by~(\ref{E3x3}).
This concludes the proof of part~(c).

To prove~(d), we suppose the non-constant function $F$ preserves $\TN$
on symmetric $4 \times 4$ matrices, and use part~(c) with the usual
embedding to obtain $c > 0$ and $\alpha \geq 1$ such that
$F( x ) = c x^\alpha$. To rule out $\alpha \in ( 1 , 2 )$, let
$x \in ( 0, 1 )$ and note that the infinite matrix
$( 1 + x^{i + j} )_{i,j \geq 0}$ is the moment matrix of the two-point
measure $\delta_1 + \delta_x$.  Its leading principal $4 \times 4$
submatrix $D$ is $\TN$, by classical results in the theory of moments
\cite{GK,STmoment}, but if $\alpha \in ( 1, 2 )$ then
$D^{\circ \alpha}$ is not positive semidefinite, hence not $\TN$, by
\cite[Theorem~1.1]{Jain}.  The converse follows from
\cite[Proposition~5.6]{FJS}. This proves~(d).

Finally, suppose $F$ is non-constant and preserves $\TN$ on
$5 \times 5$ symmetric matrices, and apply part (d) to obtain $c > 0$
and $\alpha \in \{ 1 \} \cup [ 2, \infty )$ such that
$F( x ) = c x^\alpha$. To rule out the case $\alpha \geq 2$, we appeal
to \cite[Example~5.10]{FJS}, which studies the symmetric, totally
non-negative matrices
\begin{equation}\label{E5matrix}
T( x ) := \bJ_{5 \times 5} + x \begin{pmatrix}
2 & 3 & 6 & 14 & 36 \\
    3 & 6 & 14 & 36 & 98\\
        6 & 14 & 36 & 98 & 276\\
	    14 & 36 & 98 & 284 & 842\\
	         36 & 98 & 276 & 842 & 2604
\end{pmatrix} \qquad ( x > 0 ).
\end{equation}
It is shown there that, for every $\alpha > 1$, there exists
$\eps = \eps( \alpha ) > 0$ such that the upper right
$4 \times 4$ submatrix of $T( x )^{\circ \alpha}$ has negative
determinant whenever $x \in ( 0, \eps )$. It now follows that
$F( x ) = c x$ if $d = 5$. The general case, where $d \geq 5$,
follows by the usual embedding trick, and the converse is once again
immediate.

The final assertion is immediate, via padding by zeros.
\end{proof}

We conclude this section with a characterization of symmetric $\TN_p$
preservers which is parallel to Remark~\ref{Rtnp}.

\begin{theorem}\label{Tsymm1}
Let $F: [ 0, \infty ) \to \R$ and let $d$ and $p$ be positive
integers, with $p < d$. The following are equivalent.
\begin{enumerate}
\item $F$ preserves $\TN_p$ entrywise on symmetric $d \times d$ matrices.
\item $F$ preserves $\TN_p$ entrywise on $d \times d$ matrices.
\item $F$ is either a non-negative constant or
\begin{enumerate}
\item[(a)] $(p = 1)$ $F( x ) \geq 0$;
\item[(b)] $(p = 2)$ $F( x ) = c x^\alpha$ for some $c > 0$ and some
$\alpha \geq 0$;
\item[(c)] $(p = 3)$ $F( x ) = c x^\alpha$ for some $c > 0$ and some
$\alpha \geq 1$;
\item[(d)] $(p \geq 4)$ $F( x ) = c x$ for some $c > 0$.
\end{enumerate}
\end{enumerate}
The same functions characterize the preservers of symmetric $\TN_p$
kernels on $X \times X$, where $X$ is a totally ordered set of size at
least $p + 1$ (and possibly infinite).
\end{theorem}

In a sense, this result not immediately following from
Theorem~\ref{Tsymmetric} is a manifestation of the fact that the
definition of $\TN_p$ for a symmetric kernel differs from asking for
every \emph{principal} $r \times r$ minor being non-negative for
$1 \leq r \leq p$.

\begin{proof}
That $(2) \implies (1)$ is immediate, while the equivalence
$(2) \iff (3)$ follows from Theorem~\ref{Tfixeddim}, since $(2)$ is
equivalent to preserving $\TN$ for $p \times p$ matrices. To see that
$(1) \implies (3)$, it suffices to note that test matrices used to prove
Theorems~\ref{Tfixeddim} and~\ref{Tsymmetric} occur as
submatrices of $d \times d$ symmetric matrices which are $\TN_p$.

This is immediate for $p = 1$, while for $p = 2$ the matrices
in~(\ref{E2matrices}) and~(\ref{Esymrank1}) embed as required,
using (\ref{Etp2}) for the former and padding with zeros as necessary.
Now working as in the proof of Theorem~\ref{Tfixeddim}(b) gives
that $F( x ) = c x^\alpha$ with $c > 0$ and~$\alpha \geq 0$.

Next, suppose $p = 3$. Then the $p=2$ case, together with the
matrix~(\ref{E3matrix}), implies as in the proof of
Theorem~\ref{Tfixeddim}(c) that $\alpha \geq 1$. Finally, if $p = 4$,
then the matrices~(\ref{E5matrix}) imply as in the proof of
Theorem~\ref{Tsymmetric}(e) that $\alpha = 1$.

This concludes the proof for matrices, and the extension to kernels
follows once again via padding by zeros.
\end{proof}

\section{Total-positivity preservers. I. Semi-finite domains}\label{Spos}

We now turn to the more challenging problem of determining the
functions which leave invariant the set of totally positive kernels,
\[
\sFTP_{X, Y} := \{ F : ( 0, \infty ) \to \R \mid \text{if } %
K : X \times Y \to \R \text{ is totally positive, so is } %
F \circ K \}.
\]

There are two technical challenges one encounters once the underlying
inequalities are strict. First, the embedding technique used to prove
Theorem~\ref{Tfixeddim}, which realises totally non-negative
$d \times d$ matrices as submatrices of totally non-negative
$( d + 1 ) \times ( d + 1 )$ matrices, is lost. Second, the crucial
property of multiplicative mid-point convexity is no longer available,
since the matrices in (\ref{E2matrices}) and (\ref{Esymrank1}) are not
always totally positive. Following the approach of the previous
section, we begin by indicating how these challenges can be addressed
in the finite-dimensional case.

\begin{theorem}\label{Ttp}
Let $F: ( 0, \infty ) \to \R$ be a function and let
$d := \min\{ m, n \}$, where $m$ and~$n$ are positive
integers. The following are equivalent.
\begin{enumerate}
\item $F$ preserves total positivity entrywise on $m \times n$
matrices.
\item $F$ preserves total positivity entrywise on $d \times d$
matrices.
\item The function $F$ satisfies
\begin{enumerate}
\item[(a)] $(d = 1)$ $F( x ) > 0$;
\item[(b)] $(d = 2)$ $F( x ) = c x^\alpha$ for some $c > 0$ and some
$\alpha > 0$;
\item[(c)] $(d = 3)$ $F( x ) = c x^\alpha$ for some $c > 0$ and some
$\alpha \geq 1$.
\item[(d)] $(d \geq 4)$ $F(x) = c x$ for some $c > 0$.
\end{enumerate}
\end{enumerate}
\end{theorem}

In order to prove Theorem~\ref{Ttp}, we formulate two auxiliary
results. We say that an $m \times n$ matrix $( a_{i j} )$
occurs as a \emph{submatrix} of a kernel $K$ on $X \times Y$ if there
exist $x_1 < \cdots < x_m$ and $y_1 < \cdots < y_n$ such that $a_{i j}
= K( x_i, y_j )$ for all $i \in [m]$ and $j \in [n]$.

\begin{lemma}\label{L2x2tp}
Fix integers $m \geq 2$ and $n \geq 2$. Every totally positive
$2 \times 2$ matrix occurs as the leading principal
submatrix of a positive multiple of a $m \times n$
generalized Vandermonde matrix, which is necessarily totally positive.
\end{lemma}
 
In fact, any $\TP$ $2 \times 2$ matrix can be embedded at any
specified location within a generalized Vandermonde matrix.%
\footnote{We thank Prakhar Gupta and Pranjal Warade for this
observation.}
A stronger version of this result is given by
Theorem~\ref{Ttpcompletion} below.

Lemma~\ref{L2x2tp} is an example of a \emph{totally positive
completion problem} \cite{FJSm}. Embedding results are known for
arbitrary totally positive matrices, using, for example, the
exterior-bordering technique discussed in \cite[Chapter 9]{FJ} or the
parametrizations available in \cite{BFZ,FZ-2}. Lemma~\ref{L2x2tp} has
the advantage of providing an explicit embedding into the well-known
class of Vandermonde kernels, and is crucial to our final
characterization results, found in the penultimate section of this
paper.

The second result we require is a density theorem derived by
A.M.~Whitney in~1952, using generalized Vandermonde matrices and the
Cauchy--Binet identity. The symmetric variant has the same proof as
the version without this requirement.

\begin{theorem}[Whitney, {\cite[Theorem 1]{Whi}}]\label{Twhitney}
Given positive integers $m$, $n$, and $p$, the set of $\TP_p$
$m \times n$ matrices is dense in the set of $\TN_p$ $m \times n$
matrices. The same is true if both sets of matrices are taken to be
symmetric.
\end{theorem}

With these two observations to hand, we can now classify
total-positivity preservers.

\begin{proof}[Proof of Theorem~\ref{Ttp}]
That $(3) \implies (2)$ and $(2) \implies (1)$ are immediate, with the
former using~(\ref{E3x3}) when $d = 3$. We now prove that
$(1) \implies (3)$. The case $d = 1$ is immediate, so we assume that
$d \geq 2$. By Lemma~\ref{L2x2tp}, the map $F[-]$ preserves $\TP$ on
$2 \times 2$ matrices. Considering the action of $F[-]$ on the
matrices
\[
\begin{pmatrix}
y & x \\
x & x 
\end{pmatrix} \qquad ( y > x > 0)
\]
gives that $F$ takes positive values and is increasing on
$( 0, \infty )$. Thus $F$ is Borel measurable and continuous
outside a countable set. Let $a > 0$ be a point of continuity
and consider the totally positive matrices
\[
A( x, y, \eps ) := \begin{pmatrix} a x & a x y \\
a - \eps & a y \end{pmatrix} \text{ and } %
B( x, y, \eps ) := \begin{pmatrix} a x y & a x \\
a y & a + \eps
\end{pmatrix} %
\quad ( x , y > 0, \ 0 < \eps < a ).
\]
Then
\begin{align*}
0 & \leq \lim_{\eps \to 0^+} \det F[ A( x, y, \eps ) ] = %
F( a x ) F( a y ) - F( a x y ) F( a ) \\
\text{and} \quad 0 & \leq %
\lim_{\eps \to 0^+} \det F[ B( x, y , \eps ) ] = %
F( a ) F( a x y ) - F( a x ) F( a y ).
\end{align*}
Hence, letting $G( x ) := F( a x ) / F( a )$, we have that
\[
G( x y ) = G( x ) G( y ) \qquad \text{for all } x, y > 0.
\]
Since $G$ is measurable, classical results of Sierpi\'nsky
\cite{Sierpinsky} and Banach \cite{Banach} on the Cauchy functional
equation imply there exists $\alpha \in \R$ such that
$G( x ) = x^\alpha$ for all $x > 0$. Thus if
$c := F( a  ) a^{-\alpha} > 0$, then
\[
F( x ) = F( a ) ( x / a )^\alpha = c x^\alpha
\qquad \text{for all } x > 0.
\]
As $F$ is increasing, it holds that $\alpha > 0$. Hence
$F( x ) = c x^\alpha$ for some $c > 0$ and $\alpha > 0$. The result
follows immediately if $d = 2$.

Finally, suppose $d \geq 3$. Since $F( x ) = c x^\alpha$ for some
$c > 0$ and $\alpha > 0$, it admits a continuous extension $\tilde{F}$
to $[ 0, \infty )$. By Theorem~\ref{Twhitney}, we conclude that
$\tilde{F}$ preserves $\TN$ entrywise on $m \times n$ matrices.
Theorem~\ref{Tfixeddim} gives the form of $\tilde{F}$, and restricting
to~$( 0, \infty )$ shows that $F$ is as claimed. This proves that
$(1) \implies (3)$, which completes the proof.
\end{proof}

The proof of Theorem~\ref{Ttp} relies on Lemma~\ref{L2x2tp}; we will
prove a stronger result presently. For now, we determine the set
$\sFTP_{X, Y}$ when at most one of $X$ and $Y$ is infinite.

\begin{theorem}\label{Ttpkernel}
Let $X$ and $Y$ be non-empty totally ordered sets. Then
\begin{enumerate}
\item[(a)]
$\sFTP_{X, Y} = %
\{ F: ( 0, \infty ) \to ( 0, \infty ) \}$ if
$\min\{ | X |, | Y | \} = 1$.
\item[(b)]
$\sFTP_{X, Y} = \{ c x^\alpha : c > 0, \ \alpha > 0 \}$ if
$\min\{ | X |, | Y | \} = 2$.
\item[(c)]
$\sFTP_{X, Y} = \{ c x^\alpha : c > 0, \ \alpha \geq 1\}$ if
$\min\{ | X |, | Y | \} = 3$.
\item[(d)]
$\sFTP_{X, Y} = \{ c x : c > 0 \}$ if
$4 \leq \min\{ | X |, | Y | \} < \infty$.
\end{enumerate}
Furthermore, if $p \in \N$ and both $X$ and $Y$ are of size at least $p$
(and possibly infinite), then the functions preserving $\TP_p$ kernels on
$X \times Y$ are as above, with $\min\{ |X|, |Y| \}$ replaced by $p$.
\end{theorem}

When $X$ and $Y$ are both finite, Theorem~\ref{Ttpkernel} follows
directly from Theorem~\ref{Ttp}. The case where one of $X$ or $Y$ is
infinite is significantly more complicated. First, observe that for
$X_0 \subseteq X$ and $Y_0 \subseteq Y$, a $\TP$ kernel on
$X_0 \times Y_0$ cannot be extended to a $\TP$ kernel on $X \times Y$
simply by ``padding by zeros'', that is, by defining $K( x, y ) = 0$
on the complement of $X_0 \times Y_0$. In the absence of a suitable
extension result, we instead generalize Whitney's approximation
theorem (Theorem~\ref{Twhitney}) to arbitrary domains; see
Theorem~\ref{Twhitney2} above. The proof is obtained in
Section~\ref{SWhitneyExt} with the help of a form of discretized
Gaussian convolution.

Moreover, contrary to the $\TN$ case of Proposition~\ref{Pdim}, the
set $\sFTP_{X, Y}$ does not only depend on whether $X$ and $Y$ are
finite or not. For example, suppose $X$ is of cardinality strictly
larger than the continuum, and $| Y | \geq 2$.  Choose distinct $y_1$
and $y_2$ in $Y$; since $| X | > | \R^2 |$, by the pigeonhole
principle it follows that $K|_{X \times \{ y_1, y_2 \}}$ contains a
$2 \times 2$ submatrix with equal columns, which is therefore
singular. This shows that $K$ cannot be $\TP$, and the ``test set'' in
$\sFTP_{X, Y}$ is, in fact, empty. Thus, the existence of a $\TP$
kernel on $X \times Y$ already imposes constraints on the sets $X$
and~$Y$.

The following result characterizes such sets, which must be
order-isomorphic to subsets of the real line.

\begin{lemma}\label{Ltptosets}
Suppose $X$ and $Y$ are non-empty totally ordered sets. The following
are equivalent.
\begin{enumerate}
\item There exists a totally positive kernel $K : X \times Y \to \R$.

\item There exists a $\TP_2$ kernel $K : X \times Y \to \R$.

\item Either $X$ or $Y$ is a singleton, or there exist order-preserving
injections from $X$ and $Y$ into $( 0, \infty )$.
\end{enumerate}
The same equivalence holds if $X = Y$ and the kernels in $(1)$ and
$(2)$ are taken to be symmetric.
\end{lemma}

\begin{proof}
If $(3)$ holds and $X$ or $Y$ is a singleton, then the constant kernel
$K \equiv 1$ shows that $(1)$ holds. Otherwise, identify $X$ and $Y$
with subsets of $\R$ via order-preserving injections, and note that
the restriction of $K'$ from Example~\ref{EgenVDM} is totally
positive. Hence $(3) \implies (1)$. Clearly $(1) \implies (2)$, so it
remains to show that $(2) \implies (3)$.

Suppose $(2)$ holds, and neither $X$ nor $Y$ is a singleton. Fix
$y_1 < y_2$ in $Y$; the $\TP_2$ property of $K$ implies that the ratio
function
\[
\varphi : X \to ( 0, \infty ); \ x \mapsto K( x, y_2 ) / K( x, y_1 )
\]
is strictly increasing, so is an order-preserving injection. The same
working applies with the roles of $X$ and $Y$ exchanged, and so (3)
holds.

Finally, note that the same proof goes through verbatim if $X = Y$ and
all kernels under consideration are required to be symmetric.
\end{proof}

Lemma~\ref{Ltptosets} is useful not only in proving
Theorem~\ref{Ttpkernel}, but also for proving a stronger form of
Lemma~\ref{L2x2tp} that was promised above. A $\TP$ $2 \times 2$
matrix, which is necessarily proportional to one of generalized
Vandermonde form, can be embedded in any position, not just in a $\TP$
matrix, but in a Vandermonde kernel on an essentially arbitrary
domain.

\begin{theorem}\label{Ttpcompletion}
Let $A$ be a real $2 \times 2$ matrix. The following are equivalent.
\begin{enumerate}
\item Given $\{ i_1 < i_2 \} \subseteq [m]$ and
$\{ j_1 < j_2 \} \subseteq [n]$, where $m$, $n \geq 2$, there exists
an $m \times n$ matrix $\widetilde{A}$, which is a positive
multiple of generalized Vandermonde matrix, such
that $\widetilde{A}_{i_p, j_q} = a_{p q}$ for $p$, $q = 1$, $2$.

\item Given totally ordered sets $X$ and $Y$, such that $X \times Y$
admits a $\TP$ kernel, and pairs $\{ x_1 < x_2 \} \subseteq X$ and
$\{ y_1 < y_2 \} \subseteq Y$, there exists a $\TP$ kernel $K$ on
$X \times Y$ such that $K[ ( x_1, x_2 ); ( y_1, y_2 ) ] = A$.

\item The matrix $A$ is $\TP$.
\end{enumerate}
\end{theorem}

This immediately implies Lemma~\ref{L2x2tp}, and so completes the
proof of Theorem~\ref{Ttp}.

\begin{proof}
Clearly, $(1)$ and $(2)$ each imply $(3)$. We will show that
$(3) \implies (2)$; the construction used for this also shows that
$(3) \implies (1)$. Furthermore, as $X$ and $Y$ both embed inside
$\R$, by Lemma~\ref{Ltptosets}, henceforth we will consider $X$ and
$Y$ to be subsets of~$\R$.

We first show that an arbitrary $\TP$ $2 \times 2$ matrix
$A$ has the form $\lambda^{-1} ( u_i^{\alpha_j} )_{i, j = 1}^2$, where
the terms $\lambda$, $u_1$, and $u_2$ are positive, $\alpha_1$ and
$\alpha_2$ are real, and either $u_1 < u_2$ and $\alpha_1 < \alpha_2$,
or $u_1 > u_2$ and $\alpha_1 > \alpha_2$. The proof goes through
various cases.

Suppose first that three entries of $A$ are equal. Rescaling the
matrix $A$, there are four cases to consider:
\[
A_1 = \begin{pmatrix} x & 1 \\ 1 & 1 \end{pmatrix}, \quad %
A_2 = \begin{pmatrix} 1 & y \\ 1 & 1 \end{pmatrix}, \quad %
A_3 = \begin{pmatrix} 1 & 1 \\ y & 1 \end{pmatrix}, \quad %
\text{and} \quad %
A_4 = \begin{pmatrix} 1 & 1 \\ 1 & x \end{pmatrix}, 
\]
where $x > 1$ and $0 < y < 1$. In the first case, the matrix
$A_1$ equals $( u_i^{\alpha_j} )$ where $u_1 = x$, $u_2 = 1$,
$\alpha_1 = 1$, and $\alpha_2 = 0$. A similar construction can easily
be obtained for $A_2, A_3$, and $A_4$.

Next, suppose two entries in a row or column of $A$ are equal. 
There are again four cases:
\[
A_5 = \begin{pmatrix} 1 & 1 \\ x & y \end{pmatrix}, \quad %
A_6 = \begin{pmatrix} y & x \\ 1 & 1 \end{pmatrix}, \quad %
A_7 = \begin{pmatrix} y & 1 \\ x & 1 \end{pmatrix}, \quad %
\text{and} \quad %
A_8 = \begin{pmatrix} 1 & x \\ 1 & y \end{pmatrix},
\]
where $y > x > 0$ and $x$, $y \neq 1$. For $A_5$, we can take
$u_1 = 1$, $u_2 = x$, $\alpha_1 = 1$, and
$\alpha_2 = \log y / \log x$. If $u_1 < u_2$, then
$\alpha_1 < \alpha_2$; similarly, when $u_1 > u_2$, we have that
$\alpha_1 > \alpha_2$. Thus $A_5$ can be written as desired. The other
cases are similar.

The remaining case is when
\[
A := \begin{pmatrix} v & w \\ x & y \end{pmatrix} \qquad 
( v, w, x, y > 0, \  v y - w x > 0 ),
\]
with $\{ v, y \} \cap \{ w, x \} = \emptyset$. We claim there exist
$\lambda$, $u_1$, $u_2 > 0$, $\alpha_1 = 1$, and $\alpha_2$ such that
\[
\lambda \begin{pmatrix} v & w \\ x & y \end{pmatrix} =
\begin{pmatrix} u_1 & u_1^{\alpha_2} \\ u_2 & u_2^{\alpha_2}
\end{pmatrix},
\]
and either $u_1 < u_2$ and $\alpha_1 < \alpha_2$, or
$u_1 > u_2$ and $\alpha_1 > \alpha_2$. Applying the logarithm
entrywise to both matrices and computing the determinants gives that
\[
( L + V ) ( L + Y ) = ( L + W )( L + X ),
\]
where $L = \log \lambda$, $V = \log v$, $W = \log w$, $X = \log x$,
and $Y = \log y$. This yields a linear equation in~$L$, whence
\[
\lambda = \exp \left(
\frac{\log w \log x - \log v \log y}{\log( v y / w x )}
\right).
\]
Clearly, $u_1 = \lambda v$ and $u_2 = \lambda x$. Solving for
$\alpha_2$ explicitly, we obtain
\[
\alpha_2 = \frac{\log( w / y )}{\log( v / x )}.
\]
There are now two cases: if $u_1 < u_2$, then $v < x$, so
$w / y < v / x < 1$ and $\alpha_2 > 1 = \alpha_1$. If, instead,
$u_1 > u_2$, then $v / x > 1$ and $\alpha_2 < 1 = \alpha_1$.

Thus, for some $\lambda > 0$, the matrix
$\lambda A$ is of the form $( \exp( \alpha_i \beta_j ) )_{i,j=1}^2$
with either $\alpha_1 < \alpha_2$ and $\beta_1 < \beta_2$, or
$\alpha_1 > \alpha_2$ and $\beta_1 > \beta_2$. Furthermore, the latter
case reduces to the former, since
\[
\lambda A = ( \exp( \alpha'_i \beta'_j ) )_{i,j = 1}^2, %
\quad \text{with } %
\alpha'_i = -\alpha_i \text{ and } \beta'_j = -\beta_j.
\]
Thus $A$ occurs as a submatrix of the scaled Vandermonde kernel
\[
\R \times \R \to \R; \ ( x, y ) \mapsto \lambda^{-1} \exp( x y ).
\]

To pass to a kernel on $X \times Y$, where $X$ and $Y$ are real sets,
fix $x_1$, $x_2 \in X$ and $y_1$, $y_2 \in Y$, where $x_1 < x_2$ and
$y_1 < y_2$, and let $\varphi_X : X \to \R$ and $\varphi_Y : Y \to \R$
be the unique linear maps such that $\varphi_X( x_i ) := \alpha_i$ and
$\varphi_Y( y_j ) := \beta_j$ ($i$, $j = 1$, $2$). Then $A$ occurs as
the submatrix $K[ (x_1, x_2 ); ( y_1, y_2 )]$ of the kernel
\[
K : X \times Y \to \R; \ ( x, y ) \mapsto %
\lambda^{-1} \exp( \varphi_X( x ) \varphi_Y( y ) ).
\qedhere
\]
\end{proof}

Using these results, we can now classify the preservers of $\TP$
kernels on possibly infinite domains.

\begin{proof}[Proof of Theorem~\ref{Ttpkernel}]
We consider the two settings in a uniform manner: suppose $p \in \N$
and either (i)~$p = \min \{ |X|, |Y| \}$ and $F$ preserves $\TP$
kernels on $X \times Y$, or (ii)~$X$ and $Y$ both have size at least~$p$
and $F$ preserves $\TP_p$ kernels on $X \times Y$.

If $p = 1$ then the result is immediate, so suppose $p \geq 2$. By
Lemma~\ref{Ltptosets}, $X$ and $Y$ can be identified with subsets of
$( 0, \infty )$. Furthermore, by Lemma~\ref{L2x2tp} and using
suitable order-preserving maps, every $\TP$ $2 \times 2$ matrix can be
realized as a submatrix of a $\TP$ kernel on $X \times Y$.
Hence, by Theorem~\ref{Ttp}(3b), the function $F$ has the form
$F( x ) = c x^\alpha$ for some $c > 0$ and $\alpha > 0$. Conversely,
every such $F$ is easily seen to satisfy~(i) and~(ii) above, which
completes the case $p = 2$.

Otherwise, note first that $F$ extends continuously to
$[ 0, \infty )$. Let $d = \min\{ p, 4 \}$ and suppose
$A = ( a_{ij} )_{i, j = 1}^d$ is $\TN$. Fix $\bx \in \inc{X}{d}$ and  
$\by \in \inc{Y}{d}$, and let $\eps > 0$ be such that
$\min\{ x_{i + 1} - x_i, \ y_{i + 1} - y_i : i \in [d - 1] \} > %
2 \eps$, 
where $X$ and $Y$ are identified with subsets of $\R$.
Define
\[
K : \R \times \R \to \R; \ ( x, y ) \mapsto %
\begin{cases}
 a_{i j} & \textrm{if } | x - x_i | < \eps \text{ and } %
| y - y_j | < \eps \quad (i, j \in [d]), \\
 0 & \textrm{otherwise}.
\end{cases}
\]
Then $K$ is $\TN_p$ and, by Theorem~\ref{Twhitney2}, there exists a
sequence of $\TP_p$ kernels $( K_l )_{l \geq 1}$ converging to $K$
at $( x_i, y_j )$ for all $i$, $j \in [d]$. Hence $F \circ K_l$ is
$\TP_p$ for all $l \geq 1$ and therefore $F[ A ]$ is $\TN$. Since $A$
was arbitrary, it follows that $F$ preserves $\TN$ entrywise on
$d \times d$ matrices. By Theorem~\ref{Tfixeddim}, we see that
$F$ has the form claimed. The converse follows from
Theorem~\ref{Ttp}(3c) and (3d).
\end{proof}

The classification problems for preservers of $\TP$ kernels on $X
\times Y$ is still to be resolved in the case when $X$ and $Y$ are
both infinite, and the same is true when $X = Y$ and the kernels are
required to be symmetric. As a first step in this direction, we show
next that any such preserver must be a power function.

\begin{proposition}\label{Ppowers}
Suppose $X$ and $Y$ are totally ordered sets, each of size
at least $2$ and possibly infinite. If there exists a $\TP$ kernel on
$X \times Y$, and
$F : ( 0, \infty ) \to ( 0, \infty )$ preserves all such kernels,
or all $\TP_2$ kernels on $X \times Y$, then
$F( x ) = c x^\alpha$ for some $c > 0$ and $\alpha > 0$. The same
holds if $X = Y$ and the kernels are taken to be symmetric.
\end{proposition}

\begin{proof}
By Theorem~\ref{Ttpcompletion}, any $\TP$ $2 \times 2$ matrix $A$ has
the form $\lambda^{-1} ( \exp( \alpha_i \beta_j ) )_{i, j = 1}^2$,
where $\lambda > 0$, $\alpha_1 < \alpha_2$, and $\beta_1 < \beta_2$.
Fix $x_1 < x_2$ in $X$ and $y_1 < y_2$ in $Y$, considered as subsets
of $\R$, and let $\alpha : X \to \R$ and $\beta : Y \to \R$ be
order-preserving bijections such that $\alpha( x_i ) = \alpha_i$ and
$\beta( y_j ) = \beta_j$ for $i$, $j = 1$, $2$. Then the kernel
\[
K : X \times Y \to \R; \ ( x, y ) \mapsto %
\lambda^{-1} \exp( \alpha( x ) \beta( y ) )
\]
is totally positive and contains $A$ as a submatrix; since $F \circ K$
is totally positive and $A$ is arbitrary, it follows from
Theorem~\ref{Ttp} that $F$ has the form claimed.

If $X = Y$ then we may arrange that $\alpha = \beta$, in which case
$K$ is symmetric. Thus, $F$ has the same form as before.
\end{proof}

The full resolution of this classification question when $X$ and $Y$
are infinite is provided in Section~\ref{Sqns}. For now, we apply
Proposition~\ref{Ppowers} to classify the $\TP$ preservers on
Vandermonde matrices.

\begin{corollary}
The functions which preserve the $\TP$ property of the scaled
Vandermonde kernels
\[
\R \times \R \to \R; \ ( x, y ) \mapsto \mu \exp( x y ) %
\qquad ( \mu  > 0 )
\]
are precisely the power functions $F( x ) = c x^\alpha$, where $c > 0$
and $\alpha > 0$. The same holds if ``$\TP$'' is replaced by
``$\TP_2$''.
\end{corollary}

\subsection{Preservers of symmetric $\TP$ matrices}

As in the totally non-negative case, Theorems~\ref{Ttp}
and~\ref{Ttpkernel} have analogues for symmetric matrices and
kernels. The following result should be compared with
Theorem~\ref{Tsymmetric}.

\begin{theorem}\label{TsymmetricTP}
Let $F: ( 0, \infty ) \to \R$ and let $d$ be a positive integer.
The following are equivalent.
\begin{enumerate}
\item $F$ preserves total positivity entrywise on
symmetric $d \times d$ matrices.
\item The function $F$ satisfies
\begin{enumerate}
\item[(a)] $(d = 1)$ $F( x ) > 0$;
\item[(b)] $(d = 2)$ $F$ is positive, increasing,
and multiplicatively mid-convex, that is,
$F( \sqrt{x y} )^2 \leq F( x ) F( y )$ for all $x$,
$y \in ( 0, \infty )$, so continuous;
\item[(c)] $(d = 3)$ $F( x ) = c x^\alpha$ for some $c > 0$ and some
$\alpha \geq 1$;
\item[(d)] $(d = 4)$ $F( x ) = c x^\alpha$ for some $c > 0$ and some
$\alpha \in \{1\} \cup [ 2, \infty )$.
\item[(e)] $(d \geq 5)$ $F( x ) = c x$ for some $c > 0$.
\end{enumerate}
\end{enumerate} 
\end{theorem}

We now outline our proof strategy. Akin to Theorem~\ref{Ttp}, the idea
is to derive the continuity of $F$ from the $2 \times 2$ case, without
the use of multiplicative mid-convexity, and then use the density of
symmetric $\TP$ matrices in symmetric $\TN$ matrices. For the first
step, we require the solution of a symmetric totally positive
completion problem. The following result is analogous to
Theorem~\ref{Ttpcompletion}.

\begin{theorem}\label{L2symtp}
Let the real $2 \times 2$ matrix $A$ be symmetric. The following are
equivalent.
\begin{enumerate}
\item For any $d \geq 2$ and any pair $\{ k_1 < k_2 \} \subseteq [d]$,
there exists a $\TP$ Hankel $d \times d$ matrix $\widetilde{A}$
such that $\widetilde{A}_{k_p, k_q} = A_{pq}$ for $p$, $q = 1$, $2$.

\item Given a totally ordered set $X$, such that
$X \times X$ admits a $\TP$ kernel, and a pair
$\{ x_1 < x_2 \} \subseteq X$, there exists a $\TP$ continuous Hankel
kernel $K$ on $X \times X$ such that
$K[ ( x_1, x_2 ); ( x_1, x_2 ) ] = A$.

\item The matrix $A$ is $\TP$.
\end{enumerate}
\end{theorem}

\begin{proof}
Clearly $(1)$ and $(2)$ each imply $(3)$, and $(1)$ is a special case
of $(2)$. We will show that $(3)$ implies $(2)$.

The general form of a $\TP$ symmetric $2 \times 2$ matrix is
\[
A = \begin{pmatrix} a & b \\ b & c \end{pmatrix} \qquad %
( a, b, c > 0, \ a c > b^2 ).
\]
If $\alpha = \frac12 \log( a c / b^2 )$ and
$\beta = \frac12 \log( b^4 / a^3 c )$, then the continuous Hankel
kernel
\[
K : \R \times \R \to \R; \ ( x, y ) \mapsto %
a \exp( \alpha( x + y )^2 + \beta( x + y ) ),
\]
which is $\TP$ by Examples~\ref{EgenVDM} and~\ref{EgenVDM2}, contains
$A$ as the submatrix $K[ ( 0, 1 ); ( 0, 1 ) ]$. Working as in the
final paragraph of the proof of Theorem~\ref{Ttpcompletion} gives the
general result.
\end{proof}

With Theorem~\ref{L2symtp} at hand, we can classify the preservers of
total positivity on the set of symmetric matrices.

\begin{proof}[Proof of Theorem~\ref{TsymmetricTP}]
The result is clear when $d = 1$, so we assume $d \geq 2$
henceforth. First, suppose $(1)$ holds. By Theorem~\ref{L2symtp},
$F[ - ]$ must preserve total positivity for symmetric
$2 \times 2$ matrices. As shown in the proof of Theorem~\ref{Ttp},
it follows that $F$ is positive and increasing on
$( 0, \infty )$. In particular, $F$ has countably many
discontinuities, and each of these is a jump. Let
$F^+( x ) := \lim_{y \to x^+} F( y )$ for all $x > 0$. Then
$F^+$ is increasing, coincides with $F$ at every point where $F$ is
right continuous, and has the same jump as $F$ at every point
where $F$ is not right continuous. Applying $F[ - ]$ to the
totally positive matrices
\[
M( x, y, \eps ) := \begin{pmatrix}
x + \eps & \sqrt{x y} + \eps \\
\sqrt{x y} + \eps & y + \eps
\end{pmatrix} \qquad ( x, y, \eps > 0, \ x \neq y ),
\]
it follows that
\[
F( \sqrt{x y} + \eps)^2 < F( x + \eps ) F( y + \eps ).
\]
Letting $\eps \to 0^+$, we conclude that
\[
F^+( \sqrt{x y} )^2 \leq F^+( x ) F^+( y )
\qquad \text{for all } x, y > 0;
\]
this inequality holds trivially when $x = y$. Thus $F^+$ is
multiplicatively mid-convex on~$( 0, \infty )$. As in the proof
of Theorem~\ref{Tfixeddim}, it follows by
\cite[Theorem~71.C]{roberts-varberg} that $F^+$ is continuous. We
conclude that $F$ has no jumps and is therefore also continuous.

For $d = 2$, this completes the proof that $(1) \implies (2)$.
If, instead, $d \geq 3$, note that $F$ extends to a continuous
function $\tilde{F}$ on $[ 0, \infty )$. As observed in
\cite[Theorem~2.6]{FJS}, the set of symmetric totally positive
$r \times r$ matrices is dense in the set of symmetric totally
non-negative  $r \times r$ matrices. By continuity, it follows that
$\tilde{F}$ preserves total non-negativity entrywise, and (2)
now follows immediately from Theorem~\ref{Tsymmetric}.

Conversely, suppose (2) holds for $d = 2$, and consider the
totally positive matrix
\[
A = \begin{pmatrix}
a & b \\ b & c
\end{pmatrix} \qquad ( a, b, c > 0, \ a c - b^2 > 0 ).
\]
Since $F$ is increasing, we have $F( \sqrt{b^2} ) < F( \sqrt{a c} )$.
Using the multiplicative convexity of $F$, we conclude that
\[
F(b)^2 = F( \sqrt{b^2} )^2 < F( \sqrt{a c} )^2 \leq F( a ) F( c ).
\] 
Thus $F[ A ]$ is totally positive and (1) holds. The implications for
$d = 3$ and $d = 4$ follow from
\cite[Theorem~5.2 and Proposition~5.6]{FJS}, respectively, and
the case of $d = 5$ is clear.
\end{proof}

\subsection{Preservers of symmetric $\TP$ kernels}

The following result is the immediate reformulation of
Theorem~\ref{TsymmetricTP} to the setting of matricial kernels.
 
\begin{corollary}\label{CtpkernelSym}
Let $X$ be a totally ordered set of size $d \in \N$, and let
$F: ( 0, \infty) \to \R$. Then $F \circ K$ is totally positive for any
symmetric totally positive kernel $K: X \times X \to \R$ if and only
if
\begin{enumerate}
\item[(a)] $(d = 1)$ $F( x ) > 0$;
\item[(b)] $(d = 2)$ $F$ is positive, increasing,
and multiplicatively mid-convex, that is,
$F( \sqrt{x y} )^2 \leq F( x ) F( y )$ for all
$x$, $y \in ( 0,\infty )$, so continuous;
\item[(c)] $(d = 3)$ $F( x ) = c x^\alpha$ for some $c > 0$ and some
$\alpha \geq 1$;
\item[(d)] $(d = 4)$ $F( x ) = c x^\alpha$ for some $c > 0$ and some
$\alpha \in \{1\} \cup [ 2, \infty )$;
\item[(e)] $(d \geq 5)$ $F( x ) = c x$ for some $c > 0$. 
\end{enumerate}
\end{corollary}

Next, we formulate and prove a parallel result, in the spirit of the
final assertion in Theorem~\ref{Ttpkernel}.

\begin{theorem}\label{Tsymm2}
Let $F: ( 0, \infty ) \to \R$ and let $p$ and $d$ be positive
integers, with $p < d$. The following are equivalent.
\begin{enumerate}
\item $F$ preserves $\TP_p$ entrywise on symmetric $d \times d$ matrices.
\item $F$ preserves $\TP_p$ entrywise on $d \times d$ matrices.
\item $F$ preserves $\TP$ entrywise on $p \times p$ matrices.
\item The function $F$ satisfies
\begin{enumerate}
\item[(a)] $(p = 1)$ $F( x ) > 0$;
\item[(b)] $(p = 2)$ $F( x ) = c x^\alpha$ for some $c > 0$ and some
$\alpha > 0$;
\item[(c)] $(p = 3)$ $F( x ) = c x^\alpha$ for some $c > 0$ and some
$\alpha \geq 1$;
\item[(d)] $(p \geq 4)$ $F( x ) = c x$ for some $c > 0$.
\end{enumerate}
\end{enumerate}
The same functions characterize the preservers of symmetric $\TP_p$
kernels on $X \times X$, where $X$ is a totally ordered set of size at
least $p + 1$ (and possibly infinite).
\end{theorem}

\begin{proof}
By Theorem~\ref{Ttpkernel}, the statements (2), (3) and (4) are
equivalent, and clearly $(2) \implies (1)$. Now suppose~(1) holds; we
show~(4) in several steps. If $p = 1$, then~(4) is immediate, so we
suppose henceforth that $p \geq 2$. Now, by Lemma~\ref{L2symtp}, every
symmetric $2 \times 2$ $\TP$ matrix admits an extension to a symmetric
$d \times d$ $\TP$ matrix. Repeating the proof of
Theorem~\ref{TsymmetricTP}, it follows that $F$ is continuous. Hence
$F$ admits a continuous extension $\widetilde{F}$ to $[ 0, \infty )$,
so, by the symmetric version of Whitney's Theorem~\ref{Twhitney}, it
follows that $\widetilde{F}$ preserves the class of symmetric $\TN_p$
$d \times d$ matrices. The claim now follows from Theorem~\ref{Tsymm1}
and the fact that $F$ cannot be constant.

For the final assertion involving kernels, first note that if $F$ is as
in~(4), then (2) holds, and so $F$ preserves symmetric $\TP_p$ kernels
on $X \times X$. Conversely, suppose $F$ preserves the symmetric
$\TP_p$ kernels on $X \times X$. Then the arguments using in the
proof of Theorem~\ref{Ttpkernel} show that $(1) \implies (4)$, with
Theorem~\ref{TsymmetricTP} giving continuity of $F$ when $p \geq 2$
and Theorem~\ref{Tsymm1} used in place of Theorem~\ref{Tfixeddim}.
\end{proof}

\section{Total-positivity preservers are continuous}\label{Svasudeva}

As the vigilant reader will have noticed, we have shown above two similar
assertions, that an entrywise map $F[-]$ preserves total non-negativity
on the set of $2 \times 2$ symmetric matrices if and only if $F$ is
non-negative, non-decreasing, and multiplicatively mid-convex, with the
corresponding changes if weak inequalities are replaced by strict ones.
The variation lies in reducing the set of test matrices with which to
work, while arriving at very similar conclusions.

Such a result was proved originally by Vasudeva \cite{vasudeva79},
when classifying the entrywise preservers of positive semidefiniteness
for $2 \times 2$ matrices with positive entries. To date, this remains
the only known classification of positivity preservers in a fixed
dimension greater than $1$.

It is natural to seek a common strengthening of the results above, as
well as of Vasudeva's result. We conclude by recording for
completeness such a characterization, which uses a small test set of
totally positive $2 \times 2$ matrices.

\begin{notation}
Let $\cP$ denote the set of symmetric totally non-negative
$2 \times 2$ matrices with positive entries, and let the subsets
\begin{align*}
\cP' & := \left\{ A( a, b ) := \begin{pmatrix}
a & b \\ b & a \end{pmatrix} : %
a > b > 0, \text{ $a$ and $b$ not both irrational } \right\} \\
\text{and} \quad \cP'' & := \left\{
B( a, b, c ) := \begin{pmatrix}
a & b \\ b & c \end{pmatrix} : %
a, b, c > 0, \ a c > b^2, \text{ $a$, $b$ and $c$ rational} \right\}.
\end{align*}
\end{notation}

\begin{theorem}\label{Tvasu}
Let $F : ( 0, \infty ) \to \R$ be a function. The following are
equivalent.
\begin{enumerate}
\item The map $F[-]$ preserves total non-negativity on the set $\cP$.

\item The function $F$ is non-negative, non-decreasing, and
multiplicatively mid-convex on $( 0, \infty )$.

\item The map $F[-]$ preserves positive semidefiniteness on the set
$\cP' \cup \cP''$.
\end{enumerate}
Moreover, every such function is continuous, and is either nowhere
zero or identically zero.
\end{theorem}

The sets $\cP'$ and $\cP''$ are in bijection with the sets
$( \R \times \Q ) \cup ( \Q \times \R )$ and $\Q$, respectively, whereas
$\cP$ is a three-parameter family. The equivalence of (1) and (2) is
Vasudeva's result.

\begin{proof}
To see that $(2) \implies (1)$, note that if
$A = \begin{pmatrix} a & b \\ b & c \end{pmatrix} \in \cP$,
then $a$, $b$, $c > 0$ and $0 < b \leq \sqrt{a c}$. By (2), the matrix
$F[ A ]$ has non-negative entries and
\[
0 \leq F( b )^2 \leq F( \sqrt{ac} )^2 \leq F( a ) F( c ),
\]
so $F[A]$ is totally non-negative. Clearly, $(1) \implies (3)$.
The main challenge in the proof is to show $(3) \implies (2)$. The
first step is to observe that $F$ is non-negative and non-decreasing
on $( 0, \infty )$. Let $y > x > 0$, choose rational $a$ such
that $x < a < y$, and consider the matrices $F[ A( a, x ) ]$ and
$F[ A( y, a ) ]$, which are both positive semidefinite. From this, it
follows that $F( y )$ is non-negative, and
$F( y )^2 \geq F( a )^2 \geq F( x )^2$.

We now show that $F$ is identically zero if it vanishes anywhere.
Suppose $F( x ) = 0$ for some $x > 0$. Then, as $F$ is non-decreasing
and non-negative, $F \equiv 0$ on $( 0, x ]$. Given $y > x > 0$,
choose rational $b$ and $c$ such that $0 < c < x < y < b$. Considering
$F[ B( 1 + ( b^2 / c), b, c ) ]$ and then $F[ A( b, y ) ]$ shows that
$F( b ) = 0$ and then $F( y ) = 0$. It follows that $F \equiv 0$.

Finally, we claim that $F$ is multiplicatively mid-convex and continuous.
Clearly this holds if $F \equiv 0$, so we assume that $F$ is never zero.
We first show that the function
\[
F^+ : ( 0, \infty) \to [ 0, \infty );
\qquad F^+( x ) := \lim_{y \to x^+} F( y )
\]
is multiplicatively mid-convex and continuous. Note that $F^+$
is well defined because $F$ is monotone. Given $x$, $y > 0$,
we choose rational numbers $a_n \in ( x, x + 1 / n)$ and
$c_n \in (y,y+1/n)$ for each positive integer $n$. Since $a_n c_n > x y$,
we may choose rational $b_n \in ( \sqrt{x y}, \sqrt{a_n c_n})$.
The matrix $B( a_n, b_n, c_n) \in \cP''$ for each~$n$, therefore
\[
0 \leq \lim_{n \to \infty} \det F[ B( a_n, b_n, c_n ) ] =
F^+( x ) F^+(y) - F^+( \sqrt{x y} )^2.
\]
Thus $F^+$ is multiplicatively mid-convex on $(0,\infty)$, and
$F^+$ is non-decreasing since $F$ is. Repeating the argument in
the proof of Theorem~\ref{Tfixeddim}, which requires the
function to take positive values, it follows that $F^+$ is
continuous. Hence $F$ and $F^+$ are equal, and this gives the
result.
\end{proof}

\begin{remark}
The analogous version of Theorem~\ref{Tvasu} holds for any bounded
domain, that is, for matrices with entries in $( 0, \rho )$, with
$\rho > 0$. The proof is a minimal modification of that given above,
except for the argument to show that either $F \equiv 0$ or~$F$
vanishes nowhere. For this, see
\cite[Proposition~3.2(2)]{GKR-lowrank}. Also, it is clear that the set
of rational numbers in the definitions of $\cP'$ and $\cP''$ may be
replaced with any countable dense subset of the domain of~$F$.
\end{remark}

\section{Extensions of Whitney's approximation theorem}%
\label{SWhitneyExt}

The present section is devoted to constructive approximation schemes
derived from discrete convolutions with the Gaussian kernel. The proof
of Theorem~\ref{Twhitney2} is obtained as an application.

\subsection{Discretized Gaussian convolution}

\begin{notation}\label{Ngauss}
For all $\kappa > 0$, let
\[
G_\kappa : \R \times \R \to \R; \ %
( x, y ) \mapsto \exp( -\kappa ( x - y )^2 ).
\]
A key observation, going back at least to P\'olya and Schoenberg, is
the total positivity of this kernel. In our terminology, this means
that $G_\kappa$ is $\TP_p$ for all $p \in \N$, and this follows as a
particular case of Example~\ref{EgenVDM2}.
\end{notation}

\begin{proposition}\label{Pconvolve}
Let the kernel $K : A \times B \to \R$ be $\TN_p$, where $A$,
$B \subseteq \R$ and $p \in \N$. Suppose that $\kappa > 0$ and $n$,
$N \in \N$ are greater than or equal to $p$. If $\bz \in \inc{A}{n}$
and $\bw \in \inc{B}{N}$, then
\[
T_{\kappa, \bz, \bw}( K ) : A \times B \to \R; \ %
( x, y ) \mapsto \sum_{j = 1}^n \sum_{k = 1}^N %
G_\kappa( x, z_j ) K( z_j, w_k ) G_\kappa( w_k, y )
\]
is $\TN_p$ and $\TP_{\min\{ p, r \}}$, where $r$ is the rank of
$K[ \bz; \bw ]$, which is the same as the rank
of~$T_{\kappa, \bz, \bw}( K )[ \bz; \bw ]$.
\end{proposition}
\begin{proof}
Let $\bx \in \inc{A}{m}$ and $\by \in \inc{B}{m}$, where
$m \in [p]$. The Cauchy--Binet formula gives that
\begin{equation}\label{EdetCB}
\det T_{\kappa, \bz, \bw}( K )[ \bx; \by ] = %
\sum_{\bj \in \inc{[n]}{m}} \sum_{\bk \in \inc{[N]}{m}} %
\det G_\kappa[ \bx; \bz_\bj ] \det K[ \bz_\bj; \bw_\bk ] %
\det G_\kappa[ \bw_\bk; \by ],
\end{equation}
where
$\bz_\bj := ( z_{j_1}, \ldots, z_{j_m} )$ if
$\bj = ( j_1, \ldots, j_m )$ and similarly for $\bw_\bk$.

Since $K[ \bz; \bw ]$ has rank $r$, it has a non-zero $r \times r$
minor, and so a non-zero minor of all smaller dimensions, but every
strictly larger minor is zero. The result now follows.
\end{proof}

\begin{notation}
Given a vector $\bmu \in \R^m$ and a positive-definite matrix
$V \in \R^{m \times m}$, where~$m$ is a positive integer, the
 multivariate Gaussian probability density
\[
f_{\bmu, V} : \R^m \to [ 0, \infty ); \ \bx \mapsto %
\frac{( \det V )^{1 / 2}}{( 2 \pi )^{m / 2}} %
e^{-\frac{1}{2} ( \bx -\bmu )^T V ( \bx - \bmu )}
\]
has mean $\bmu$ and inverse covariance matrix $V$. Note that
\[
f_{\bmu, V}( \bx ) = c_{\bmu, V} \, g_{\bmu, V}( \bx ) \quad %
\text{for all } \bx \in \R^m,
\]
where
\[
c_{\bmu, V} := %
( 2 \pi )^{-m / 2} ( \det V )^{1 / 2} e^{-\frac{1}{2} \bmu^T V \bmu} %
\qquad \text{and} \qquad %
g_{\bmu, V}( \bx ) := e^{-\frac{1}{2} \bx^T V \bx + \bx^T V \bmu}.
\]
For all $n \in \N$, let the $n \times n$ matrix $Q$ be defined by
setting $Q_1 = 1$ and
\[
Q_{n + 1} = Q_n \oplus 0_{1 \times 1} + 0_{n - 1 \times n - 1} \oplus %
\begin{pmatrix} 1 & -1 \\ -1 & 1 \end{pmatrix},
\]
so that
\[
Q_n = \begin{pmatrix}
 2 & -1 &  0 & \cdots &  \\
-1 &  2 & -1 & &  \\
 0 & -1 &  2 & -1 &  \\
 \vdots & & \ddots & 2 & -1 \\
 & & & -1 & 1 
\end{pmatrix}.
\]
\end{notation}

\begin{lemma}\label{Lgaussian}
Let $\kappa > 0$. If $x_0$, \ldots, $x_m \in \R$, then
\begin{equation}\label{EmvGauss}
\prod_{j = 1}^m G_\kappa( x_{j - 1}, x_j ) = %
e^{-\kappa x_0^2} g_{\bmu, V}( x_1, \ldots, x_m ) = %
( \pi / \kappa )^{m / 2} f_{\bmu, V}( x_1, \ldots, x_m ),
\end{equation}
where $\bmu = x_0 \bJ_{m \times 1}$ and $V = 2 \kappa Q_m$. Moreover,
$\det V = ( 2 \kappa )^m$ and
$e^{-\kappa x_0^2} = e^{-\frac{1}{2} \bmu^T V \bmu}$.
\end{lemma}
\begin{proof}
Let $x_0$, \ldots, $x_{m + 1} \in \R$ be arbitrary. The first identity
holds when $m = 1$, because
\[
G_\kappa( x_0, x_1 ) = %
\exp( -\kappa x_0^2 ) \exp( -\kappa x_1^2 + 2 \kappa x_0 x_1 ) = %
\exp( -\kappa x_0^2 ) g_{x_0, 2 \kappa}( x_1 ).
\]
Now suppose
\[
G_\kappa( x_0, x_1 ) \cdots G_\kappa( x_{m - 1}, x_m ) = %
\exp( -\kappa x_0^2 ) g_{\bmu, V}( x_1, \ldots, x_m )
\]
for some $\bmu = ( \mu_1, \ldots, \mu_m ) \in \R^m$ and
$V \in \R^{m \times m}$. Then
\[
\prod_{j = 1}^{m + 1} G_\kappa( x_{j - 1}, x_j ) = %
\exp( -\kappa x_0^2 ) g_{\bmu, V}( x_1, \ldots, x_m ) %
G_\kappa( x_m, x_{m + 1} )
\]
and, letting $\bx := ( x_1, \ldots, x_{m + 1} )^T$ and $z \in \R$,
\begin{align*}
g_{\bmu, V}( x_1, \ldots, x_m ) & G_\kappa( x_m, x_{m + 1} ) \\
& = \exp( -\frac{1}{2} \bx^T ( V \oplus 0 ) \bx + %
\bx^T ( V \oplus 0 ) ( \bmu \oplus z ) - %
\kappa ( x_m - x_{m + 1} )^2 ) \\
 & = \exp( -\frac{1}{2} \bx^T V' \bx + \bx^T V' \bmu' ),
\end{align*}
where
\[
V' := V \oplus 0_{1 \times 1} + 0_{m - 1 \times m - 1} \oplus %
\begin{pmatrix}
2 \kappa & -2 \kappa \\ -2 \kappa & 2 \kappa
\end{pmatrix} \quad \text{and} \quad \bmu' = \bmu \oplus \mu_m.
\]
By induction, this gives the first identity. For the penultimate
claim, note that adding the last row of $V'$ to the penultimate row
gives the matrix
\[
V \oplus 0_{1 \times 1} + 0_{m - 1 \times m - 1} \oplus %
\begin{pmatrix} 0 & 0 \\ -2 \kappa & 2 \kappa \end{pmatrix},
\]
which has determinant equal to $2 \kappa$ times the determinant
of~$V$. The final identity is immediate, and the second
identity in~(\ref{EmvGauss}) now follows.
\end{proof}

\begin{notation}\label{Ndelta}
Given $z \in A \subseteq \R$ and $w \in B \subseteq \R$, let
\[
\delta_{( z, w )} : A \times B \to \R; \ %
( x, y ) \mapsto \begin{cases}
 1 & \text{if } x = z \text{ and } y = w, \\
 0 & \text{otherwise}.
\end{cases}  
\]
\end{notation}
 
\begin{proposition}\label{Pfconestep}
With the notation and hypotheses of Proposition~\ref{Pconvolve} and
Notation~\ref{Ndelta}, if $r < p$, then the kernel
\[
T_{\kappa, \bz, \bw}%
( T_{\kappa, \bz, \bw}( K ) + e^{-\kappa} \delta_{( z_1, w_1 )} )
\]
is $\TN_p$ and $\TP_{r + 1}$.
\end{proposition}
\begin{proof}
Let
$K' := T_{\kappa, \bz, \bw}( K ) + e^{-\kappa} \delta_{( z_1, w_1 )}$
and fix $m \in [p]$. If $\bj \in \inc{[ n ]}{m}$ and
$\bk \in \inc{[N]}{m}$, then
\[
K'[ \bz_\bj; \bw_\bk ] = %
T_{\kappa, \bz, \bw}( K )[ \bz_\bj; \bw_\bk ] + %
e^{-\kappa} \delta_{( z_1, w_1 )}( z_{j_1}, w_{k_1} ) E_{1 1},
\]
where $E_{1 1}$ is the $m \times m$ matrix with $( 1, 1 )$ entry equal
to $1$ and $0$ elsewhere. Hence
\begin{multline*}
\det K'[ \bz_\bj; \bw_\bk ] \\
 = \det T_{\kappa, \bz, \bw}( K )[ \bz_\bj; \bw_\bk ] + %
e^{-\kappa} \delta_{( z_1, w_1 )}( z_{j_1}, w_{k_1} ) %
\det T_{\kappa, \bz, \bw}( K )[ \bz_\bj \setminus \{z_1 \}; %
\bw_\bk \setminus \{ w_1 \} ],
\end{multline*}
where $\det T_{\kappa, \bz, \bw }( K )[ \emptyset, \emptyset ] = 1$. 
This shows that $K'$ is $\TN_p$ and that $K'[ \bz_\bj; \bw_\bj ]$ has
positive determinant if $\bj = ( 1, \ldots, r + 1 )$, since
$T_{\kappa, \bz, \bw}( K )$ is $\TP_r$. Thus $K'[ \bz; \bw ]$ has 
rank at least $r + 1$, but since $K'[ \bz; \bw ]$ is a rank-one
perturbation of $T_{\kappa, \bz, \bw}( K )[ \bz; \bw ]$, its rank is
exactly $r + 1$. The result now follows from Proposition~\ref{Pconvolve}.
\end{proof}

\begin{corollary}\label{Cfullrank}
With the notation and hypotheses of Proposition~\ref{Pconvolve} and
Notation~\ref{Ndelta}, if $m = \max\{ 0, p - r \} + 1$, then the
kernel $K^{(m)}_{\kappa, \bz, \bw}$ is $\TP_p$, where
\begin{align*}
K^{(1)}_{\kappa, \bz, \bw} & := T_{\kappa, \bz, \bw}( K ) \\
\text{and} \quad K^{(m)}_{\kappa, \bz, \bw} & := %
T^m_{\kappa, \bz, \bw}( K ) + %
e^{-\kappa} \sum_{j = 1}^{m - 1} %
T^j_{\kappa, \bz, \bw}( \delta_{( z_1, w_1 )} ) \qquad ( m \geq 2 ).
\end{align*}
\end{corollary}

\subsection{Finite--continuum kernels}

As a pr\'elude to our main result, we establish the following.
Recall that the \emph{set of continuity} for a function is the
set of points in its domain where it is continuous.

\begin{theorem}\label{TAlex1}
Let $d$, $p \in \N$, and suppose $K : [d] \times \R \to \R$ is bounded
and $\TN_p$. Then there exists a sequence of $\TP_p$ kernels
$( K_l )_{l \geq 1}$ converging to $K$ locally uniformly on its
set of continuity.
\end{theorem}

This theorem is an immediate consequence of the next result on
discrete Gaussian convolution. For any $d$, $n \in \N$, let
$\bd := ( 1, \ldots, d )$ and
\[
\bz_n := ( -n, -n + 2^{-n}, \ldots, n ) \in \inc{[ -n, n ]}{N}, %
\qquad \text{where } N = n 2^{n + 1} + 1.
\]

\begin{proposition}\label{Pfcconv}
Let $K : [d] \times \R \to \R$ be bounded and let $K^{(m)}$ be as in
Corollary~\ref{Cfullrank}. Then
\[
2^{-m n} ( n / \pi )^{m / 2} K^{(m)}_{n, \bd, \bz_n}  \to K %
\qquad \text{as } n \to \infty,
\]
locally uniformly on the set of continuity for $K$.
\end{proposition}

\begin{proof}
Note first the elementary estimate
\[
\| T_{\kappa, \bz, \bw}( G ) \|_\infty \leq d N \| G \|_\infty
\]
for any kernel $G : A \times B \to \R$, where $\kappa > 0$,
$\bz \in \inc{A}{d}$, $\bw \in \inc{B}{N}$, and $\| \cdot \|_\infty$
is the supremum norm on $A \times B$. Thus, if $j \in [ m ]$, then
\[
\| T_{n, \bd, \bz_n}^j( \delta_{( 1, z_1 )} ) \|_\infty \leq %
d^j N^j \leq d^m ( 4 n )^m 2^{m n} = ( 4 d )^m n^m 2^{m n},
\]
hence
\begin{align*}
\| 2^{-m n} ( n / \pi )^{m / 2} e^{-n} \sum_{j = 1}^{m - 1} %
T_{n, \bd, \bz_n}^j( \delta_{( 1, z_1 )} ) \|_\infty & \leq %
 2^{-m n} ( n / \pi )^{m / 2} e^{-n} m ( 4 d )^m n^m 2^{m n} \\
 & = m ( 16 d^2 / \pi )^{m / 2} n^{3 m / 2} e^{-n} \\
 & \to 0 \qquad \text{as } n \to \infty.
\end{align*}
Next, let $( j_0, x_0 ) \in [d] \times \R$. Lemma~\ref{Lgaussian}
gives that
\[
2^{-m n} ( n / \pi )^{m / 2} T_{n, \bd, \bz_n}^m( K ) = %
\sum_{j_1, \ldots, j_m \in [d]} I_{j_1, \ldots, j_m},
\]
where
\[
I_{j_1, \ldots, j_m}( j_0, x_0 ) := 
\exp( -n \sum_{k = 1}^m ( j_{k - 1} - j_k )^2 ) %
\int_{[ -n, n ]^m} K( j_m, \varphi_n( x_m ) ) %
f_{x_0 \bJ, V}( \varphi_n( \bx ) ) \intd\bx,
\]
with $V = 2 n Q_m$ and
$\varphi_n( z ) := \lfloor 2^n z \rfloor 2^{-n}$ if $z \in \R$ and
$\varphi_n( \bz ) := ( \varphi_n( z_1 ), \ldots, \varphi_n( z_m ) )$
if $\bz \in \R^m$. Now let
$\eps_n( \bz ) := \bz - \varphi_n( \bz ) \in [ 0, 2^{-n} ]^m$ and note
that
\[
( \varphi_n( \bz ) - \bmu )^T V ( \varphi_n( \bz ) - \bmu ) = %
\bz^T V \bz - 2 \eps_n( \bz )^T V ( \bz - \bmu ) + %
\eps_n( \bz )^T V \eps_n( \bz ),
\]
so
\[
f_{\bmu, V}( \varphi_n( \bz ) ) = %
f_{\bmu, V}( \bz ) \exp( R_n( \bz; \bmu ) ),
\]
where
\begin{equation}\label{ERn}
R_n( \bz; \bmu ) := \eps_n( \bz )^T V ( \bz - \bmu ) - %
\frac{1}{2} \eps_n( \bz )^T V \eps_n( \bz ).
\end{equation}
Note that, if $\bz \in [ -n, n ]^m$, then
\[
| R_n( \bz; x_0 \bJ ) | \leq C_n( x_0 ):= %
m 2^{-n} ( n + | x_0 | + 2^{-n - 1} ) %
( 6 m - 5 )^{1 / 2} ( 2 n )^m \to 0
\]
as $n \to \infty$. Hence, if
$( j_1, \ldots, j_m ) \neq ( j_0, \ldots, j_0 )$, then
\[
| I_{j_1, \ldots, j_m} | \leq e^{-n} \| K \|_\infty %
\exp( C_n( x_0 ) ) \to 0 \qquad \text{as } n \to \infty,
\]
locally uniformly in $x_0$ on $\R$. Furthermore, if the random
variable $X$ has the probability density function $f_{x_0 \bJ, V}$, then
$X \to x_0 \bJ$ in distribution as $n \to \infty$, by L\'evy's
continuity theorem \cite[p.~383]{Bi} and the fact that $V^{-1} = ( 2 n
)^{-1} Q_m^{-1} \to 0$ as $n \to \infty$. Thus $X \to x_0 \bJ$ in
probability, and so
\begin{multline*}
| I_{j_0, \ldots, j_0} - %
\int_{\R^m} K( j_0, \varphi_n( x_m ) ) %
f_{x_0 \bJ, V}( \bx ) \intd\bx | \\
 \leq \int_{[ -n, n ]^m} \| K \|_\infty f_{x_0 \bJ, V}( \bx ) ( %
\exp( C_n( x_0 ) ) - 1 ) \intd\bx + %
\| K \|_\infty \Pr( \| X \|_\infty > n ) \to 0
\end{multline*}
as $n \to \infty$, locally uniformly in $x_0$ on $\R$: if
$n > | x_0 | + 1 / 2$ then
\begin{align*}
\Pr( \| X \|_\infty > n ) & \leq %
\Pr( \| X \|_\infty > | x_0 | + 1 / 2 ) \\
 & \leq \Pr( \bigl| \| X \|_\infty - | x_0 | \bigr| > 1 / 2 ) \\
 & \leq \Pr( \| X - x_0 \bJ \|_2 > 1 / 2 ),
\end{align*}
since
\[
\| X - x_0 \bJ \|_2 \geq | X_j - x_0 | \geq %
\bigl| | X_j | - | x_0 | \bigr| \qquad \text{for any } j \in [ m ].
\]
Finally, if $\eps > 0$, then 
\begin{align*}
| & \int_{\R^m} K( j_0, \varphi_n( x_m ) ) %
f_{x_0 \bJ, V}( \bx ) \intd\bx - K( j_0, x_0 ) | \\
 & \leq %
\int_{[ x_0 - \eps, x_0 + \eps ]^m} | K( j_0, \varphi_n( x_m ) - %
K( j_0,  x_0 ) | f_{x_0 \bJ, V}( \bx ) \intd\bx + %
2 \| K \|_\infty \Pr( \| X - x_0 \bJ \|_\infty > \eps ) \\
& \leq \sup\{ | K( j_0, \varphi_n( x ) ) - K( j_0, x_0 ) | : %
| x - x_0 | \leq \eps \} + %
2 \| K \|_\infty \Pr( \| X - x_0 \bJ \|_2 > \eps ).
\end{align*}
This gives the result.
\end{proof}

\subsection{Continuum--continuum kernels}

Finally, we provide the technical heart of the proof of
Theorem~\ref{Twhitney2}, which is the following modification of
Proposition~\ref{Pfcconv}.

\begin{theorem}\label{PAlex2}
Let $K : \R \times \R \to \R$ be bounded and let $K^{(m)}$ be as in
Corollary~\ref{Cfullrank}. Then
\[
4^{-m n} ( n / \pi )^m K^{(m)}_{n, \bz_n, \bz_n} \to K \qquad %
\text{as } n \to \infty,
\]
locally uniformly on the set of continuity for $K$.
\end{theorem}

\begin{proof}
Working as in the proof of Proposition~\ref{Pfcconv}, note first that
\begin{align*}
\| 4^{-m n} ( n / \pi )^m e^{-n} \sum_{j = 1}^{m - 1} %
T_{n, \bz_n, \bz_n}^j( \delta_{( z_1, w_1 )} ) \|_\infty & \leq %
4^{-m n} ( n / \pi )^m e^{-n} m ( 4 n )^{2 m} 4^{m n} \\
 & = m ( 16 / \pi )^m n^{3 m} e^{-n} \\
 & \to 0 \qquad \text{as } n \to \infty.
\end{align*}
Next, fix $( x, y ) \in \R^2$. Lemma~\ref{Lgaussian} gives that
\begin{align*}
J_n & := 4^{-m n} ( n / \pi )^m T_{n, \bz_n, \bz_n}^m( K )( x, y ) \\
 & \hphantom{:}= \int_{[ -n, n ]^m} \int_{[ -n, n ]^m} %
K( \varphi_n( x_m ), \varphi_n( y_m ) ) %
f_{x \bJ, V}( \varphi_n( \bx ) ) %
f_{y \bJ, V}( \varphi_n( \by ) ) \intd\bx \intd\by,
\end{align*}
where $V = 2 n Q_m$,
$\varphi_n( z ) := \lfloor 2^n z \rfloor 2^{-n}$ if $z \in \R$ and
$\varphi_n( \bz ) := ( \varphi_n( z_1 ), \ldots, \varphi_n( z_m ) )$
if $\bz \in \R^m$. As above, let
$\eps_n( \bz ) := \bz - \varphi_n( \bz ) \in [ 0, 2^{-n} ]^m$ and note
that
\[
f_{\bmu, V}( \varphi_n( \bz ) ) = %
f_{\bmu, V}( \bz ) \exp( R_n( \bz; \bmu ) ),
\]
where $R_n( \bz; \bmu )$ is as in~(\ref{ERn}). Hence
\begin{align*}
\Bigl| J_n & - \int_{[ -n, n ]^m} \int_{[ -n, n ]^m} %
K( \varphi_n( x_m ), \varphi_n( y_m ) ) f_{x \bJ, V}( \bx ) %
f_{y \bJ, V}( \by ) \intd\bx \intd\by \Bigr| \\
& \leq \| K \|_\infty \int_{[ -n, n ]^m} \int_{[ -n, n ]^m} %
| \exp( R_n( \bx; x \bJ ) + R_n( \by; y \bJ ) ) - 1 | %
f_{x \bJ, V}( \bx ) f_{y \bJ, V}( \by ) \intd\bx \intd\by \\
 & \leq \| K \|_\infty %
( \exp( m 2^{-n} ( 2 n + | x | + | y | + 2^{-n} ) %
( 6 m - 5 )^{1 / 2} ( 2 n )^m ) - 1 ) \\
 & \to 0 \qquad \text{as } n \to \infty,
\end{align*}
locally uniformly in $( x, y )$ on $\R^2$.

Next, let $( X, Y )$ have probability density function
$f_{x \bJ, V} \times f_{y \bJ, V}$ and note that
$( X, Y ) \to ( x \bJ, y \bJ )$ in probability. Thus,
\begin{align*}
\int_{\R^m} \int_{\R^m} & %
( 1 - 1_{[ -n, n ]^m}( \bx ) 1_{[ -n, n ]^m}( \by ) )
K( \varphi_n( x_m ), \varphi_n( y_m ) ) %
f_{x \bJ, V}( \bx ) f_{y \bJ, V}( \by ) \intd\bx \intd\by \\
& \leq \| K \|_\infty ( %
\Pr( \| X \|_\infty > n ) + \Pr( \| Y \|_\infty > n ) ) \\
 & \to 0 \qquad \text{as } n \to \infty,
\end{align*}
locally uniformly in $( x, y )$ on $\R^2$. Finally, fix $\eps > 0$ and
note that
\begin{align*}
\int_{\R^m} & \int_{\R^m} %
| K( \varphi_n( x_m ), \varphi_n( y_m ) ) - K( x, y ) | %
f_{x \bJ, V}( \bx ) f_{y \bJ, V}( \by ) \intd\bx \intd\by \\
 & \leq %
\int_{[ x - \eps, x + \eps ]^m}\int_{[ y - \eps, y + \eps ]^m} %
| K( \varphi_n( x_m ), \varphi_n( y_m ) ) - K( x, y ) | %
f_{x \bJ, V}( \bx ) f_{y \bJ, V}( \by ) \intd\bx \intd\by \\
 & \qquad + %
2 \| K \|_\infty ( \Pr( \| X - x \bJ \|_\infty > \eps ) + %
\Pr( \| Y - y \bJ \|_\infty > \eps ) ) \\
& \leq %
\sup\{ | K( \varphi_n( \xi ), \varphi_n( \eta ) ) - K( x, y ) | : %
| \xi - x | \leq \eps, \ | \eta - y | \leq \eps \} \\
 & \qquad + 2 \| K \|_\infty ( \Pr( \| X - x \bJ \|_2 > \eps ) + %
\Pr( \| Y - y \bJ \|_2 > \eps ) ).
\end{align*}
The result follows.
\end{proof}

\begin{remark}[Symmetric kernels]\label{Rsymm}
If $K : \R \times \R \to \R$ is symmetric, then, since $G_\kappa$
is symmetric as well,
\[
T_{\kappa, \bz, \bw}( K )( x, y ) = %
T_{\kappa, \bw, \bz}( K )( y, x ) \quad \text{for all } %
\kappa > 0, \ \bz \in \inc{X}{n}, \ \bw \in \inc{X}{N} %
\text{ and } x, y \in X,
\]
In particular, the map $T_{\kappa, \bz, \bz}$ preserves symmetry.
Thus, the kernels $K^{(m)}_{n, \bz_n, \bz_n}$ used in
Proposition~\ref{PAlex2} are symmetric if $K$ is.
\end{remark}

We conclude as follows.

\begin{proof}[Proof of Theorem~\ref{Twhitney2}]
We first extend $K$ to a $\TN_p$ kernel $\widetilde{K}$ on
$\R \times \R$ via padding by zeros. Proposition~\ref{PAlex2} now
gives a sequence $( \widetilde{K}_l )_{l \geq 1}$ of $\TP_p$ kernels
on $\R \times \R$ converging locally uniformly on the set of
continuity of $\widetilde{K}$, which contains the points in the
interior of $X \times Y$ where $K$ is continuous. The result now
follows by restricting each of these to $X \times Y$. The symmetric
variant is proved in the same way, noting that if $K$ is symmetric,
then so are~$\widetilde{K}$ and the kernels $\widetilde{K}_l$, by
Remark~\ref{Rsymm}.
\end{proof}

\begin{remark}
Propositions~\ref{Phankel} and~\ref{Ptoeplitz} provide further
$\TP$-density results, for $\TN$ Hankel kernels and P\'olya frequency
functions, respectively.
\end{remark}

\section{Totally non-negative and totally positive Hankel kernels}%
\label{Shankel}

Having explored variations on our original theme, we now return to
classification problems for total non-negativity and total positivity,
now in the presence of additional structure. First, we consider Hankel
matrices and kernels; in the following two sections, we examine the
case of Toeplitz kernels.

\subsection{Totally non-negative Hankel matrices}

As noted in \cite{BGKP-hankel}, the collection of $\TN$ Hankel
matrices constitutes a test set that is closed under addition,
multiplication by non-negative scalars, entrywise products, and
pointwise limits. In particular, this test set, in each fixed
dimension, is a closed convex cone. As the functions $1$ and $x$
preserve total non-negativity when applied entrywise, the same holds
for any absolutely monotonic function
$\sum_{k = 0}^\infty c_k x^k$, where the Maclaurin coefficient
$c_k \geq 0$ for all $k$. It is natural to ask if there are any other
preservers. In \cite{BGKP-hankel}, we show that, up to a possible
discontinuity at the origin, there are no others.

\begin{theorem}[\cite{BGKP-hankel}]\label{Thankel}
Given a function $F : [ 0, \infty ) \to \R$, the following are
equivalent.
\begin{enumerate}
\item Applied entrywise, $F[-]$ preserves $\TN$ for Hankel
matrices of all sizes.
\item Applied entrywise, $F[-]$ preserves positivity for $\TN$
Hankel matrices of all sizes.
\item $F( x ) = \sum_{k = 0}^\infty c_k x^k$ on $( 0, \infty )$ with
$c_k \geq 0$ for all $k$, and $0 \leq F( 0 ) \leq c_0$.
 
\end{enumerate}
\end{theorem}

Theorem~\ref{Thankel} thus completely resolves the problem of
characterising entrywise $\TN$ preservers on the set of Hankel
matrices of all dimensions. 

For the fixed-dimension context, we provide a brief summary of some
recent progress. The following result provides a necessary condition,
analogous to a result of Horn \cite{horn} for positivity preservers.

\begin{theorem}[\cite{BGKP-hankel}]\label{Tnecessary}
Suppose $F : [ 0, \infty ) \to \R$ is such that $F[-]$ preserves $\TN$
on the set of $d \times d$ Hankel matrices. Then $F$ is
$( d - 3 )$-times continuously differentiable, with $F$, $F'$,
\ldots, $F^{( d - 3 )}$ non-negative on $( 0, \infty )$, and
$F^{( d - 3 )}$ is convex and non-decreasing. If, instead, $F$ is
analytic, then the first $d$ non-zero Maclaurin coefficients of $F$
are positive.
\end{theorem}

Theorem~\ref{Tnecessary} implies strong restrictions for the class of
$\TN$ preservers of Hankel matrices. For instance, if one restricts to
power functions $x^\alpha$, the only such preservers in dimension $d$
correspond to $\alpha$ being a non-negative integer or greater than
$d - 2$. The converse, that such functions preserve $\TN$ for
$d \times d$ Hankel matrices, was shown in \cite{FJS}. This is the
same as the set of entrywise powers preserving positivity on
$d \times d$ matrices, as proved by FitzGerald and Horn
\cite{FitzHorn}.

We conclude by noting that there exist power series which preserve
total non-negativity on Hankel matrices of a fixed dimension and do
not have all Maclaurin coefficients non-negative. The question of
which of these coefficients can be negative was settled in \cite{KT}.
Again, the characterization is the same as that for the class of
positivity preservers, and this coincidence is explained by the
following result of Khare and Tao.

Given $k$, $d \in \N$, with $k \leq d$ and a constant
$\rho \in ( 0, \infty ]$, we let $\cP_d^k( [ 0, \rho ) )$ denote the
set of positive-semidefinite $d \times d$ matrices of rank
at most $k$ and with entries in $[ 0, \rho )$.

\begin{theorem}[{\cite[Proposition~9.7]{KT}}]
Suppose $F : [ 0, \rho ) \to \R$ is such that the entrywise map
$F[ - ]$ preserves positivity on $\cP_d^k( [ 0, \rho ) )$, where
$k \leq d$ and $\rho \in ( 0, \infty ]$. Then $F[-]$ preserves total
non-negativity on the set of Hankel matrices in
$\cP_d^k( [ 0, \rho ) )$.
\end{theorem}

\subsection{Hankel totally non-negative and totally positive kernels
on infinite domains}

We now consider the problem of classifying the preservers of $\TN$ and
$\TP$ Hankel kernels on $X \times X$, where $X \subseteq \R$ is
infinite. A Hankel kernel has the form
\[
X \times X \to \R; ( x, y ) \mapsto f( x + y )
\]
for some function $f : X + X \to \R$, and so is automatically
symmetric. Examples of such kernels abound; for example, given
positive scalars $c_1$, \ldots, $c_n$ and $u_1$, \ldots, $u_n$, the
kernel
\[
\R \times \R \to \R; \ ( x, y ) \mapsto \sum_{i = 1}^n c_i u_i^{x + y}
\]
is Hankel and $\TN$ on $\R \times \R$, as we will see below.

If $X$ is an arbitrary subset of $\R$, then minors drawn from $X
\times X$ may not embed in a larger Hankel matrix drawn from
$X \times X$, since the arguments may be linearly independent over
$\Q$. This issue is avoided by assuming that $X$ is an interval and
any kernel under consideration is a continuous functions of its
arguments.

Recall that the \emph{Schur} or \emph{pointwise} product of kernels
$K$ and $K'$ with common domain $X \times X$ is the kernel
\[
K \cdot K' : X \times X \to \R; \ %
( x, y ) \mapsto K( x, y ) K'( x, y ).
\]
We equip the set of kernels on a given domain $X \times X$ with the
topology of pointwise convergence. The following proposition
summarizes some of the important properties of Hankel kernels. In
particular, under appropriate assumptions, the sets of $\TN$ and $\TP$
kernels form convex cones that are closed under taking Schur products.
See \cite{FJS} for analogous results in the matrix case.

\begin{proposition}\label{Phankel}
Suppose $X \subseteq \R$ is an interval.
\begin{enumerate}
\item The space of $\TN$ continuous Hankel kernels on $X \times X$ is a
closed convex cone, which is also closed under Schur products.

\item Suppose $X$ is an open interval and $K : X \times X \to \R$ 
is a continuous Hankel kernel. The following are equivalent.
\begin{enumerate}
\item $K$ is $\TN$.

\item $K$ is positive semidefinite.

\item $K$ is of the form
\begin{equation}\label{EHankel}
X \times X \to \R; \ ( x, y ) \mapsto %
\int_\R e^{ ( x + y ) u} \intd\sigma( u )
\end{equation}
for some non-decreasing function $\sigma$.
\end{enumerate}
Furthermore, $K$ is $\TP$ if and only if the measure
corresponding to $\sigma$ has infinite support.

\item If $X$ is an open interval, then the set of $\TP$ continuous
Hankel kernels on $X \times X$ is dense in the set of $\TN$
continuous Hankel kernels on $X$.

\item If $X$ is an open interval, then the set of $\TP$ continuous
Hankel kernels on $X \times X$ is a convex cone, which is closed under
Schur products.
\end{enumerate}
\end{proposition}

The second part of Proposition~\ref{Phankel} solves a Hamburger-type
inverse problem for exponential moments of non-negative measures on $\R$.
The third provides a extension of Whitney's theorem for Hankel kernels.

The proof of Proposition~\ref{Phankel} uses several preliminary
results. We begin with a well-known 1912 result of Fekete~\cite{Fe}.
Recall that a minor is \emph{contiguous} if it is formed from
consecutive rows and columns.

\begin{proposition}\label{Pfekete}
Suppose $m$, $n \in \N$ and let $A$ be an $m \times n$ matrix such
that all its contiguous minors are positive. Then $A$ is $\TP$.
\end{proposition}

From Proposition~\ref{Pfekete}, we deduce the following corollary,
which will be used below. Given a matrix $A$, we denote by $A^{(1)}$
the matrix obtained from $A$ by deleting its first row and last
column.

\begin{corollary}\label{Cfekete}
A square Hankel matrix $A$ is $\TP$ if and only if $A$ and $A^{(1)}$
are positive definite. A square Hankel matrix $A$ is $\TN$ if and
only if $A$ and $A^{(1)}$ are positive semidefinite.
\end{corollary}

\begin{proof}
The forward implication is immediate in both cases. For the converse,
first suppose $A$ and $A^{(1)}$ are positive definite. Note that any
contiguous minor of $A$ is a principal minor of either $A$ or
$A^{(1)}$, and so is positive, hence the claim follows by
Proposition~\ref{Pfekete}.

Finally, suppose $A$ and $A^{(1)}$ are positive semidefinite. By the
above observation, so is every contiguous square submatrix of $A$. Now
let the matrix $B$ be Hankel, $\TP$ and the same size as $A$;
Example~\ref{EgenVDM2} provides the existence of such. Using the
previous observation again, every contiguous square submatrix of $B$
is positive definite. Hence for all $\eps > 0$, every contiguous minor
of $A + \eps B$ is positive. It follows by Proposition~\ref{Pfekete}
that $A + \eps B$ is $\TP$, whence $A$ is $\TN$, as desired.
\end{proof}

The final preliminary result is as follows.

\begin{lemma}\label{Lhankel}
Let $K : X \times X \to ( 0, \infty )$, where $X \subseteq \R$ is an
interval. Each of the following statements implies the next.
\begin{enumerate}
\item $K$ is $\TN$.

\item All principal submatrices drawn from $K$ are $\TN$.

\item All principal submatrices drawn from $K$ with arguments in
arithmetic progression are $\TN$.

\item All principal submatrices drawn from $K$ with arguments in
arithmetic progression are positive semidefinite.
\end{enumerate}
Conversely, $(2) \implies (1)$ for all $K$, $(3) \implies (2)$ if $K$
is continuous, and $(4) \implies (3)$ if $K$ is continuous and Hankel.
\end{lemma}

\begin{proof}
Clearly $(1) \implies (2) \implies (3) \implies (4)$.

If $(2)$ holds, then, given $\bx$, $\bz \in \inc{X}{n}$ for some
$n \in \N$, the matrix $K[ \bx; \bz ]$ is a submatrix of
$K[ \bx \cup \bz; \bx \cup \bz ]$, where $\bx \cup \bz$ is obtained by
taking the union of $\bx$ and $\bz$ in increasing order. Hence~$(1)$
holds.

Next, suppose~$(3)$ holds and $K$ is continuous.
Let $\bx \in \inc{X}{n}$ for some $n \in \N$; by continuity, we
may assume that each term $x_j \in \bx$ is rational. Choose a
positive integer $N$ such that $N ( x_j - x_1 )$ is an integer for all
$j$, and let
\[
\by := ( x_1, x_1 + N^{-1}, x_1 + 2 N^{-1}, \ldots, x_n ).
\]
By assumption, the matrix $K[ \by; \by ]$ is $\TN$, thus so
is the submatrix $K[ \bx; \bx ]$. This shows that~$(2)$ holds.

Finally, suppose~$(4)$ holds, and let a principal submatrix $A$ be
obtained by evaluating $K$ at an arithmetic progression in $X$, say
$x_1 < \cdots < x_n$. By assumption, $A$ is positive semidefinite;
furthermore, so is the $( n - 1 ) \times ( n - 1 )$ matrix $B$
obtained by evaluating $K$ at the arithmetic progression
\[
\frac{x_1 + x_2}{2} < \frac{x_2 + x_3}{2} < \cdots < %
\frac{x_{n-1} + x_n}{2}.
\]
But $B = A^{(1)}$, so $(3)$ follows by Corollary~\ref{Cfekete}.
\end{proof}

We now have the ingredients we require.

\begin{proof}[Proof of Proposition~\ref{Phankel}]\hfill
Part (1) holds because property (4) of Lemma~\ref{Lhankel} is closed
under addition, dilation, pointwise limits, and Schur products.

For part (2), note first that Lemma~\ref{Lhankel} gives the
equivalence of (a) and (b). That positive semidefiniteness is
necessary and sufficient for $K$ to have the form~(\ref{EHankel}) is a
result of Bernstein~\cite{Bernstein} and
Widder~\cite{widder2} which uses prior works of Hamburger and Mercer;
see also \cite[Theorem~5.5.4]{Akhiezer}.

If the measure $\mu$ corresponding to $\sigma$ has finite support, so
may be written as $\sum_{k = 1}^r c_k \delta_{u_k}$, and $\bx$,
$\by \in \inc{X}{n}$, then the submatrix
\begin{equation}\label{Efinrank}
K[ \bx; \by ] = %
\sum_{k = 1}^r c_k ( e^{( x_i + y_j ) u_k} )_{i, j = 1}^n = %
\sum_{k = 1}^r c_k \bz_k \bw_k^T,
\end{equation}
where $\bz_k := ( e^{x_1 u_k}, \ldots, e^{x_n u_k} )^T$ and
$\bw_k := ( e^{y_1 u_k}, \ldots, e^{y_n u_k} )^T$. Thus,
submatrices of $K$ have rank at most $r$, so $K$ cannot be $\TP$.

Finally, if $\mu$ has infinite support, then the basic composition
formula of P\'{o}lya and Szeg\H{o} \cite[p.17]{Karlin} gives that
\[
\det K[ \bx; \by ] = \int_{\inc{\R}{m}} %
\det( \exp( x_i u_j ) )_{i, j = 1}^n %
\det( \exp( u_j y_k ) )_{j, k = 1}^n %
\intd\sigma( u_1 ) \cdots \intd\sigma( u_m )
\]
for any $\bx$, $\by \in \inc{X}{m}$, and so $K$ is $\TP$. This
observation completes the proof of part (2).

For part (3), note that if $K$ is a $\TN$ continuous Hankel kernel as
in (2), then the continuous Hankel kernel
\[
X \times X \to \R; \ ( x, y ) \mapsto K( x, y ) + %
\eps\int_0^1 e^{( x + y ) u} \intd u
\]
is $\TP$ for all $\eps > 0$, since the measure corresponding to the
representative function $\sigma_\eps$ has infinite support.

For the final part, note first that $\TP$ kernels are closed under
positive rescaling. Furthermore, if the $\TP$ kernels $K'$ and $K''$
have representative functions $\sigma'$ and $\sigma''$, then the
corresponding measures have infinite support, and therefore so does
the measure corresponding to $\sigma' + \sigma''$. It follows
$K' + K''$ is $\TP$.

Finally, to see that $K' \cdot K''$ is $\TP$, note first that it is
$\TN$, so has a representative function $\tau$. We assume the measure
$\nu$ corresponding to $\tau$ has finite support, say of size $r$, and
derive a contradiction. Suppose $\bx \in \inc{X}{r + 1}$ is an
arithmetic progression, and consider the principal submatrices
$M' = K'[ \bx; \bx ]$ and $M'' = K''[ \bx; \bx ]$. Both submatrices
are $\TP$ by assumption, and Hankel by the choice of $\bx$. Hence so
is $M' \circ M''$, by Corollary~\ref{Cfekete} above and the Schur
product theorem, so it must have rank $r+1$. But this contradicts the
fact that $\nu$ has support of size $r$, by~(\ref{Efinrank}) with
$K = K' \cdot K''$.
\end{proof}

Having gained a better understanding of our test set, we proceed to
classify its preservers. As in the case of matrices of all sizes, the
preservers of $\TN$ continuous Hankel kernels are absolutely monotonic
functions.

\begin{theorem}\label{Thankel2}
Suppose $X \subseteq \R$ is an interval containing at least two points
and let $F : [ 0, \infty ) \to \R$. The following are equivalent.
\begin{enumerate}
\item The map $C_F$ preserves $\TN$ for continuous Hankel kernels on
$X \times X$.

\item The map $C_F$ preserves positive semidefiniteness for
$\TN$ continuous Hankel kernels on $X \times X$.

\item $F( x ) = \sum_{k = 0}^\infty c_k x^k$ on $( 0, \infty )$, with
$c_k \geq 0$ for all $k$, and $F( 0 ) \geq 0$.
\end{enumerate}
\end{theorem}

The proof of this theorem uses the following observation about $\TN_2$
Hankel kernels that vanish at a point. Recall that $\partial X$ denotes
the topological boundary of the set $X$; in particular, if
$X \subseteq \R$ is an interval, then $\partial X$ is the set of
endpoints.

\begin{lemma}\label{Lzeros}
Suppose $X \subseteq \R$ is an interval and the kernel
$K : X \times X \to \R$ is Hankel and $\TN_2$. If $K( x, y ) = 0$ for
some point $( x, y ) \in X \times X$, then $K$ vanishes on
$X \times X \setminus \{ ( x_0, x_0 ) : x_0 \in \partial X \}$. In
particular, if $K$ is also continuous, then $K \equiv 0$.
\end{lemma}

\begin{proof}
Suppose $K$ is as in the statement of the lemma, and $X$ has interior
$( a, b )$ where $-\infty \leq a < b \leq \infty$. If $K( x, y ) = 0$,
then, since $K$ is Hankel, $K( d_0, d_0 ) = 0$, where
$d_0 := ( x + y ) / 2$. By the Hankel property of $K$, it suffices
to show $K( d, d ) = 0$ for all $d \in X \setminus \partial X$. Now
let $c \in ( a, d_0 )$; the positivity of
$K[ ( c, d_0 ); ( c, d_0 ) ]$ gives that
\[
0 \leq K( c, d_0 )^2 \leq K( c, c ) K( d_0, d_0 ) = 0,
\]
so
$K( c, d_0 ) = 0 = K( ( c + d_0 ) / 2, ( c + d_0 ) / 2 )$.

If $a = -\infty$, then this shows that $K( d, d ) = 0$ for all
$d \in ( a, d_0 )$. If, instead, $a > -\infty$, then this shows
that $K( d, d ) = 0$ for all $d \in ( ( a + d_0 ) / 2 , d_0 )$.

We proceed inductively, assuming that $d_0 > a$ (otherwise there is
nothing to prove). Let
\[
d_n := ( a + 3 d_{n - 1} ) / 4 %
\in ( ( a + d_{n - 1} ) / 2 , d_{n - 1} ) \qquad ( n \in \N )
\]
and note that $K( d_n, d_n ) = 0$, so the previous working shows that
$K( d, d ) = 0$ for all $d \in ( ( a + d_n ) / 2, d_0 )$. Since
$d_n  \to a$ as $n \to \infty$, we see that $K( d, d ) = 0$ whenever
$d \in ( a, d_0 )$.

A similar argument shows that $K( d, d )$ vanishes if
$d \in ( d_0, b )$. The extended result when $K$ is continuous is
immediate.
\end{proof}

\begin{proof}[Proof of Theorem~\ref{Thankel2}]
That $(1) \implies (2)$ is immediate. Next, we assume~$(3)$ and
show~$(1)$, so suppose the continuous Hankel kernel $K : X \times X
\to \R$ is $\TN$. If $K$ is never zero on $X \times X$, then
$F \circ K$ is again $\TN$, continuous, and Hankel, by
Proposition~\ref{Phankel}(1). Otherwise $K$ vanishes at a point,
so Lemma~\ref{Lzeros} applies and $K \equiv 0$, but then
$F \circ \bZ_{X \times X} = F( 0 ) \bJ_{X \times X}$ is indeed
$\TN$, continuous, and Hankel.

Finally, to show $(2) \implies (3)$, we appeal to the following
result.

\begin{theorem}%
[{\cite[Theorem~4.2 and Remark~4.3]{BGKP-hankel}}]\label{Tbgkp}
Fix $u_0 \in ( 0, 1 )$ and suppose the function
$F : ( 0, \infty ) \to \R$ is such that $F[-]$ preserves
positive semidefiniteness for $2 \times 2$ matrices of the form
\[
\begin{pmatrix} a & b \\ b & b \end{pmatrix} %
\quad \text{and} \quad %
\begin{pmatrix} c^2 &  c d \\ c d & d^2 \end{pmatrix} \qquad %
( a, b, c, d > 0, \ a > b )
\]
as well as for the matrices
$( p + q u_0^{i+j} )_{i, j = 0}^n$ for all $p$, $q \geq 0$ with
$p + q > 0$ and all $n \in \N$. Then $F$ is smooth and
$F^{(k)} \geq 0$ on $( 0, \infty )$ for all $k \geq 0$.
\end{theorem}

A function $F$ satisfying the hypotheses of this theorem is therefore
absolutely monotonic on $( 0, \infty )$, and so has a power-series
representation there with non-negative Maclaurin coefficients.

Now suppose~$(2)$ holds. When $K = x \bJ_{X \times X}$, with
$x \geq 0$, then $F \circ K$ being $\TN$ implies $F( x ) \geq 0$. To
apply Theorem~\ref{Tbgkp}, fix $n \in \N$ and choose points $x_0$,
$x_n \in X$ with $x_0 < x_n$. Let $g : \R \to \R$ be the linear
function such that $g( x_0 ) = 0$ and $g( x_n ) = n$, and let
$x_i = g^{-1}( i )$ for $i = 1$, \ldots, $n - 1$. Let $p$, $q \geq 0$
be such that $p + q > 0$. By assumption, the map $C_F$ preserves
positive semidefiniteness on the $\TN$ continuous Hankel kernel
\[
K : X \times X \to \R; \ ( x, y ) \mapsto p + q u_0^{g( x ) + g( y )},
\]
which contains $( p + q u_0^{i + j} )_{i, j = 0}^n$ as the principal
submatrix $K[ ( x_0, \ldots, x_n ); ( x_0, \ldots, x_n ) ]$.
Similarly, given positive $a$, $b$, $c$, and $d$, with $a > b$, the
$\TN$ continuous Hankel kernels
\[
K' : X \times X \to \R; \ ( x, y ) \mapsto %
\frac{( 2 a - b )^2}{4 a - 3 b} %
\left( \frac{b}{2 a - b}\right)^{g( x ) + g( y )} + %
\frac{b ( a - b )}{4 a - 3 b} 2^{g( x ) + g( y )}
\]
and
\[
K'' : X \times X \to \R; \ %
( x, y ) \mapsto c^2 ( d / c )^{g( x ) + g( y )}
\]
have submatrices $K'[ ( x_0, x_1 ); ( x_0, x_1 ) ]$ and
$K''[ ( x_0, x_1 ); ( x_0, x_1 ) ]$ which appear in the statement of
Theorem~\ref{Tbgkp}. Thus $F[-]$ preserves $\TN$ on these matrices, so
the hypotheses of Theorem~\ref{Tbgkp} are satisfied. It follows that
$F$ is as claimed.
\end{proof}

To conclude this part, we classify the preservers of $\TP$ Hankel
kernels.

\begin{theorem}\label{TPhankel}
Suppose $X \subseteq \R$ is an open interval and let
$F : ( 0, \infty ) \to \R$. The following are equivalent. 
\begin{enumerate}
\item The map $C_F$ preserves $\TP$ continuous Hankel kernels on
$X \times X$.

\item The map $C_F$ preserves positive definiteness for $\TP$
continuous Hankel kernels on $X \times X$.

\item $F( x ) = \sum_{k = 0}^\infty c_k x^k$ on $( 0, \infty )$, where
$c_k \geq 0$ for all $k$, and $F$ is non-constant.
\end{enumerate}
\end{theorem}

\begin{proof}
Clearly $(1) \implies (2)$. We now assume~$(3)$ and show~$(1)$.
Suppose $F$ is as specified and let $n_0 \in \N$ be such that
$c_{n_0} > 0$. If $K : X \times X \to \R$ is a $\TP$ continuous Hankel
kernel, then so is $K^{n_0}$, by Proposition~\ref{Phankel}(4). Let
$G( x ) := F( x ) - c_{n_0} x^{n_0}$ and note that
$G \circ K$ is $\TN$, by Theorem~\ref{Thankel2}. Now $K^{n_0}$ and
$G \circ K$ have integral representations as in
Proposition~\ref{Phankel}(2), with corresponding measures $\mu$ and
$\nu$, respectively. Furthermore, the measure $\mu$ has infinite
support, and therefore so does $c_{n_0} \mu + \nu$. Thus $F \circ K$ is
$\TP$.

Finally, suppose~$(2)$ holds. By Theorem~\ref{L2symtp}, any $\TP$
symmetric $2 \times 2$ matrix occurs as a submatrix of a continuous
Hankel $\TP$ kernel on $X \times X$. It follows from
Theorem~\ref{TsymmetricTP} that $F$ is continuous on $( 0, \infty )$.
But then, by the density assertion in Proposition~\ref{Phankel}(3),
the map $C_F$ preserves the set of $\TN$ continuous Hankel kernels on
$X \times X$. It now follows from Theorem~\ref{Thankel2} that $F$ is a
power series with non-negative Maclaurin coefficients, and $F$ cannot
be constant as then it cannot preserve positive definiteness. This
shows~$(3)$.
\end{proof}

\section{P\'olya frequency functions and Toeplitz kernels}%
\label{Stoeplitz}

As the analysis in the previous sections shows, only small test sets
of matrices and kernels are required to assure the rigidity of $\TN$
and $\TP$ endomorphisms, that is, to obtain Theorem~\ref{T1}. In this
section, we explore another classical family of distinguished kernels,
those associated to P\'olya frequency functions. Such kernels are of
central importance for time-frequency analysis and the theory of
splines. The landmark contributions of Schoenberg, starting with his
first full article on the subject~\cite{Schoenberg51}, are highly
recommended to the uninitiated reader. See also the monographs of
Karlin~\cite{Karlin} and Hirschman and Widder~\cite{HW}.

\begin{definition}
A \emph{P\'olya frequency function} is a function
$\Lambda : \R \to [ 0, \infty )$ which is Lebesgue integrable,
non-zero at two or more points and such that the Toeplitz kernel
\[
T_\Lambda : \R \times \R \to \R; \ ( x, y ) \mapsto \Lambda( x - y)
\]
is totally non-negative.
\end{definition}

This is a Toeplitz counterpart of the Hankel kernels encountered in
Section~\ref{Shankel}. Even the condition that the kernel $T_\Lambda$
is $\TN_2$ is very restrictive. Indeed, this, measurability and the
non-vanishing condition imply that
\begin{equation}\label{Econvex}
\Lambda( x ) = \exp( -\phi( x ) ) \qquad ( x \in \R ),
\end{equation}
where the function $\phi$ is convex on an open interval, so continuous
there, with possible discontinuities at the boundary and infinite
values outside: see \cite[Definition~3 and Lemma~1]{Schoenberg51}.
It also implies \cite[Lemma~2]{Schoenberg51} that $\Lambda$ either
decays exponentially at infinity, and so is integrable, or is monotone.

Schoenberg proved \cite[Corollary~2]{Schoenberg51} that the
only discontinuous P\'olya frequency functions are affine transforms
$x \mapsto \lambda( a x + b )$, where $a$, $b \in \R$ and $a \neq 0$,
of the map
\[
\lambda : \R \to [ 0, \infty ); \ x \mapsto \begin{cases}
 0 & ( x < 0 ), \\
 e^{-x} & ( x \geq 0 ),
\end{cases}
\]
except possibly at the origin. In fact, one can alter the function
$\lambda$ to obtain
\begin{equation}\label{Elambda}
\lambda_d : \R \to [ 0, \infty ); \ x \mapsto \begin{cases}
 0 & ( x < 0 ),\\
 d & ( x = 0 ),\\
 e^{-x} & ( x > 0 ),
\end{cases}
\end{equation}
where $d \in [ 0, 1 ]$, without affecting the total non-negativity
property. This may be proved directly in the same way as
\cite[Chapter~IV, Lemma~7.1a]{HW}; see also \cite[p.16]{Karlin}. The
requirement (\ref{Econvex}) ensures that this is the only possible
variation on $\lambda$.

This class of kernels was linked by P\'olya and Schur
\cite{Polya-Schur} to earlier studies pursued by Laguerre and devoted
to coefficient operations which preserve polynomials with purely real
roots. More specifically, convolution with such kernels maps
polynomials to polynomials of the same degree and such a map possesses
a series of striking root-location and root-counting properties. It
should be no surprise, then, that the Fourier--Laplace transform of a
P\'olya frequency function is very special; see, for instance,
\cite{Hamburger1920}. The main theorems of Schoenberg
\cite{Schoenberg51} are the culmination of half a century of
discoveries on this theme. To be precise, the bilateral Laplace
transform $\cB\{ \Lambda \}$ of a P\'olya frequency
function~$\Lambda$, given by
\[
\cB\{ \Lambda \}( z ) := %
\int_\R e^{-x z} \Lambda( x ) \intd x,
\]
is an analytic function in the vertical strip
$\{ z \in \C : \alpha < \Re z < \beta \}$, where the bounds $\alpha$
and $\beta$ can have infinite values and are such that
\[
\alpha = \lim_{x \to \infty} \frac{\log \Lambda( x )}{x} < 0 %
\qquad \text{and} \qquad %
\beta = \lim_{x \to -\infty} \frac{\log \Lambda( x )}{x} > 0;
\]
see \cite[Lemma 10]{Schoenberg51}. The characteristic feature of the
bilateral Laplace transform of a P\'olya frequency function is the
structure of its reciprocal.

\begin{theorem}[{\cite[Theorems 1 and 2]{Schoenberg51}}]\label{Tschoenberg}
If $\Lambda$ is a P\'olya frequency function then
the map $z \mapsto 1 / \cB\{ \Lambda \}( z )$ is, up to an exponential
factor, the restriction of an entire function of genus zero or one,
with purely real zeros.
\end{theorem}

Examples abound, and in general they are related to Hadamard
factorizations of elementary transcendental functions
\cite{Hamburger1920,Schoenberg51,Karlin}. For instance, $e^{-x^2}$,
$e^{-|x|}$, $1 / \cosh x$ and $e^{-x - e^{-x}}$ are all P\'olya
frequency functions.

\subsection{Preservers of P\'olya frequency functions}

In this subsection, we classify all composition transforms
$\Lambda \mapsto F \circ \Lambda$ which leave invariant the class of
P\'olya frequency functions. The Gaussian kernel stands out, as the
sole generator via affine changes of coordinates of a prominent family
of test functions. We start by investigating this particular
situation.

An immediate inspection of such transforms applied to the $\TP$ kernels
\[
c G_\kappa : \R \times \R \to \R; \ ( x, y ) \mapsto %
c \exp( -\kappa ( x - y )^2 ) \qquad ( c > 0, \ \kappa > 0 )
\]
shows that we may expect a larger class of preservers than found in
the rigid conclusions contained in our general theorems. Indeed, all
maps of the form $c_0 x^\alpha$ for positive $c_0$ and $\alpha$
preserve $\TP$ on these kernels. However, more exotic preservers
exist in this setting. As
\[
K_\alpha : \R \times \R \to \R; \ %
( x, y ) \mapsto \exp( -\alpha | x - y | )
\]
is also a P\'{o}lya frequency function for any $\alpha > 0$, it follows
that
\[
F : ( 0, \infty ) \to \R; \ t \mapsto %
\exp( -\sqrt{-\log \max\{ t, 1 \}} )
\]
is an admissible transformer of $G_\kappa$ for all $\kappa > 0$. Such
an analysis can be refined to consider the $\TN_p$ property, but we do
not pursue this path here. For a recent characterization of P\'{o}lya
frequency functions of order at most~$3$, see \cite{Weinberger83}.

As an initial step, we obtain the following proposition.

\begin{proposition}\label{PGauss}
Given a function $F : ( 0, \infty ) \to \R$, each of the following
statements implies the next.
\begin{enumerate}
\item $F( x ) = c_0 x^\alpha$ for some $c_0 > 0$ and $\alpha > 0$.

\item $F \circ c G_\kappa$ is totally positive on $\R \times \R$ for all
$c > 0$ and $\kappa > 0$.

\item $F \circ c G_1$ is $\TP_3$ on $\R \times \R$ for all $c > 0$.

\item $F$ is positive, increasing, and continuous on $( 0, \infty )$.
\end{enumerate}
If $(3')$ $F \circ c G_1$ is $\TN_3$ on $\R \times \R$ for all
$c > 0$, then $(4')$ $F$ is non-negative, non-decreasing, and
continuous on $( 0, \infty )$.
\end{proposition}

\begin{proof}
That $(1) \implies (2)$ and $(2) \implies (3)$ is immediate. We next
assume~$(3)$ and show~$(4)$. Given $p$, $q > 0$ with $p < q$, let
$x := \sqrt{\log( q / p )}$ and $\by := ( 0, x )$. Then the
$2 \times 2$ matrix $F[ q G_1[ \by; \by ] ]$ has positive determinant
and positive entries. This shows that $F$ must be positive and
increasing on $( 0, \infty )$. In particular, the function $F$ has at
most countably many discontinuities.

Now we set $F^\pm( x ) := \lim_{y \to x^\pm} F( y )$ for all $x > 0$.
To complete the proof, we fix $p > 0$ and show that
$F^+( p ) = F( p ) = F^-( p )$. To see this, choose $q > p$ such that
$F$ is continuous at $q$, and let $x := \sqrt{\log( q / p )}$ as
before. Let $\bz := ( 0, y, x )$ and $\bw := ( 0, x, z )$ for $y$, $z
> 0$ such that $y < x < z$, and consider the positive-definite
matrices
\begin{align*}
A_y := F[ q G_1[ \bz; \bz ] ] & = \begin{pmatrix}
F( q ) & F( q e^{-y^2} ) & F( p ) \\
F( q e^{-y^2} ) & F( q ) & F( q e^{-( x - y )^2} ) \\
F( p ) & F( q e^{-( x - y )^2} ) & F( q )
\end{pmatrix} \\
\text{and} \quad %
B_z := F[ q G_1[ \bw; \bw ] ] & = \begin{pmatrix}
F( q ) & F( p ) & F( q e^{-z^2}) \\
F( p ) & F( q ) & F( q e^{-( x - z )^2}) \\
F( q e^{-z^2} ) & F( q e^{-( x - z )^2} ) & F( q )
\end{pmatrix}.
\end{align*}
Note that
\[
\lim_{y \to x^-} F(q e^{-y^2}) = F^-(p), \qquad
\lim_{z \to x^+} F(q e^{-z^2}) = F^+(p),
\]
and $F$ is continuous at $q$. Hence
\[
\lim_{y \to x^-} \det A_y = -F( q ) ( F^-( p ) - F( p ) )^2,
\qquad
\lim_{z \to x^-} \det B_z = -F( q ) ( F^+( p ) - F( p ) )^2.
\]
Since both limits are non-negative, and $F( q ) > 0$ from the previous
working, it follows that $F^+( p ) = F( p ) = F^-( p )$, as required.
Hence $(3) \implies (4)$.

To show $(3') \implies (4')$, we may repeat the argument above,
assuming without loss of generality that $F$ is non-constant, so that
given $p > 0$, we may choose a continuity point $q > p$ with
$F( q ) > 0$.
\end{proof}

The next result shows that the square or higher integer powers do not
preserve P\'olya frequency functions.

\begin{lemma}\label{Lintpowers}
There exists a P\'olya frequency function $M$ such that
\begin{enumerate}
\item $M$ is even, continuous and vanishes nowhere,
\item $M$ is increasing on $( -\infty, 0 ]$ and decreasing on
$[ 0, \infty )$, and
\item $M^n : x \mapsto M( x )^n$ is not a P\'olya frequency function
for any integer $n \geq 2$.
\end{enumerate}
\end{lemma}

\begin{proof}
We claim that the Laplace-type function
$M( x ) := 2 e^{-| x |} - e^{-2| x |}$ has the desired properties.
More generally, we provide a one-parameter family of functions, each
of which is as required. Given a real number $\alpha > 0$,
let
\begin{equation}
M_\alpha : \R \to ( 0, \infty ); \ x \mapsto %
( \alpha + 1 ) \exp( - \alpha | x | ) - %
\alpha \exp( -( \alpha + 1 ) | x | ).
\end{equation}
It is readily verified that $M = M_\alpha$ has properties~(1) and~(2).
Furthermore, a short calculation shows that
\[
\cB\{ M \}( s ) = %
\frac{2 \alpha ( \alpha + 1 ) ( 2 \alpha + 1 )}%
{( s^2 - \alpha^2 ) ( s^2 - ( \alpha + 1 )^2 )}
\]
on a neighborhood of $0$. Hence $1 / \cB\{ M \}( s )$ is a polynomial
function with non-zero real roots and positive at the origin, and
so $M$ is a P\'olya frequency function \cite[Theorem~1]{Schoenberg51}.

We now analyze the Laplace transform of the higher integer powers of
$M$. A second calculation reveals that
\[
\cB\{ M^n \}( s ) = 2 \sum_{k = 0}^n ( -1 )^{k + 1} \binom{n}{k} %
\frac{\alpha^k ( \alpha + 1 )^{n - k} ( n \alpha + k )}%
{s^2 - ( n \alpha + k )^2} = \frac{p_n( s )}{q_n( s )} %
\qquad ( n \in \N ),
\]
where the polynomial
$q_n( s ) := \prod_{k = 0}^n ( s^2 - ( n \alpha + k )^2 )$ has
simple roots and degree $2 n + 2$, and the polynomial $p_n( s )$ has
degree no more than $2 n$.

We claim that, if $n > 1$, then $p_n$ is non-constant and coprime to
$q_n$. This implies that $q_n / p_n$ is not an entire function, whence
$M^n$ is not a P\'olya frequency function. To see this claim, note
that
\[
p_n( \pm (n \alpha + k ) ) = 2 ( -1 )^{k + 1} \binom{n}{k} %
\alpha^k ( \alpha + 1 )^{n - k} ( n \alpha + k ) %
\prod_{j \neq k} ( ( n \alpha + k )^2 - (n \alpha + j)^2 ) \neq 0
\]
for $k = 0$, \ldots, $n$, and so
\[
\frac{p_n( n \alpha )}{p_n( n \alpha + n )} = %
\prod_{j = 1}^{n-1} \frac{( 2 n \alpha + j ) ( \alpha + 1 )}%
{( 2 n ( \alpha + 1 ) - j) \alpha}.
\]
When $n > 1$, each factor in the final product is greater than
$1$. This shows the claim, and concludes the proof.
\end{proof}

We now use this result to obtain the very small class of maps that
preserve all P\'olya frequency functions.

\begin{theorem}\label{TPolya}
Let $F : [ 0, \infty ) \to [ 0, \infty )$. If $F \circ \Lambda$ is a
P\'olya frequency function for every P\'olya frequency function
$\Lambda$, then $F( x ) = c x$ for some $c > 0$.
\end{theorem}

The converse is, of course, immediate.

\begin{proof}
As $c G_1$ is a P\'{o}lya frequency function for all $c > 0$,
Proposition~\ref{PGauss} implies that $F$ is non-decreasing and
continuous on $( 0, \infty )$. Furthermore, since $F \circ \lambda$ is
a P\'olya frequency function, the integrability condition gives that
$F( 0 ) = 0$.

Since $F \circ \lambda$ is non-zero at least at two points, there
exists $t_0 > 0$ with $F( t_0 ) > 0$. Thus $F \circ ( t_0 \lambda )$
has a point of discontinuity, as it has distinct left and right limits
at the origin, and therefore
\[
F( t_0 \lambda ( x ) ) = c_0 e^{-b_0 x} \qquad \text{for all } x > 0,
\]
where $c_0$ and $b_0$ are positive constants. Therefore
\[
F( t ) = c_0 t^{b_0} \qquad \text{for all } t \in [ 0, t_0 );
\]
if $t_1 > t_0$, then, as $F$ is non-decreasing, repeating this working
shows the existence of positive constants $c_1$ and $b_1$ such that
\[
F( t ) = c_1 t^{b_1} \qquad \text{for all } t \in [ 0, t_1 ).
\]
It is readily seen that $b_1 = b_0$ and $c_1 = c_0$, and therefore
$F( t ) = c_0 t^{b_0}$ for all $t \geq 0$.

Next, since $\phi( x ) = x \lambda( x )$ is also a P\'olya frequency
function \cite[pp.~343]{Schoenberg51},  it follows
that $x^{b_0} \lambda( b_0 x )$ is a P\'olya frequency
function. The bilateral Laplace transform of this function is
\[
\int_0^\infty e^{-x s} x^{b_0} e^{-b_0 x} \intd x = %
\int_0^\infty e^{-x ( s + b_0 )} x^{b_0} \intd x = %
\frac{\Gamma( b_0 + 1 )}{( s + b_0 )^{b_0+1}} \qquad %
( s > -b_0 ).
\]
The reciprocal $( s + b_0 )^{b_0 + 1}$ admits an analytic continuation
to an entire function, as required by \cite[Theorem~1]{Schoenberg51},
only for integer values of~$b_0$. Lemma~\ref{Lintpowers} now gives
the result.
\end{proof}

To conclude, we provide a result that will be useful presently, as
well as being notable in its own right: the classification of
preservers of $\TN$ Toeplitz kernels. The following definition
is a slight variation on \cite[Definition~1]{Schoenberg51} that is
more convenient for our purposes.

\begin{definition}
A function $f : \R \to \R$ is \emph{totally non-negative} (or $\TN$)
if it is Lebesgue measurable and the Toeplitz kernel
\[
T_f : \R \times \R \to \R; \ ( x, y ) \mapsto f( x - y )
\]
is $\TN$.
\end{definition}

Thus a P\'olya frequency function is a $\TN$ function which is
integrable and non-zero at two or more points.

\begin{theorem}\label{Ttoeplitz}
Let $F : [ 0, \infty ) \to [ 0, \infty )$ be non-zero. The following
are equivalent.
\begin{enumerate}
\item Given any $\TN$ function $f$ that is non-zero at two or more
points, the composition $F \circ f$ is $\TN$.

\item Given any P\'olya frequency function $\Lambda$, the composition
$F \circ \Lambda$ is $\TN$.

\item The function $F$ is of the form $F( x ) = c$, $F( x ) = c x$,
$F( x ) = c \bJ_{x > 0}$ or $F( x ) = c \bJ_{x = 0}$, for some
$c > 0$.
\end{enumerate}
Similarly, the function $F$ preserves $\TN$ functions if and only if
$F( x ) = c$, $F( x ) = c x$, or $F( x ) = c \bJ_{x > 0}$, for some
$c > 0$.
\end{theorem}

\begin{proof}
We recall first a result of Schoenberg \cite[Lemma~4]{Schoenberg51},
that if $f$ is a $\TN$ function which is non-zero at two or more
points and not of the form
$f( x ) = \exp( a x + b )$, where $a$, $b \in \R$, then there exists
$\gamma \in \R$ such that $x \mapsto e^{\gamma x} f( x ) $ is a
P\'olya frequency function. (See
also~\cite[Chapter~7, Proposition~1.3]{Karlin}.)

Clearly $(1) \implies (2)$, and that $(3) \implies (1)$ is immediate
for constant or linear maps of the form under consideration, so
suppose $F( x ) = c \bJ_{x = 0}$ or $F ( x ) = c \bJ_{x > 0}$ for some
$c > 0$. By (\ref{Econvex}), the zero set of a $\TN$ function that is
non-zero at two or more points can have one of the following forms:
\[
\emptyset, \qquad ( -\infty, a \rangle, \qquad %
( \infty, a \rangle \cup \langle b, \infty ), \quad %
\text{or} \quad \langle b, \infty ) \qquad ( a, b \in \R, \ a < b ),
\]
where the angle bracket indicates the intervals may be either open or
closed. However, the third possibility is ruled out by the fact that
P\'olya frequency functions cannot have compact support
\cite[Corollary~1]{Schoenberg51}.

Thus $F \circ f$ is a non-negative constant or, up to positive scaling
and translation of the argument, one of the following functions:
\[
\bJ_{x \geq 0}, \qquad \bJ_{x > 0}, \qquad \bJ_{x \leq 0}, \quad %
\text{or} \quad \bJ_{x < 0}.
\]
Since $\lambda_1$ and $\lambda_0$ are $\TN$, so are the first
two of these; the remaining two follow from this, because a kernel
$K$ on $\R \times \R$ is $\TN$ if and only if the ``order-reversed''
kernel $K' : ( x , y ) \mapsto K( -x, -y )$ is $\TN$. Hence
$(3) \implies (1)$.

Next, suppose $F$ satisfies~$(2)$; we wish to show that $(3)$
holds. It follows from Proposition~\ref{PGauss} that $F$ is
non-negative, non-decreasing, and continuous on $( 0, \infty )$. If
$F$ has the form $c \bJ_{x = 0}$ or $c \bJ_{x > 0}$ for some $c > 0$,
then we are done, so we assume otherwise.

If $F$ is constant on $( 0, \infty )$, then the only remaining
possibility is that it is non-zero there and also non-zero at $0$.
Thus, applying $F$ to the P\'olya frequency function $\lambda_d$, where
$d \in [ 0, 1 ]$ is fixed for the remainder of the proof, we see that
\[
0 \leq %
\det( F \circ T_{\lambda_d} )[ ( -1, 1 ); ( 0, 2 ) ] = %
\begin{vmatrix}
F( 0 ) & F( 0 ) \\
F( 1 ) & F( 0 )
\end{vmatrix} = F( 0 ) ( F( 0 ) - F( 1 ) )
\]
whereas
\[
0 \leq %
\det( F \circ T_{\lambda_d} )[ ( 0, 2 ); ( -1, 1 ) ] = %
\begin{vmatrix}
F( 1 ) & F( 0 ) \\
F( 1 ) & F( 1 )
\end{vmatrix} = F( 1 ) ( F( 1 ) - F( 0 ) ).
\]
Hence $F$ is a positive constant and $(3)$ holds.

We may now suppose $F$ is not constant on $( 0, \infty )$. It follows
by continuity that $F$ is positive and not constant on an open
interval $( r, s )$, where $s > r > 0$. Now there are two cases to
consider.

First, suppose $F( 0 ) < F( t_0 )$ for some $t_0 > 0$ and fix
$t > \max\{ s, t_0 \}$. By assumption, there exists $\gamma_t \in \R$
such that
\[
\Lambda_t( x ) := e^{\gamma_t x} F( t \lambda_d( x ) )
\]
is a P\'olya frequency function or of the form $e^{a x + b}$. As
$F( t ) \geq F( t_0 ) > F( 0 )$, so $\Lambda_t$ is discontinuous
at~$0$, and therefore it cannot have the latter form. Moreover,
$F( t \lambda_d( x ) )$ is positive on an open sub-interval of
$( 0, \infty )$. It follows that
$\Lambda_t( x ) = p_t \lambda_{d_t}( q_t x )$ for suitable constants
$p_t$, $q_t > 0$ and $d_t \in [ 0, 1 ]$, so
\[
F( 0 ) = e^{\gamma_t} p_t \lambda_{d_t}( -q_t ) = 0 %
\qquad \text{and} \qquad %
F( t e^{-x} ) = p_t e^{-( \gamma_t + q_t ) x} %
\qquad \text{for all } x > 0.
\]
Since $t$ can be taken to be arbitrarily large, a simple argument
shows that $F( y ) = c y^\alpha$ for all $y > 0$, where $c > 0$ and
$\alpha > 0$ because $F$ is non-constant and non-decreasing on
$( 0, \infty )$. Applying $F$ to $\phi( x ) = x \lambda( x )$ gives a
P\'olya frequency function, since $x^\alpha e^{-\alpha x}$ is positive
and integrable on $( 0, \infty )$.  The proof of Theorem~\ref{TPolya}
now shows that $\alpha \in \N$. Furthermore, if $M$ is as in
Lemma~\ref{Lintpowers}, then $F( M ) = c^\alpha M^\alpha$ is
integrable and positive, so a P\'olya frequency function. Thus
$\alpha = 1$, as required.

The second and final case is when $F( 0 ) \geq F( t )$ for all
$t > 0$.  Choose $t \in ( r, s )$ such that $F( 0 ) > F( t ) > 0$ and
$F$ is positive and not constant on $( r, t )$. As before, note that
\[
\Lambda( x ) := e^{\gamma x} F( t \lambda_d( x ) )
\]
is a P\'olya frequency function for some choice of $\gamma$; it cannot
be of the form $e^{a x + b}$, since $\Lambda$ is discontinuous
at~$0$. This discontinuity, and the positivity of $\Lambda$ on some
sub-interval of $( 0, \infty )$, means that
$\Lambda( x ) = p \lambda_d( q x )$ for some $d \in [ 0, 1 ]$ and
constants $p$, $q > 0$. Then
\[
F( 0 ) = F( t \lambda_d( -q ) ) = e^\gamma \Lambda( -1 ) = %
e^\gamma p \lambda_d( -q ) = 0 < F( t ) < F( 0 ),
\]
a contradiction.

This shows the first set of equivalences. We now turn to the final
assertion, beginning with the ``only if'' part. As $(1) \implies (3)$,
we see that $F$ is from one of four families, and it remains to rule
out the function $F( x ) = c \bJ_{x = 0}$, where $c > 0$. This follows
by applying $F$ to itself, as $F$ is readily seen to be $\TN$, but
$F \circ F = c - F$, which is not even $\TN_2$.

Conversely, to show the ``if'' part, since $(3) \implies (1)$, it
suffices to verify that $F \circ f$ is $\TN$ when
$F( x ) = \bJ_{x > 0}$ and $f( x ) = f_a( x ) = a \bJ_{x = b}$ for
any $a \geq 0$ and $b \in \R$. In this case, either
$F \circ f_a = f_1$, when $a > 0$, or $F \circ f_0 = f_0$. This
completes the proof.
\end{proof}

\begin{remark}\label{Rtoeplitz}
The preceding proof shows that a non-zero function
$F : [ 0, \infty ) \to [ 0, \infty )$ belongs to the classes of
functions in $(3)$ if it preserves $\TN$ for the following restricted
set of test functions:
(i)~the Gaussian functions $c G_1( x )$ for all $c > 0$,
(ii)~the P\'olya frequency functions $t \lambda_d( x )$ for
one $d \in [ 0, 1 ]$ and all $t > 0$,
(iii)~the P\'olya frequency function $\phi( x ) = x \lambda( x )$, and
(iv)~the P\'olya frequency function $M$ from Lemma~\ref{Lintpowers}.
If $F$ also preserves $\TN$ for a positive multiple of the
function $\bJ_{x = 0}$, then $F$ cannot have this form itself.
\end{remark}

\subsection{Totally positive P\'olya frequency functions}

The rigidity of the above class of endomorphisms carries over to
other, related problems. We begin by showing the same rigidity for the
class of $\TP$ P\'olya frequency functions, where we say that a
P\'olya frequency function $\Lambda$ is $\TP$ whenever the associated
Toeplitz kernel $T_\Lambda$ has that property. The precise description
of conditions on the data $x_1 < x_2 < \cdots < x_n$ and
$y_1 < y_2 < \cdots < y_n$ to that ensure that
\[
\det( \Lambda( x_i - y_j ) )_{i, j = 1}^n > 0
\]
are contained in Schoenberg's third article
\cite{SchoenbergWhitney53}. As much as convolution with the Gaussian
kernel was essential in the previous sections, it is also very useful
in this new framework.

For any $\gamma > 0$, let
\[
g_\gamma : \R \to \R; \ x \mapsto %
\frac{1}{2 \sqrt{\pi \gamma}} \exp( -x^2 / 4 \gamma )
\]
be the normalized Gaussian function, so that
$\int_\R g_\gamma (t) \intd t = 1$, and let $\Lambda$ be an arbitrary
P\'olya frequency function.  As the class of P\'olya frequency
functions is closed under convolution \cite[Lemma~5]{Schoenberg51},
the convolution $g_\gamma \ast \Lambda$ is also a P\'olya frequency
function, with bilateral Laplace transform equal to the product of the
transforms of $g_\gamma$ and $\Lambda$. In view of
\cite[Theorem~1]{SchoenbergWhitney53}, the kernel
\[
\R \times \R \to \R; \ ( x, y ) \mapsto %
( g_\gamma \ast \Lambda )( x - y )
\]
is $\TP$.

A P\'olya frequency function $\Lambda$ is bounded and has left
and right limits everywhere, so, for any $x_0 \in \R$,
\[
\lim_{\gamma \to 0^+} ( g_\gamma \ast \Lambda )( x_0 ) = %
\int_\R g_1( t ) \Lambda( x_0 - t \sqrt{\gamma} ) \intd t = %
\frac{1}{2} \left( \lim_{y \to x_0^+} \Lambda( y ) + %
\lim_{y \to x_0^-} \Lambda( y  ) \right).
\]
We say that a P\'olya frequency function is \emph{regular} if it is
equal to the arithmetic mean of its left and right limits at every
point.

Putting together these observations, we obtain the following result.

\begin{proposition}\label{Ptoeplitz}
Let $\Lambda$ be a regular P\'olya frequency function. There exists a
sequence of $\TP$ P\'olya frequency functions
$( \Lambda_n )_{n \geq  1}$ such that
$\lim_{n \to \infty} \Lambda_n( x ) = \Lambda( x )$ for every
$x \in \R$.
\end{proposition}

As an application, we complement Theorems~\ref{TPolya}
and~\ref{Ttoeplitz} above, by considering $\TP$ kernels. We say
that a kernel $K$ on $\R \times \R$ is \emph{measurable} if $K$ is a
Lebesgue-measurable function.

\begin{theorem}\label{TtoeplitzTP}
Given a function $F : ( 0, \infty ) \to ( 0, \infty )$, the following
are equivalent.
\begin{enumerate}
\item $F \circ K$ is a $\TP$ Toeplitz kernel on $\R \times \R$
whenever $K$ is.

\item $F \circ K$ is a $\TP$ measurable Toeplitz kernel on
$\R \times \R$ whenever $K$ is.

\item $F \circ \Lambda$ is a $\TP$ P\'olya frequency function whenever
$\Lambda$ is.

\item $F \circ \Lambda$ is a $\TP$ P\'olya frequency function
whenever $\Lambda$ is a regular $\TP$ P\'olya frequency function.

\item $F( x ) = c x$, where $c > 0$.
\end{enumerate}
\end{theorem}

\begin{proof}
It is clear that $(1)$, $(2)$ and $(3)$ both imply $(4)$ and are
implied by $(5)$. Thus it remains to show $(4) \implies (5)$.

Let $F$ satisfy~$(4)$. Then $F$ is positive, increasing, and
continuous on $( 0, \infty )$, by Proposition~\ref{PGauss}, so $F$
extends to a continuous, increasing function
$\widetilde{F} : [ 0, \infty ) \to [ 0, \infty )$.

Now suppose $\Lambda$ is one of the regular P\'olya frequency
functions listed in Remark~\ref{Rtoeplitz}, and note that this
includes $\lambda_{1 / 2}$. By Proposition~\ref{Ptoeplitz}, there
exists a sequence $( \Lambda_n )_{n \geq 1}$ of $\TP$ P\'olya
frequency functions such that $\Lambda_n \to \Lambda$ pointwise. Hence
$F \circ \Lambda_n$ gives rise to a $\TP$ kernel for each $n \geq 1$,
and so $\widetilde{F} \circ \Lambda$ is $\TN$. By
Remark~\ref{Rtoeplitz}, the restriction $F$ of $\widetilde{F}$ is
constant or linear, and the former is impossible. This completes the
proof.
\end{proof}

\section{P\'olya frequency sequences}\label{Stoeplitz2}

We continue our study of total non-negativity preservers with an
exploration of the class of P\'olya frequency sequences. A
\emph{P\'olya frequency sequence} is a bi-infinite sequence of real
numbers $\ba = ( a_n )_{n \in \Z}$ such that the Toeplitz kernel
\[
T_\ba : \Z \times \Z \to \R; \ ( i, j ) \mapsto a_{i - j}
\]
is $\TN$. Recall in this context the groundbreaking body of work by
Aissen, Edrei, Schoenberg, and Whitney (see \cite{AESW} and the
monograph by Karlin \cite{Karlin}). These sequences are characterized
in terms of negative real-rootedness of the associated generating
polynomial when most terms $a_n$ are zero, or a product expansion when
all negatively indexed terms $a_n$ vanish. More recently, P\'olya
frequency sequences have found numerous applications in combinatorics,
owing to their connections to log concavity. See the works of
Brenti~\cite{Br1,Br2} and subsequent papers.

P\'olya frequency sequences turn out to be as rigid as P\'olya
frequency functions are, as far as their endomorphisms are concerned,
with their preservers being dilations or constants.
In order to demonstrate this fact, we first introduce Toeplitz kernels
on a more general class of domains than $\R \times \R$.

\begin{definition}\label{Dadmissible}
We say that a pair of subsets $X \subseteq \R$ and $Y \subseteq \R$ is
\emph{admissible} if, for each integer $n \geq 2$, the sets contain
$n$-step arithmetic progressions $\bx \in \inc{X}{n}$ and
$\by \in \inc{Y}{n}$ that are \emph{equi-spaced}, so that their terms
have the same increments:
\[
x_{j + 1} - x_j = y_{j + 1} - y_j = x_2 - x_1 %
\qquad \text{for all } j \in [n-1].
\]
We let
\[
X - Y := \{ x - y : x \in X, \ y \in Y \}
\]
and say that a kernel $K : X \times Y \to \R$ is \textit{Toeplitz} if
there exists a function $f : X - Y \to \R$ such that
$K( x, y ) = f( x - y )$ for all $x \in X$ and $y \in Y$.
\end{definition}

The following theorem is a variant on Theorem~\ref{Ttoeplitz} for this
new setting.

\begin{theorem}\label{Tpfseq}
Suppose $X$ and $Y$ are a pair of admissible sets. If
$F : [ 0, \infty ) \to [ 0, \infty )$ is non-zero and preserves
the $\TN$ Toeplitz kernels on $X \times Y$, then either $F( x ) = c$,
$F( x ) = c x$, or $F( x ) = c \bJ_{x > 0}$, for some $c > 0$.
\end{theorem}

The converse to this theorem does not necessarily hold for a given
admissible pair. For example, if $X = Y = \R$, then
Theorem~\ref{Ttoeplitz} shows the converse, at least for the case of
measurable kernels, but for $X = Y = \Z$ we will see below
that $c \bJ_{x > 0}$ is not a preserver of P\'olya frequency sequences.

\begin{proof}
For ease of exposition, we split the proof into several steps.

Step 1: The function $F$ is non-decreasing on $( 0, \infty )$. Fix
equi-spaced arithmetic progressions $\bx \in \inc{X}{2}$ and
$\by \in \inc{Y}{2}$, so that $x_2 - x_1 = y_2 - y_1$. Let $p$ and $q$
be positive real numbers, with $p < q$, and consider the kernel
\[
K : X \times Y \to \R; \ ( x, y ) \mapsto %
q G_1 \left( \sqrt{\log( q / p )} %
\frac{( x - x_1 ) - ( y - y_1 )}{x_2 - x_1} \right).
\]
This is a Toeplitz kernel which is $\TP$ since Gaussian kernels are,
so $F \circ K$ is $\TN$ by assumption. Thus
\[
0 \leq \det ( F \circ K )[ ( x_1, x_2 ); ( y_1, y_2 ) ] = %
\begin{vmatrix} F( q ) & F( p ) \\ F( p ) & F( q ) \end{vmatrix} = %
F( q )^2 - F( p )^2.
\]
As $p$ and $q$ are arbitrary, the claim follows. Furthermore, the
function $F$ has at most countably many discontinuities, so is Borel
measurable, and the left-limit and right-limit functions
\[
F^+ : ( 0, \infty ) \to [ 0, \infty ); \ %
x \mapsto \lim_{y \to x^+} F( y ) %
\quad \text{and} \quad %
F^- : ( 0, \infty ) \to [ 0, \infty ); \ %
x \mapsto \lim_{y \to x^-} F( y )
\]
are well defined.

Step 2: The function $F$ is continuous on $( 0, \infty )$. We will
show that, for any $p > 0$, $F^+( p ) = F( p ) = F^-( p )$. This is
trivial if $F \equiv 0$ on $( 0, \infty )$, so we assume
otherwise. Fix a point $q > p$ where $F$ is continuous and
$F( q ) > 0$, choose an integer $n \geq 3$, and let
$\bx \in \inc{X}{n}$ and $\by \in \inc{Y}{n}$ be equi-spaced
arithmetic progressions. As before, the kernel
\[
K : X \times Y \to \R; \ ( x, y ) \mapsto %
q G_1 \left( \sqrt{\log( q / p )}
\frac{( x - x_1 ) - ( y - y_1 )}{( x_2 - x_1 )( n - 2 )} \right)
\]
is Toeplitz and $\TP$. Furthermore, a straightforward computation
shows that
\[
K( x_i, y_j ) = q ( p / q )^{( i - j )^2 / ( n - 2 )^2} \qquad %
\text{for all } i, j \in [n].
\]
Since $F \circ K$ is $\TN$, we have that
\begin{align*}
0 \leq \lim_{n \to \infty} \det ( F \circ K )%
[ ( x_1, x_{n - 2}, x_{n - 1} ); ( y_1, y_{n - 2}, y_{n - 1} ) ] & %
\to \begin{vmatrix}
 F( q ) & F^-( p ) & F( p ) \\
 F^-( p ) & F( q ) & F( q )\\
 F( p ) & F( q ) & F( q )
\end{vmatrix} \\
 & = -F( q ) ( F( p ) - F^-( p ) )^2,
\end{align*}
and so $F( p ) = F^-( p )$. Similarly,
\[
0 \leq \lim_{n \to \infty} \det ( F \circ K )%
[ ( x_1, x_{n - 1}, x_n ); ( y_1, y_{n - 1}, y_n ) ] = %
-F( q ) ( F( p ) - F^+( p ) )^2.
\]
This establishes the second claim.

Step 3: The function $F$ belongs to one of the four families of
functions in Theorem~\ref{Ttoeplitz}. Suppose not, and note that
Remark~\ref{Rtoeplitz} gives a P\'olya frequency function $\Lambda$
such that $F \circ \Lambda$ is not $\TN$. As $F$ is Borel measurable,
the function $F \circ \Lambda$ is Lebesgue measurable and therefore
$F \circ T_\Lambda$ is not $\TN$; furthermore, $\Lambda$ is
continuous except possibly at the origin, and is either positive
everywhere, or zero on $( -\infty, 0 )$ and positive on
$( 0, \infty )$. It follows that $F \circ \Lambda$ is continuous
except possibly at the origin. By Lemma~\ref{Lhankel}, there exists
$\bz \in \inc{\R}{n}$, where $n \geq 2$, such that the principal
submatrix $( F \circ T_\Lambda )[ \bz; \bz ]$ has at least one minor
which is negative. As $F \circ \Lambda$ is continuous except at
possibly the origin, we may assume that
$\bz = ( z_1, \ldots, z_n ) \in \inc{\Q}{n}$.

Choose $N \in \N$ sufficiently large so that $N z_i$ is an integer for
every $i \in [n]$, and let $m = N ( z_n - z_1 ) + 1$. Let
$\bx \in \inc{X}{m}$ and $\by \in \inc{Y}{m}$ be equi-spaced
arithmetic progressions, and let
\[
K : X \times Y \to \R; \ ( x, y ) \mapsto %
\Lambda\left( \frac{( x - x_1 ) - ( y - y_1 )}{( x_2 - x_1 ) N} %
\right).
\]
This Toeplitz kernel is $\TN$, since $\Lambda$ is, and therefore so is
$F \circ K$. However, the submatrix
\[
( F \circ K )[ ( 1, \ldots, m ); ( 1, \ldots, m ) ] = %
( F( \Lambda( ( i - j ) / N ) ) )_{i, j = 1, \ldots, m}
\]
contains $( F \circ T_\Lambda )[ \bz; \bz ]$, and this is the desired
contradiction.

Step 4: The function $F$ cannot have the form $c \bJ_{x = 0}$, where
$c > 0$. To see this, fix equi-spaced sequences $\bx \in \inc{X}{2}$
and $\by \in \inc{Y}{2}$, and let
\[
K : X \times Y \to \R: \ ( x, y ) \mapsto %
\begin{cases}
 1 & \text{if } x - y = x_1 - y_1, \\
 0 & \text{otherwise}.
\end{cases}
\]
Then $K$ is Toeplitz and $\TN$, since each row of any submatrix of $K$
contains $1$ at most once, and similarly for each column. However,
$c \bJ_{x = 0} \circ K = c ( 1 - K )$ and
\[
\det( 1 - K )[ ( x_1, x_2 ); ( y_1, y_2 ) ] = %
\begin{vmatrix} 0 & 1 \\ 1 & 0 \end{vmatrix} = -1.
\qedhere
\]
\end{proof}

As a consequence, we now classify the preservers of P\'olya
frequency sequences.

\begin{corollary}\label{Cpfseq}
Let $F : [ 0, \infty ) \to [ 0, \infty )$ be non-constant. The
sequence $F \circ \ba$ is a P\'olya frequency sequence for every
P\'olya frequency sequence $\ba$ if and only if $F( x ) = c x$ for
some $c > 0$.
\end{corollary}

\begin{proof}
One implication is immediate. For the other, an application of
Theorem~\ref{Tpfseq} with $X = Y = \Z$ means we need only show that
$F$ is continuous at the origin to obtain the desired result.

Define sequences $\bb = ( b_m )_{m \in \Z}$ and
$\bc = ( c_m )_{m \in \Z}$ by setting
\[
b_m = \begin{cases}
 1 & \text{if } m = 0, \\ 0 & \text{otherwise}
\end{cases} \qquad \text{and} \qquad %
c_m = \begin{cases}
 1 & \text{if } m = 0 \text{ or } 2, \\
 2 & \text{if } m = 1, \\
 0 & \text{otherwise}.
\end{cases}
\]
These are P\'olya frequency sequences, by \cite[Theorem~6]{AESW},
since the only zeros of their polynomial generating functions $1$ and
$1 + 2 z + z^2$ are negative. Hence
\[
\lim_{t \to 0^+} %
\det( F \circ t T_\bb )[ ( 1, 3, 4 ); ( 1, 2, 3 ) ] = %
\begin{vmatrix}
  F^+( 0 ) & F( 0 ) & F( 0 ) \\
  F( 0 ) & F( 0 ) & F^+( 0 ) \\
  F( 0 ) & F( 0 ) & F( 0 ) \\
\end{vmatrix} = %
-F( 0 ) ( F^+( 0 ) - F( 0 ) )^2
\]
is non-negative, as is
\[
\lim_{t \to 0^+} %
\det( F \circ t T_\bc )[ ( 2, 3, 4 ); ( 1, 2, 3 ) ] = %
-F^+( 0 ) ( F^+( 0 ) - F( 0 ) )^2.
\]
Now either $F( 0 ) = F^+( 0 ) = 0$, in which case we are done, or at
least one of $F( 0 )$ and $F^+( 0 )$ is positive, in which case they
are equal.
\end{proof}

\begin{remark}
The test families of P\'olya frequency functions used to classify the
preservers of $\TN$ functions listed in Remark~\ref{Rtoeplitz} can be
used to obtain test sets of P\'olya frequency sequences. To see this,
suppose $\Lambda$ is a P\'olya frequency function that is continuous
except possibly at the origin, and let
$F : [ 0, \infty ) \to [ 0, \infty )$ be continuous. Then
$F \circ \Lambda$ is $\TN$ if and only if
$F \circ {\Lambda_{(N)}}$ is $\TN$ for every P\'olya frequency
sequence $\Lambda_{(N)} := ( \Lambda( n / N ) )_{n \in \Z}$, where
$N \in \N$. One implication is immediate, and the converse follows
using similar reasoning to Lemma~\ref{Lhankel} and Step~3 in the proof
of Theorem~\ref{Tpfseq}. In particular, for any integer $n \geq 2$,
there exists some $N \in \N$ such that $M^n_{(N)}$ is not a P\'olya
frequency sequence, where $M$ is as in Lemma~\ref{Lintpowers}.
\end{remark}

We now turn to classifying $\TP$ preservers for P\'olya frequency
sequences. The next result is a version of Theorem~\ref{TtoeplitzTP}
in the same setting as that of Theorem~\ref{Tpfseq}.

\begin{theorem}\label{Tpfseq2}
Let $X$ and $Y$ be a pair of admissible sets and let
$F : ( 0, \infty ) \to ( 0, \infty )$. The following are equivalent.
\begin{enumerate}
\item The composition operator $C_F$ preserves total positivity
for all Toeplitz kernels on $X \times Y$.
\item $F( x ) = c x$ for some $c > 0$.
\end{enumerate}
\end{theorem}

As an immediate consequence, we have the following result.

\begin{corollary}\label{Ctppfseq}
The preservers of total positivity for P\'olya frequency sequences are
precisely the dilations $F( x ) = c x$, where $c > 0$.
\end{corollary}

\begin{proof}[Proof of Theorem~\ref{Tpfseq2}]
That $(2) \implies (1)$ is immediate. For the converse, suppose $(1)$
holds and note that $F$ must be positive, increasing, and continuous
on $( 0, \infty )$; the argument is essentially that of Step~1 in
the proof of Theorem~\ref{Tpfseq}.

Next, we extend $F$ to a continuous, increasing function
$\widetilde{F} : [ 0, \infty ) \to [ 0, \infty )$ and suppose for
contradiction that $F$ is not a dilation. Then, by
Remark~\ref{Rtoeplitz}, there exists a regular P\'olya frequency
function $\Lambda$ that is continuous except perhaps at the origin and
such that $\widetilde{F} \circ \Lambda$ is not $\TN$.

By Lemma~\ref{Lhankel}, there exists $\bz \in \inc{\R}{n}$, where
$n \geq 2$, such that
$A = ( \widetilde{F} \circ T_\Lambda )[ \bz; \bz ]$ is not $\TN$,
that is, $A$ has at least one negative minor. By continuity, we may
assume that $\bz = ( z_1, \ldots, z_n ) \in \inc{\Q}{n}$.

By Proposition~\ref{Ptoeplitz}, there exists a sequence
$( \Lambda_k )_{k \geq 1}$ of $\TP$ P\'olya frequency functions such
that $\Lambda_k \to \Lambda$ pointwise as $k \to \infty$. Since
$A$ depends on the value of $\Lambda$ only at the finite set of values
$\{ z_i - z_j : i, j \in [n] \}$, there exists $k \in \N$ such that
$A' = ( \widetilde{F} \circ T_{\Lambda_k} )[ \bz; \bz ]$ has a
negative minor.

We now follow the last part of Step~$3$ in the proof of
Theorem~\ref{Tpfseq}. Let $N \in \N$ be sufficiently large so that
$N z_i \in \Z$ for all $i \in [n]$, let $m = N ( z_n - z_1 ) + 1$ and
choose $\bx \in \inc{X}{m}$ and $\by \in \inc{Y}{m}$ to define the
kernel
\[
\Lambda' : X \times Y \to \R; \ ( x, y ) \mapsto %
\Lambda_k\left( \frac{( x - x_1 ) - ( y - y_1 )}{( x_2 - x_1 ) N} %
\right).
\]
Then $\Lambda'$ is $\TP$, since $\Lambda_k$ is and therefore so is
$F \circ \Lambda'$. However, $A'$ occurs as a submatrix of
$F \circ \Lambda'$ and this is the desired contradiction.
\end{proof}

The final theorem in this section classifies the powers that preserve
symmetric Toeplitz matrices.

\begin{theorem}\label{TPFseqpowers}
The only power functions preserving $\TN$ symmetric Toeplitz matrices
are $F( x ) = 1$ and $F( x ) = x$.
\end{theorem}

As a first step in the proof of this theorem, we obtain the following
result.

\begin{proposition}\label{Pjain}
Let $A := ( \cos( ( i - j ) \theta ) )_{i, j = 1}^n$, where $n \geq 2$
and the angle $\theta$ is such that $0 < \theta < \pi / ( 2 n - 2 )$.
Then $A^{\circ \alpha}$ is positive semidefinite if and only if
$\alpha \in \N_0$ or $\alpha \in [ n - 2, \infty )$.
\end{proposition}

\begin{proof}
We appeal to a result of Jain~\cite[Theorem 1.1]{Jain}, which states
that, for distinct positive real numbers $x_1$, \ldots, $x_n$, the
entrywise $\alpha$th power of $X := ( 1 + x_i x_j )_{i, j = 1}^n$ is
positive semidefinite if and only if $\alpha \in \N_0$ or
$\alpha \in [ n - 2, \infty )$. Now let $x_j := \tan( j \theta )$, and
let $D$ be the diagonal $n \times n$ matrix with $( j, j )$ entry
$\cos( j \theta )$. Then $X^{\circ \alpha}$ is positive semidefinite if
and only if
$D^\alpha X^{\circ \alpha} D^\alpha = ( D X D )^{\circ \alpha}$ is,
but $D X D = A$.
\end{proof}

\begin{proof}[Proof of Theorem~\ref{TPFseqpowers}]
We begin with the observation that the cosine Toeplitz kernel
\[
K : ( -\pi/ 4, \pi / 4 ) \times ( -\pi / 4, \pi / 4 ) \to \R; \ %
( x, y ) \mapsto \cos( x - y )
\]
is $\TN$. Note first that this kernel has rank 2, by the cosine
identity for differences: given $\bx$,
$\by \in \inc{( -\pi / 4, \pi / 4)}{n}$ for some $n \in \N$, we have
that
\[
K[ \bx; \by ] = \bc_x \bc_y^T + \bs_x \bs_y^T,
\]
where
\[
\bc_x := ( \cos x_1, \ldots, \cos x_n )^T %
\quad \text{and} \quad %
\bs_x := ( \sin x_1, \ldots, \sin x_n )^T,
\]
and similarly for $\bc_y$ and $\bs_y$. Thus, every minor of size at
least $3 \times 3$ vanishes. If $( x_1, x_2 )$,
$( y_1, y_2 ) \in \inc{(-\pi/4, \pi/4)}{2}$, then a 
direct computation shows that
\[
\begin{vmatrix}
 \cos( x_1 - y_1 ) & \cos( x_1 - y_2 ) \\
 \cos( x_2 - y_1 ) & \cos( x_2 - y_2 )
\end{vmatrix} = \sin( x_2 - x_1 ) \sin( y_2 - y_1 ) > 0,
\]
and that $\cos( x - y ) \geq 0$ whenever $| x | < \pi / 4$ and
$| y | < \pi/4$ is immediate.

Given this observation, we proceed to eliminate possibilities for
$\alpha$. The test matrix
\[
\begin{pmatrix} 2 & 1 \\ 1 & 2 \end{pmatrix}
\]
shows that $\alpha$ cannot be negative. Next, suppose $\alpha$ is
positive and not an integer, and let $n$ be an integer greater than
$\alpha + 2$. If $A$ is the $n \times n$ matrix of
Proposition~\ref{Pjain}, with $\theta = \pi / 2 n$, then $A$ is $\TN$,
since it occurs as a submatrix of $K$, but $A^{\circ \alpha}$ is not
positive semidefinite, so not $\TN$.

The final case is if $F( x ) = x^k$ for some integer $k \geq 2$. Let
$M$ be the even function of Lemma~\ref{Lintpowers}, and let
$\bz \in \inc{\R}{n}$ be such that
$B := T_{M^k}[ \bz;  \bz ]$ has a negative minor. By
continuity, we may assume that
$\bz = ( z_1, \ldots, z_n ) \in \inc{\Q}{n}$. We may also translate
each coordinates by $-z_1$, since this leaves
$B$ unchanged, thus $z_1 = 0 < z_n$. Choose $N \in \N$ such that
$z_i N \in \Z$ for every $i \in [n]$, and consider the symmetric
Toeplitz matrix $C := ( M( ( i - j ) / N ))_{i, j = 0}^{z_n N}$. This
is $\TN$, since it occurs as a submatrix of $M$, and therefore so is
$F[ C ] = C^{\circ k}$, but this matrix contains $B$ as a principal
submatrix, and so has a negative minor. This contradiction completes
the proof.
\end{proof}

We conclude with two observations. First, the above classifications of
preservers of P\'olya frequency functions and sequences, including
Lemma~\ref{Lintpowers}, sit in marked contrast
to~\cite[Theorem~3]{Weinberger83}. There, Weinberger shows that the
set $\PF_3$ of functions defined analogously to P\'olya frequency
functions, but with the $\TN$ condition replaced by $\TN_3$, is closed
under taking any real power greater than or equal to~$1$.

Second, a result in the parallel paradigm of the holomorphic
functional calculus, not the Schur--Hadamard calculus considered here,
can be found in~\cite[pp.~451--452]{Karlin}. There, it is proved that
a smooth function preserves $\TN$ matrices via the functional calculus
if and only if it is a non-negative integer power. A close look at the
proof reveals that the same conclusion can be deduced by using the
smaller test set of $\TN$ upper-triangular Toeplitz matrices. This
stands in contrast to the results of the next section.

\section{One-sided P\'olya frequency functions and sequences}%
\label{Stoeplitz3}

As a final variation on the P\'olya-frequency theme, we turn to the
class of one-sided P\'olya frequency sequences, where the terms vanish
for negative indices. As discussed at the start of the previous
section, P\'olya frequency sequences, including the one-sided variant,
are well studied, with a representation theorem~\cite{AESW} and
applications in analysis, combinatorics, and other areas.

We prove here that the only preservers of this class are homotheties
and Heaviside functions. In the spirit of the previous results, we
begin by showing the analogous result for P\'olya frequency functions
and $\TN$ functions, akin to Theorem~\ref{Ttoeplitz}.

We say that a function $f : \R \to \R$ is \emph{one sided} if there
exists $x_0 \in \R$ such that either $f( x ) = 0$ for all $x < x_0$,
or $f( x ) = 0$ for all $x > x_0$.

\begin{theorem}\label{T1sidedPFF}
Let $F : [ 0, \infty ) \to [ 0, \infty )$.
\begin{enumerate}
\item The function $F \circ \Lambda$ is a one-sided P\'olya frequency
function whenever $\Lambda$ is, if and only if $F( x ) = c x$ for some
$c > 0$.

\item The function $F \circ f$ is a one-sided $\TN$ function that is
non-zero at two or more points whenever $f$ is, if and only if
$F( x ) = c x$, $F( x ) = c \bJ_{x > 0}$, or $F( x ) = c \bJ_{x = 0}$,
for some $c > 0$.

\item Suppose $F$ is non-zero. Then $F \circ f$ is a one-sided $\TN$
function whenever $f$ is, if and only if $F( x ) = c x$ or
$F( x ) = c \bJ_{x > 0}$, for some $c > 0$.
\end{enumerate}
\end{theorem}

The proof uses the following one-sided variant of
Lemma~\ref{Lintpowers}.

\begin{lemma}\label{L1sided}
Let $a_1$, $a_2$, and $a_3$ be positive real numbers, with
$a_1 < a_2 < a_3$, such that the set $\{ a_1, a_2, a_3 \}$ is linearly
independent over the rational numbers, and let the non-zero real
numbers $c_1$, $c_2$, and~$c_3$ be such that
\begin{equation}\label{Econstraints}
c_1 > 0, \qquad c_1 + c_2 + c_3 = 0, \qquad \text{and} \qquad %
a_1 c_1 + a_2 c_2 + a_3 c_3 = 0.
\end{equation}
Then
\[
N : \R \to \R; \ x \mapsto \begin{cases}
 c_1 e^{-a_1 x} + c_2 e^{-a_2 x} + c_3 e^{-a_3 x} & %
 \text{if } x \geq 0, \\
 0 & \text{if } x < 0
\end{cases}
\]
is a continuous P\'olya frequency function such that
$N^n : x \mapsto N( x )^n$ is not a P\'olya frequency function for any
integer $n \geq 2$.
\end{lemma}

The function $N$ of Lemma~\ref{L1sided} is, up to scaling, a member of
the class of one-sided P\'olya frequency functions which we call
Hirschman--Widder densities. These were studied by Hirschman and
Widder in their 1949 paper~\cite{HW49}, and their 1955
monograph~\cite{HW} contains a detailed analysis of such exponential
polynomials and their Laplace transforms.

The work of Hirschman and Widder is closely intertwined with that of
Schoenberg. In 1947, Schoenberg~\cite{Schoenberg47} announced the
notion of a P\'olya frequency function. In their 1949 work, Hirschman
and Widder studied these maps and their order of smoothness, via their
Laplace transforms. This was followed by Schoenberg's first full paper
on P\'olya frequency functions~\cite{Schoenberg51} in 1951.

In a forthcoming piece of work \cite{BGKP-HW}, we show that
Hirschman--Widder densities satisfy the conclusions of
Lemma~\ref{L1sided} generically, as long as they involve at least
three distinct exponential terms: their higher integer powers are not
$\TN$, and so are not P\'olya frequency functions. Indeed, a stronger
result is established, with powers replaced by non-homethetic
polynomial functions.

\begin{proof}[Proof of Lemma~\ref{L1sided}]
Throughout this proof, sums and products are taken over non-negative
integers satisfying the given conditions. For any $n \in \N$, let
\[
F_n( s ) := \cB\{ N^n \}( s ) = %
\sum_{i + j + k = n} \binom{n}{i, j, k} %
\frac{c_1^i c_2^j c_3^k}{s + i a_1 + j a_2 + k a_3} = %
\frac{p_n( s )}{q_n( s )}
\]
be the bilateral Laplace transform of $N$, where the monic polynomial
\[
q_n( s ) := \prod_{i + j + k = n} ( s + i a_1 + j a_2 + k a_3 )
\]
has degree $( n + 1 ) ( n + 2 ) / 2$ and
\[
p_n( s ) := %
\sum_{i + j + k = n} \binom{n}{i, j, k} c_1^i c_2^j c_3^k %
\prod_{\substack{i' + j' + k' = n \\ ( i', j', k' ) \neq ( i, j, k )}}
( s + i' a_1 + j' a_2 + k' a_3 )
\]
has degree no more than
$n ( n + 3 ) / 2$. In fact, the leading coefficient of $p_n( s )$ is
\[
\sum_{i + j + k = n} \binom{n}{i, j, k} c_1^i c_2^j c_3^k = %
( c_1 + c_2 + c_3 )^n = 0,
\]
whereas the next-highest-order coefficient of $p_n( s )$ is
\begin{align*}
\sum_{i + j + k = n} \binom{n}{i, j, k} & c_1^i c_2^j c_3^k \biggl( %
\bigl( \sum_{i' + j' + k' = n} i' a_1 + j' a_2 + k' a_3  \bigr) %
 - i a_1 - j a_2 - k a_3 \biggr) \\
 & = \frac{1}{6} n ( n + 1 ) ( n + 2 ) ( a_1 + a_2 + a_3 ) %
 ( c_1 + c_2 + c_3 )^n \\
 & \qquad %
 - n ( c_1 + c_2 + c_3 )^{n - 1} ( a_1 c_1 + a_2 c_2 + a_3 c_3 ) \\
 & = 0,
\end{align*}
since
\[
\sum_{i' + j' + k' = n} i' = %
\frac{1}{6} n ( n + 1 ) ( n + 2 )
\]
and
\[
\sum_{i + j + k = n} \binom{n}{i, j, k} i x^i y^j z^k = %
x \frac{\partial}{\partial x} ( x + y + z )^n = %
n x ( x + y + z )^{n - 1}.
\]
The constraints (\ref{Econstraints}), together with the ordering of
$a_1$, $a_2$, and $a_3$, imply that
\[
c_2 < 0, \qquad c_1 + c_2 < 0 \quad \text{and} \quad %
p_1 \equiv \frac{c_1 c_2 ( a_1 - a_2 )^2}{c_1 + c_2 } > 0,
\]
whence $1 / F_1( s ) = q_1( s ) / p_1( s )$ is a polynomial with real
roots $-a_1$, $-a_2$ and $-a_3$. Furthermore, as
\[
\frac{1}{F_1( 0 )} = \frac{q_1( 0 )}{p_1( 0 )} = %
\frac{a_1 a_2 a_3 ( c_1 + c_2 )}{c_1 c_2 ( a_1 - a_2 )^2} > 0,
\]
it follows that $N$ is a P\'olya frequency function, by
\cite[Theorem~1]{Schoenberg51}.

For the second part, we will show that
$1 / F_n( s ) = q_n( s ) / p_n( s )$ is the restriction of a rational
function with simple poles whenever $n \geq 2$, and so it cannot be
the restriction of an entire function; it then follows from
\cite[Theorem~1]{Schoenberg51} that $N^n$ is not a P\'olya frequency
function for such $n$.

Note first that none of the roots of $q_n( s )$ are roots of
$p_n( s )$, since if $i$, $j$ and $k$ are non-negative integers such
that $i + j + k = n$, then
\[
p_n( -i a_1 - j a_2 - k a_3 ) = %
\binom{n}{i, j, k} c_1^i c_2^j c_3^k %
\prod_{\substack{i' + j' + k' = n \\ ( i', j', k' ) \neq ( i, j, k )}}
\bigl( ( i' - i ) a_1 + ( j' - j ) a_2 + ( k' - k ) a_3 \bigr) \neq 0.
\]
We will now show that $p_n( s )$ is non constant, which establishes
that $1 / F_n( s )$ is as claimed. We have that
\begin{align*}
\gamma_n := \frac{p_n( -n a_3 )}{p_n( -n a_1 )} & = %
\Bigl( \frac{c_3}{c_1} \Bigr)^n \frac{n a_1 - n a_3}{n a_3 - n a_1} %
\prod_{\substack{i + j + k = n \\ i, k \neq n}} %
\frac{i a_1 + j a_2 + k a_3 - n a_3}{i a_1 + j a_2 + k a_3 - n a_1}
\\
& = -\Bigl( \frac{c_3}{c_1} \Bigr)^n %
\prod_{\substack{i + j + k = n \\ i, k \neq n}} %
\frac{( a_1 - a_3 ) i + ( a_2 - a_3 ) j}%
{( a_2 - a_1 ) j + ( a_3 - a_1 ) k} \\
& = ( -1 )^{n + 1} \Bigl( \frac{c_1 + c_2}{c_1} \Bigr)^n %
\prod_{\substack{i + j + k = n \\ i, k \neq n}} %
\frac{c_1 j - c_2 i}{(c_1 + c_2 ) j + c_2 i},
\end{align*}
where the last step uses the fact that the placeholder variables $i$
and $k$ are symmetric in the product for the denominator. Since
\[
\prod_{\substack{i + j + k = n \\ i, k \neq n}} %
\frac{c_1 j - c_2 i}{( c_1 + c_2 ) j + c_2 i} = %
( -1 )^{n - 1} \bigl( \frac{c_1}{c_1 + c_2} \bigr)^n %
\prod_{2 \leq i + j \leq n} %
\frac{c_1 j - c_2 i}{( c_1 + c_2 ) j + c_2 i},
\]
it follows that
\[
\frac{\gamma_{n + 1}}{\gamma_n} = %
\prod_{\substack{i + j = n + 1 \\ i, j \geq 1}} %
\frac{c_1 j - c_2 i}{( c_1 + c_2 ) j + c_2 i} = %
( -1 )^n \prod_{j = 1}^n %
\frac{n + 1 - ( \alpha + 1 ) j}{n + 1 + \alpha j},
\]
where
\[
\alpha = \frac{c_1}{c_2} = \frac{-( a_3 - a_2 )}{a_3 - a_1} \in %
( -1, 0 ).
\]
If $j = 1$, \ldots, $n$, then
\[
\frac{n + 1 - ( \alpha + 1 ) j}{n + 1 + \alpha j} < 1 \iff %
n + 1 - ( \alpha + 1 ) j < n + 1 + \alpha j \iff %
0 < 2 \alpha + 1,
\]
and
\[
0 < 2 \alpha + 1 = \frac{a_3 - a_1 - 2 a_3 + 2 a_2}{a_3 - a_1} \iff %
a_3 - a_2 < a_2 - a_1,
\]
whereas
\[
\frac{n + 1 - ( \alpha + 1 ) j}{n + 1 + \alpha j} > 1 \iff %
0 > 2 \alpha + 1 \iff a_3 - a_2 > a_2 - a_1.
\]
Since equality is impossible, by linear independence, it follows that
either $| \gamma_{n + 1} / \gamma_n | < 1$ for all $n \in \N$ or
$| \gamma_{n + 1} / \gamma_n | > 1$ for all $n \in \N$; in each
case, since $\gamma_1 = 1$, the polynomial~$p_n$ is non-constant for
all $n \geq 2$. This completes the proof.
\end{proof}

With the preliminary result established, we can now obtain the
promised charactizations.

\begin{proof}[Proof of Theorem~\ref{T1sidedPFF}]
Let $F : [ 0, \infty ) \to [ 0, \infty )$. We show first that, if
$G = F \circ t \lambda_d$ is a $\TN_2$ function for all $t > 0$ and
all $d \in ( 0, 1 )$, in that $\det T_G[ \bx; \by ] \geq 0$ for all
$\bx$, $\by \in \inc{\R}{2}$, then $F$ is non-decreasing on
$( 0, \infty )$. This follows because, if $t > 0$, $\eps > 0$,
and $d = e^{-2 \eps}$, then
\[
0 \leq %
\det( F \circ t e^\eps T_{\lambda_d} )%
[ ( \eps, 2 \eps); ( 0, \eps ) ] = %
\begin{vmatrix}
 F( t ) & F( t e^{-\eps} ) \\
 F( t e^{-\eps} ) & F( t )  
\end{vmatrix}.
\]
Next, letting
$t_1 := \sup \{ t > 0 : F( t ) = 0 \} \cup \{ 0 \} \in [ 0, \infty ]$,
we see that $F$ vanishes on $( 0, t_1 )$ and is positive on
$( t_1, \infty )$.

We now show that $F$ is continuous on $( t_1, \infty )$. If
$t_1 = \infty$, then this is immediate. Otherwise, we first show that
$F$ is multiplicatively mid-concave on $( t_1, \infty )$. Given
any $q > p > t_1$, choose $t > q$ and let $a$ and $b$ be such that
$q = t e^{-a}$ and $p = t e^{-b}$. Then, for any $d \in ( 0, 1 )$,
\[
0 \leq \det( F \circ t T_{\lambda_d} )%
[ ( ( b + a ) / 2, b) ; ( 0, ( b - a ) / 2 ) ] = %
F( \sqrt{p q} )^2 - F( p ) F( q ).
\]
The continuity of $F$ on $( t_1, \infty )$ now follows by
\cite[Theorem~71.C]{roberts-varberg}. We can now establish each of the
three assertations of the theorem.

(1). Every homothety preserves the class of one-sided P\'olya
frequency functions. Conversely, since $F \circ \lambda$ is a P\'olya
frequency function, it is integrable, so $F( 0 ) = 0$, and has
unbounded support, so $t_1 = 0$: if $t_1 > 0$ then $F( e^{-x} )$ will
vanish for all sufficiently large $x$. We now follow the relevant
parts of the proof of Theorem~\ref{TPolya}. By considering
$F \circ p \lambda$ for any $p > 0$, we obtain the existence of
positive constants $b_0$ and~$c_0$ such that $F( t ) = c_0 t^{b_0}$
for all $t \in ( 0, p )$; since $p$ is arbitrary, this identity must
hold everywhere on $( 0, \infty )$. Moreover, the exponent $b_0$ must
be an integer, by \cite[Theorem~1]{Schoenberg51} and the form of the
bilateral Laplace transform of $F \circ \phi$, where
$\phi( x ) = x \lambda( x )$. That this integer equals $1$ now follows
from Lemma~\ref{L1sided}.

(2). That the given functions are preservers follows from the
proof that $(3) \implies (1)$ in Theorem~\ref{Ttoeplitz}. Conversely,
let $t_1$ be as above and note that $F( t \bJ_{x > 0} ) \equiv 0$
whenever $t \in ( 0, t_1 )$. Hence $t_1 = 0$ and we may now follow the
proof that $(2) \implies (3)$ in Theorem~\ref{Ttoeplitz}, replacing
the use of Lemma~\ref{Lintpowers} with Lemma~\ref{L1sided}.

(3). As $\bJ_{x = 0}$ is now in our test set, the function
$F( x ) = c \bJ_{x = 0}$ is no longer a preserver of total
non-negativity; the other functions in the preceding part do preserve
$a \bJ_{x = b}$ for any $a \geq 0$ and $b \in \R$. For the converse,
if $t_1 = 0$, then the working for (2) and the previous observation
gives the result. Otherwise, if $t_1 > 0$ and $F( 0 ) \neq 0$, then
$F( t \bJ_{x = 0} ) = F( 0 ) ( 1 - \bJ_{x = 0} )$ is not $\TN$
whenever $t \in ( 0, t_1 )$. It remains to consider the case
$F( 0 ) = 0$ and $t_1 \in ( 0, \infty )$, but then
$F \circ 2 t_1 \lambda$ would be a $\TN$ function which is non-zero at
two or more points and has compact support, an impossibility. This
completes the proof.
\end{proof}

With Theorem~\ref{T1sidedPFF} to hand, we show a similar result for
the preservers of one-sided P\'olya frequency sequences, but in
slightly greater generality. See Definition~\ref{Dadmissible} for the
definition of an admissible pair; we say that a Toeplitz kernel
$T_f : X \times Y \to \R$ is \emph{one sided} if the associated
function $f : X - Y \to \R$ is, that is, there exists
$x - y \in X - Y$ such that $f$ vanishes on
$\{ z \in X - Y : z < x - y \}$ or on
$\{ z \in X - Y : z > x - y \}$.

\begin{theorem}\label{T1sidedPFseq}
Let $X$ and $Y$ be a pair of admissible sets such that $X - Y$ does
not have a maximum or minimum element. If
$F : [ 0, \infty ) \to [ 0, \infty)$ is non-zero and preserves the
one-sided $\TN$ Toeplitz kernels on $X \times Y$, then $F( x ) = c x$
or $F( x ) = c \bJ_{x > 0}$ for some $c > 0$.
\end{theorem}

\begin{proof}
We suppose that $F$ is non-zero and preserves one-sided $\TN$ Toeplitz
kernels on $X \times Y$ and proceed in a number of steps.

Step 1: $F( 0 ) = 0$ and $F$ is non-decreasing. Since $F \not\equiv 0$
on $[ 0, \infty )$, there exists $t > 0$ such that
$F( t ) \neq F( 0 )$. For any $x - y \in X - Y$, the one-sided
function
\[
\Lambda_0( z ) = t \bJ_{z = x - y}
\]
is a P\'olya frequency function, so the kernel
\[
( F \circ T_{\Lambda_0} )( z ) = %
F( 0 ) + ( F( t ) - F( 0 ) ) \bJ_{z = x - y}
\]
is one-sided. As $F$ equals the constant $F( 0 )$ on $X - Y$, except
at $x - y$, and $X - Y$ has no extremal element, it follows that
$F( 0 ) = 0$. In particular, the operator $C_F$ preserves one-sided
kernels, and henceforth we focus on preserving $\TN$.

We next show that $F$ is non-decreasing, by following the proof of
Theorem~\ref{T1sidedPFF}. From the previous working, we have that
$F( 0 ) = 0 \leq F( t )$ for all $t > 0$. Next, given any $t > 0$ and
$\eps > 0$, we see that $F( t ) \geq F( t e^{-\eps} )$, by
considering $\det( F \circ T_{\Lambda_1} )[ \bx; \by ]$,
where the arithmetic progressions $\bx = ( x_1, x_2 ) \in \inc{X}{2}$
and $\by = ( y_1, y_2 ) \in \inc{Y}{2}$ are equi-spaced, and
$\Lambda_1$ is the one-sided $\TN$ function
\[
\Lambda_1( z ) := t e^\eps %
\lambda_{e^{-2 \eps}}%
\Bigl( \eps \frac{z - x_1 + y_2}{x_2 - x_1} \Bigr).
\]
This shows the monotonicity of $F$, and so the existence of
$t_1 \in [ 0, \infty ]$ such that $F$ vanishes on $( 0, t_1 )$ and is
positive on $( t_1, \infty )$.

Step 2: $F$ is continuous on $( t_1, \infty )$. It suffices to show
that $F$ is multiplicatively mid-concave on $( t_1, \infty )$, by
\cite[Theorem~71.C]{roberts-varberg}. Given $p$,
$q \in ( t_1, \infty )$ with $p < q$, choose any $t > q$ and note that
\[
0 < a := \log( t / q ) < b := \log( t / p ).
\]
With equi-spaced arithmetic progressions $\bx \in \inc{X}{2}$ and
$\by \in \inc{Y}{2}$ as in the previous step, and any
$d \in [ 0, 1 ]$, we consider the one-sided $\TN$ function
\[
\Lambda_2( z ) := t \lambda_d\Bigl( a + %
\frac{z - x_1 + y_2}{2 ( x_2 - x_1 )} ( b - a ) \Bigr).
\]
The determinant
\[
\det( F \circ T_{ \Lambda_2} )[ \bx; \by ] = %
\begin{vmatrix}
 F( t e^{-(a + b) / 2} ) & F( t e^{-a} ) \\
 F( t e^{-b} ) & F( t e^{-( a + b ) / 2} )
\end{vmatrix} = %
F( \sqrt{p q} )^2 - F( p ) F( q )
\]
is non-negative, as required. Hence $F$ is continuous on
$( t_1, \infty )$.

Step 3: $t_1 = 0$. We suppose otherwise, and show first that
$F( t_1 ) = 0$. For this, it suffices to prove
that $F^+( t_1 ) := \lim_{t \to t_1^+} F( t ) = 0$. We fix equi-spaced
arithmetic progressions $\bx = ( x_1, x_2, x_3 ) \in \inc{X}{3}$ and
$\by = ( y_1, y_2, y_3 ) \in \inc{Y}{3}$. For arbitrary $\eps > 0$,
consider the one-sided $\TN$ function
\[
\Lambda_{3, \eps}( z ) := t_1 e^{5 \eps} %
\lambda_1\Bigl( 2 \eps \frac{z - x_1 + y_2}{x_2 - x_1} \Bigr).
\]
A straightforward computation, using the fact that $F( t ) = 0$ for
all $t < t_1$, gives that
\[
( F \circ T_{\Lambda_{3, \eps}} )[ \bx; \by ] = \begin{pmatrix}
 F( t_1 e^{3 \eps} ) & F( t_1 e^{5 \eps} ) & 0 \\
 F( t_1 e^\eps ) & F( t_1 e^{3 \eps} ) & F( t_1 e^{5 \eps} ) \\
 0 & F( t_1 e^\eps ) & F( t_1 e^{3 \eps} ) \end{pmatrix}.
\]
As this matrix is $\TN$ by assumption, we have that
\[
0 \leq \lim_{\eps \to 0^+} %
\det( F \circ T_{\Lambda_{3, \eps}} )[ \bx; \by ] = -F^+( t_1 )^3.
\]
Thus $F^+( t_1 ) = 0$.

We now fix $t_2 > t_1$. Given any $k \in \N$, we choose equi-spaced
arithmetic progressions
$\bx = ( x_1, \ldots, x_{k + 2} ) \in \inc{X}{k + 2}$ and
$\by = ( y_1, \ldots, y_{k + 2} ) \in \inc{Y}{k + 2}$, let
\[
\delta = \delta( k ) := \frac{1}{k} \log( t_2 / t_1 ) > 0,
\]
and consider the one-sided $\TN$ function
\[
\Lambda_{4, k}( z ) :=  t_1 e^{( 2 k + 1 ) \delta} %
\lambda_{e^{-\delta}}\Bigl( \delta %
\frac{z - x_1 + y_{k + 1}}{x_2 - x_1} \Bigr).
\]
As $F( 0 ) = 0 = F( t_1 )$, it follows that
\[
( F \circ T_{\Lambda_{4, k}} )%
[ ( x_1, x_2, x_{k + 2} ); ( y_1, y_{k + 1}, y_{k + 2} ) ] = %
\begin{pmatrix}
 F( t_2 e^\delta ) & F( t_2^2 / t_1 ) & 0 \\
 F( t_2 ) & F( t_2^2 / t_1 ) & F( t_2^2 / t_1 ) \\
 0 & F( t_2 ) & F( t_2 e^\delta ) \end{pmatrix}.
\]
Since this matrix is $\TN$, and $F$ is positive on $( t_1, \infty )$,
taking determinants gives that
\[
2 F( t_2 ) \leq F( t_2 e^\delta ).
\]
Letting $k \to \infty$ yields $0 < 2 F( t_2 ) \leq F( t_2 )$, as $F$
is continuous at $t_2$. Thus $F( t_2 ) \leq 0$, a contradiction since
$t_2 > t_1$.

Step 4: $F$ has the form claimed. The previous steps give that
$F$ is continuous, positive, and non-decreasing on $( 0, \infty )$.
We first assume that $F$ does not have the form $c \bJ_{x > 0}$ for
any $c > 0$, so that $F$ is non-constant on $( 0, \infty )$. If
$F \circ f$ is a $\TN$ function whenever $f$ has the form
$t \lambda_d$ for $t > 0$ and $d \in [ 0, 1 ]$,
$\phi( x ) = x \lambda( x )$ or $N$ as in Lemma~\ref{L1sided}, then
working as in the proof of Theorem~\ref{Ttoeplitz} shows that
$F( x ) = c x$ for some $c > 0$. Hence we assume otherwise: suppose
$G := F \circ f$ is not $\TN$ for one of these functions. Then, by
Lemma~\ref{Lhankel}, there exists $\bz \in \inc{\R}{n}$, where
$n \geq 2$, such that $A := T_G[ \bz; \bz ]$ is not $\TN$. Since $G$
is continuous except possibly at the origin, we may assume that
$\bz \in \inc{\Q}{n}$. Taking $N \in \N$ such that $N z_i$ is an
integer for all $i \in [n]$, we set $m = N ( z_n - z_1 ) + 1$, choose
equi-spaced arthimetic progressions $\bx \in \inc{X}{m}$ and
$\by \in \inc{Y}{m}$, and let
\[
\Lambda_5 : X \times Y \to \R; \ ( x, y ) \mapsto %
G\Bigl( \frac{( x - x_1 ) - ( y - y_1 )}{( x_1 - y_1 ) N} \Bigr).
\]
Then $\Lambda_5 = F \circ T_f |_{X \times Y}$ is $\TN$, since $T_f$ is
a one-sided $\TN$ Toeplitz kernel on $X \times Y$, but~$\Lambda_5$
contains $A$ as a principal submatrix. This contradiction completes
the proof.
\end{proof}

We conclude with the case of P\'olya frequency sequences, where
$X = Y = \Z$. Note that a shift of origin
$( a_n )_{n \in \Z} \mapsto ( a_{n + k} )_{n \in \Z}$ preserves the
$\TN$ property for any $k \in \Z$, as does the reflection
$( a_n )_{n \in \Z} \mapsto ( a_{-n} )_{n \in \Z}$, so we may
consider only one-sided sequences that vanish for negative indices. 

\begin{corollary}\label{C1sidedPFS}
Let $F : [ 0, \infty ) \to [ 0, \infty )$. Then $F$ preserves the
set of one-sided P\'olya frequency sequences if and only if
$F( x ) = c x$ for some $c \geq 0$.
\end{corollary}

\begin{proof}
One implication is immediate. For the converse, by
Theorem~\ref{T1sidedPFseq}, we have that $F( x ) = c x$ or
$F( x ) = c \bJ_{x > 0}$; to rule out the latter possibility, we show
that $F$ is continuous at the origin. Given the one-sided P\'olya
frequency sequence $\ba$ such that $a_0 = a_2 = 1$, $a_1 = 2$, and
$a_n = 0$ otherwise, we note that
\[
C_\eps := \det( F \circ \eps T_\ba )[ ( 0, 1, 2 ); ( -1, 0, 1 ) ] = %
\begin{vmatrix}
 F( 2 \eps ) & F( \eps ) & 0 \\
 F( \eps ) & F( 2 \eps ) & F( \eps ) \\
 0 & F( \eps ) & F( 2 \eps )
\end{vmatrix} \geq 0
\]
for any $\eps > 0$. Hence
\[
0 \leq \lim_{\eps \to 0^+} C_\eps = \lim_{\eps \to 0^+} -F( \eps )^3,
\]
which gives continuity at the origin as claimed.
\end{proof}

We conclude with a corollary on lower-triangular matrices, in the
spirit of the final observation in the previous section.

\begin{corollary}
Let $F : [ 0, \infty ) \to [ 0, \infty )$. Then $F$ preserves the set
of $\TN$ lower-triangular Toeplitz matrices if and only if
$F( x ) = c x$ for some $c \geq 0$.
\end{corollary}
\begin{proof}
One implication is immediate. For the converse, suppose $F( 0 ) = 0$
but $F$ is not of the form $F( x ) = c x$ for any $c \geq 0$. By
Corollary~\ref{C1sidedPFS}, there exists a one-sided P\'olya frequency
sequence $\ba = ( a_n )_{n \in \Z}$ such that $F \circ \ba$ is not a
one-sided P\'olya frequency sequence, so there exist
$\bx = ( x_1, \ldots, x_n ) \in \inc{\Z}{n}$ and
$\by = ( y_1, \ldots, y_n ) \in \inc{\Z}{n}$, where $n \in \N$, such
that $A := ( F \circ T_\ba )[ \bx; \by ]$ has negative determinant.
Let $m := \min\{ x_1, y_1 \}$, $M := \max\{ x_n, y_n \}$, and let
$\bz := ( m, \ldots, M )$. Then the Toeplitz matrices
$T_\ba[ \bz; \bz ]$ and $( F \circ T_\ba )[ \bz; \bz ]$ are lower
triangular, since $a_n = 0$ whenever $n < 0$, but the latter
contains~$A$ as a submatrix, so cannot be $\TN$. This contradiction
gives the result.
\end{proof}

The observation prior to Corollary~\ref{C1sidedPFS} means that the
previous result applies equally to upper-triangular matrices.


\section{Total-positivity preservers. II: General domains}\label{Sqns}

It is time to return to our original problem, of classifying the
preservers of $\TP$ kernels on $X \times Y$ for totally ordered sets
$X$ and $Y$. So far, we have resolved this when at least one of $X$
and $Y$ is finite. The picture is completed by the following result.

\begin{theorem}\label{Tmain1}
Suppose $X$ and $Y$ are infinite totally ordered sets such that there
exists a $\TP$ kernel on $X \times Y$. A function
$F : ( 0, \infty ) \to ( 0, \infty )$ preserves the set of $\TP$
kernels on $X \times Y$ if and only $F( x ) = c x$ for some $c > 0$.
\end{theorem}

\begin{proof}
Without loss of generality, we assume $X$ and
$Y$ are infinite subsets of $\R$, by Lemma~\ref{Ltptosets}. Next, we
establish a stronger version of the chain property used to prove
Proposition~\ref{Pdim}, that is, the existence of an order-preserving
bijection $\varphi_X : X \to X'$, where $X' \subseteq \R$ contains an
arithmetic progression of length $2^n$ and increment $4^{-n}$ for each
integer $n \geq 2$. To show this, let $( a_n )_{n \geq 1}$ be an
increasing sequence of positive real numbers, containing an arithmetic
progression of length $2^n$ and increment $4^{-n}$ for each integer
$n \geq 2$, and converging to~$1$. Such a sequence can be constructed,
by taking each arithmetic progression of the form $j 4^{-n}$, where
$j \in [2^n]$, and concatenating these, at each stage adding the last
term of the existing sequence to each term of the next progression:
\[
\frac{1}{4}, \quad \frac{2}{4}, \quad %
\frac{2}{4} + \frac{1}{16}, \quad \ldots, \quad %
\frac{2}{4} + \frac{4}{16}, \quad %
\frac{2}{4} + \frac{4}{16} + \frac{1}{64}, \quad \ldots, \quad %
\frac{2}{4} + \frac{4}{16} + \frac{8}{64}, \ldots.
\]
As observed in Section~\ref{Snon-neg}, the set $X$ has either an
infinite ascending chain or an infinite descending chain. Without loss
of generality, we assume the former; the argument for a descending
chain is similar. If $X$ is unbounded above, let
$( x_n )_{n \geq 1} \subseteq X$ be an increasing sequence such that
$x_n \to \infty$. For all $n \in \N$, let $\varphi_X( x_n ) = a_n$ and
extend $\varphi_X$ piecewise linearly on
$\{ x \in X : x_n < x < x_{n+1} \}$. This provides an order-preserving
embedding of $\{x \in X : x \geq x_1\}$ into $( 0, 1 )$ containing the
desired arithmetic progressions. If, instead, $X$ is bounded above,
with $\sup X = m$, let $( x_n )_{n \geq 1} \subseteq X$ be any
increasing sequence in $X$, and let $x \in \R$ be its limit. Set
$\varphi_X( x_n ) = a_n$, $\varphi_X( x ) = 1$, and
$\varphi_X( m ) = 2$ if $m \neq x$, and extend $\varphi_X$ piecewise
linearly between these points. Then $\varphi_X$ is again an
order-preserving embedding of $\{x \in X : x \geq x_1\}$ into
$( 0, 2 ]$ and containing the desired arithmetic progressions. A
similar argument can be used to extend $\varphi_X$ to the whole of~$X$
by mapping $\{x \in X : x < x_1 \}$ into $[ -2, 0 )$. This gives
$\varphi_X$ as claimed.

To complete the proof, let $\varphi_X : X \to X'$ and
$\varphi_Y : Y \to Y'$ be order-preserving bijections as constructed
above, and note that $X'$ and $Y'$ are an admissible pair in the sense
of Definition~\ref{Dadmissible}. Thus, if $F$ preserves $\TP$ for
kernels on $X \times Y$ but is not a dilation, Theorem~\ref{Tpfseq2}
gives a $\TP$ Toeplitz kernel $K'$ on $X' \times Y' \to \R$ such that
$F \circ K'$ is not $\TP$. However, the kernel
\[
K : X \times Y \to \R; \ ( x, y ) \mapsto %
K'( \varphi_X( x ), \varphi_Y( y ) )
\]
is $\TP$, since $K'$ is, and therefore so is $F \circ K$, by
the assumption on $F$. As any submatrix of $F \circ K'$ occurs as a
submatrix of $F \circ K$, we have a contradiction. As dilations
clearly preserve $\TP$, the proof is complete.
\end{proof}

For our final result, we consider the symmetric version of the
previous theorem.

\begin{theorem}\label{Ttpsymm}
Suppose $X$ is an infinite totally ordered set such that there exists a
symmetric $\TP$ kernel on $X \times X$. A function
$F : ( 0, \infty ) \to ( 0, \infty )$ preserves $\TP$ for symmetric
kernels on $X \times X$ if and only if $F( x ) = c x$ for some
$c > 0$.
\end{theorem}

\begin{proof}
In the usual fashion, we first identity $X$ with a subset of $\R$
using Lemma~\ref{Ltptosets}. Proposition~\ref{Ppowers} now
implies that any such preserver must be of the form
$F( x ) = c x^\alpha$ for positive constants $c$ and $\alpha$.

The construction in the proof of Theorem~\ref{Tmain1} gives an
order-preserving bijection $\varphi : X \to X'$, where
$X' \subseteq \R$ contains arithmetic progressions of arbitrary
length. This bijection provides a correspondence between $\TP$
symmetric kernels on $X \times X$ and those on $X' \times X'$, so we
may assume, without loss of generality, that $X$ contains arithmetic
progressions of arbitrary length.

Given $u_0 \in [ 0, 1 )$, $p$, $q \geq 0$ such that $p + q > 0$ and
$\eps > 0$, note first that the Hankel kernel
\[
K : \R \times \R \to \R; \ ( x, y ) \mapsto %
p + q u_0^{x + y} + %
\eps \int_0^1 e^{( x + y ) u} \intd u,
\]
is $\TP$, by Proposition~\ref{Phankel}. Let
$\bx = ( x_1, \ldots, x_n ) \in \inc{X}{n}$ be an arithmetic
progression, where $n \geq 2$, and note that the kernel
\[
K' : X \times X \to \R; \ ( x, y ) \mapsto %
K\left( \frac{x - x_1}{x_2 - x_1}, %
\frac{y - x_1}{x_2 - x_1} \right)
\]
is symmetric and $\TP$, hence so is
\[
( F \circ K' )[ \bx; \bx ] = %
( F \circ K )[ ( 0, 1, \ldots, n - 1 ); ( 0, 1, \ldots, n - 1 ) ].
\]
By the continuity of $F$, the matrix
\[
\lim_{\eps \to 0^+} %
( F \circ K )[ ( 0, 1, \ldots, n - 1 ); ( 0, 1, \ldots, n - 1 ) ] = %
( F( p + q u_0^{i + j} ) )_{i, j = 0}^{n - 1}
\]
is positive semidefinite. As $n$ can be arbitrarily large, it
follows from \cite[Theorem~4.1]{BGKP-hankel} that $F$ is the
restriction of an entire function $\sum_{k = 0}^\infty c_k x^k$ with
$c_k \geq 0$ for all $k$. Since $F( x ) = c x^\alpha$, it must be that
$\alpha$ is a positive integer.

To conclude, we suppose for contradiction that $\alpha \geq 2$ and let
$M$ be a P\'olya frequency function as in Lemma~\ref{Lintpowers}. Then
$M$ is non-vanishing and, by Lemma~\ref{Lhankel}, there exists
$\bz \in \inc{\R}{n}$ such that $( F \circ T_M )[ \bz; \bz ]$ has
a negative minor; by continuity, we may assume that
$\bz \in \inc{\Q}{n}$. Working as in the final two paragraphs of the
proof of Theorem~\ref{Tpfseq2}, with $X = Y$ and $\bx = \by$, now
gives the result.
\end{proof}

\section{Concluding remarks: Minimal test families}

In the first part of this section, we identify a few directions for
future exploration. The second part contains an enumeration of minimal
test criteria to demonstrate the rigidity of preservers.

\subsection{Open questions}

By specializing the families of maps that the post-composition
transform leaves invariant, we may obtain a plethora of classification
questions. Some of these questions, which appear artificial at first
sight, might gain weight due to future applications. Here, we simply
touch the surface.

\begin{question}
Which functions preserve the class of one-sided P\'olya frequency
sequences with finitely many non-zero terms, or those generated by
evaluating a polynomial?
\end{question}

This question has more positive answers than just the homotheties. For
example, the power maps $x^n$ are preservers of both sub-classes of
one-sided P\'olya frequency sequences for all $n \in \N$, by results
of Mal\'o~\cite{Malo} and Wagner~\cite{Wag92}, respectively.

\begin{question}
Are the totally positive P\'olya frequency sequences dense in the set
of all P\'olya frequency sequences?
\end{question}

\begin{question}
Given a $\TN$ kernel $K : X \times Y \to \R$, where $X$ and $Y$ are
infinite subsets of $\R$, can $K$ be approximated by a sequence of
$\TP$ kernels, at least at points of continuity?
\end{question}

For the latter question, recall that Section~\ref{SWhitneyExt}
contains such an approximation by $\TP_p$ kernels, for every
$p \in \N$.

While positive solutions to the preceding questions could help provide
alternate proofs of the classifications of the classes of $\TP$
preservers in these settings, we have already achieved these
classifications via different methods.

\subsection{Minimal testing families}

Many of the proofs above isolate some minimal classes of kernels
against which putative $\TN$ or $\TP$ preservers must be tested. For
the reader interested solely in the dimension-free setting of
Theorem~\ref{T1}, we end with some toolkit observations.

If a function $F : [ 0, \infty ) \to \R$  preserves $\TN$ for (a)~all
$\TN$ $2 \times 2$ matrices, (b)~the $3 \times 3$ matrix $C$ from
(\ref{E3matrix}), and (c)~the two-parameter family of $4 \times 4$
matrices $N( \eps, x)$ defined above (\ref{E4matrix}), then $F$ is
either constant or linear. Specifically, as the proof of
Theorem~\ref{Tfixeddim} shows, preserving $\TN$ for the $2 \times 2$
test set implies that $F$ is either a non-negative constant or
$F( x ) = c x^\alpha$ for some $c > 0$ and $\alpha \geq 0$. Using the
matrix $C$, we see that $\alpha \geq 1$. Finally, using the test set
$\{ N( \eps, x ) : \eps \in ( 0, 1 ), \ x > 0 \}$, we obtain
$\alpha = 1$.

As noted in Remark~\ref{Rtoeplitz}, a non-zero function
$F : [ 0, \infty ) \to [ 0, \infty )$ preserves P\'olya frequency
functions if and only if the transforms by $F$ of $t \lambda( x )$ and
$x \lambda( x )$ are $\TN$ for all $t > 0$, as are the transforms of
Gaussians $c G_1$ for all $c > 0$, as well as that of a single
function $M$ as in Lemma~\ref{Lintpowers}.

We also note that Theorem~\ref{T1} for $\TN$ preservers was proved by
different means, in the context of Hankel positivity preservers, in
\cite[Section~5]{BGKP-hankel}. By comparison, the proof given here has
clear benefits, including completing the classification in every fixed
size and isolating a small set of matrices on which the preservation
of the $\TN$ property can be tested. Our present approach also leads
to the classification of preservers of total positivity for matrices
of a prescribed size, as well as classifications of the preservers
when restricted to symmetric matrices.

\subsection*{List of symbols}

For the convenience of the reader, we list some of the symbols used in
this paper.

\begin{itemize}
\item $\N$ and $\N_0 = \N \cup \{ 0 \}$ denote the sets of positive
integers and non-negative integers, respectively.

\item For any $n \in \N$, $[n]$ denotes the set $\{ 1, \ldots, n \}$.

\item Given a totally ordered set $X$ and $n \in \N$, the set
$\inc{X}{n}$ comprises all increasing $n$-tuples
$( x_1, \ldots, x_n )$ in $X$, so that $x_1 < \cdots < x_n$.

\item Let $X$ and $Y$ be totally ordered sets and
suppose $K : X \times Y \to \R$. Given tuples
$\bx \in \inc{X}{m}$ and $\by \in \inc{Y}{n}$, $K[ \bx; \by ]$ denotes
the $m \times n$ matrix with $( i, j )$ entry $K( x_i, y_j )$.

\item For totally ordered sets $X$ and $Y$,
\begin{align*}
\sFTN_{X, Y} & := \{ F : [ 0, \infty ) \to \R \mid %
\text{if $K : X \times Y \to \R$ is totally non-negative, so is } %
F \circ K \} \\
\text{and} \\
\sFTP_{X, Y} & := \{ F : ( 0, \infty ) \to \R \mid \text{if } %
K : X \times Y \to \R \text{ is totally positive, so is } %
F \circ K \}.
\end{align*}

\item Given $p \in \N$, $\TN_p$ and $\TP_p$ are the sets of matrices
or kernels whose submatrices of order $d \times d$ have non-negative
or positive determinants, respectively, for all $d \in [p]$.

\item Given a set $X$ and a domain $I \subseteq \R$,
\begin{align*}
\sF_X^\psd( I ) & := \{ f : I \to \R \mid %
\text{if } K : X \times X \to I %
\text{ is positive semidefinite, so is } f \circ K \} \\
\text{and} \\
\sF_X^\pd( I ) & := \{ f : I \to \R \mid %
\text{if } K : X \times X \to I %
\text{ is positive definite, so is } f \circ K \}.
\end{align*}
\end{itemize}



\end{document}